\documentclass[11pt,reqno]{amsart}
\usepackage[margin=1in]{geometry}
\usepackage{amsmath,amssymb,amsthm,graphicx,amsxtra, setspace}
\usepackage[utf8]{inputenc}
\usepackage{mathrsfs}
\usepackage{hyperref}
\usepackage{upgreek}
\usepackage{mathtools}
\usepackage{hyperref}
\usepackage{alltt}
\usepackage{fouridx}
\usepackage{dsfont}
\usepackage{dsfont}
\usepackage{cancel}
\usepackage{xcolor}
\usepackage{multirow}
\usepackage[mathcal]{euscript}
\usepackage{shadethm}
\usepackage{float}
\usepackage{todonotes}
\usepackage{textcomp}
\usepackage{slashbox}
\allowdisplaybreaks

\usepackage[cyr]{aeguill}

\colorlet{darkblue}{blue!50!black}

\hypersetup{
	colorlinks,%
	citecolor=blue,%
	filecolor=red,%
	linkcolor=red,%
	urlcolor=blue,%
	pdfnewwindow=true,%
	pdfstartview={FitH}
}

\usepackage{todonotes}

\usepackage[pagewise]{lineno}

\newtheorem{theorem}{Theorem}[section]
\newtheorem{lemma}[theorem]{Lemma}
\newtheorem{proposition}[theorem]{Proposition}

\newtheorem{corollary}[theorem]{Corollary}
\newtheorem{definition}[theorem]{Definition}

\newtheorem{remark}[theorem]{Remark}

\newtheorem{hypothesis}[theorem]{Hypothesis}

\let\originalleft\left
\let\originalright\right
\renewcommand{\left}{\mathopen{}\mathclose\bgroup\originalleft}
\renewcommand{\right}{\aftergroup\egroup\originalright}

\newcommand{\Tr}{\mathop{\mathrm{Tr}}}
\renewcommand{\d}{\/\mathrm{d}\/}

\def\w{\textbf{W}^{\varepsilon}_{{\theta}^{\varepsilon}}}

\def\e{\varepsilon}

\def\t{t\wedge\tau_N^n}
\def\s{t\wedge\tau_N}

\def\T{T\wedge\tau_N^n}
\def\L{\mathbb{L}}
\def\A{\mathrm{A}}

\def\C{\mathrm{C}}
\def\f{\boldsymbol{f}}

\def\B{\mathrm{B}}
\def\D{\mathrm{D}}
\def\y{\boldsymbol{y}}

\def\E{\mathbb{E}}
\def\X{\mathbb{X}}

\def\g{\boldsymbol{g}}

\def\h{\boldsymbol{h}}
\def\z{\mathfrak{z}}
\def\u{\boldsymbol{X}}
\def\v{\boldsymbol{Y}}
\def\w{\boldsymbol{Z}}
\def\W{\mathrm{W}}

\def\Q{\mathrm{Q}}

\def\N{\mathbb{N}}

\def\V{\mathbb{V}}
\def\wi{\widetilde}
\def\Q{\mathrm{Q}}

\def\H{\mathbb{H}}
\def\n{\boldsymbol{n}}

\newcommand{\R}{\mathbb{R}}

\renewcommand{\d}{\/\mathrm{d}\/}

\newcommand{\Addresses}{{
		\footnote{
\noindent \textsuperscript{\textdagger}Center for Mathematics and Applications (NOVA Math), NOVA School of Science and Technology (NOVA FCT), Caparica,	Portugal.\par\nopagebreak
\noindent \textsuperscript{\textparagraph}Center for Mathematics and Applications (NOVA Math) and Department of Mathematics, NOVA School of Science and Technology (NOVA FCT),  Caparica, Portugal.\par\nopagebreak
			\noindent 
            \textsuperscript{\textdaggerdbl}Department of Mathematics, Indian Institute of Technology Roorkee-IIT Roorkee, Haridwar Highway, Roorkee, Uttarakhand 247667, INDIA.\par\nopagebreak
			\noindent  
            \textit{e-mail:} \texttt{kushkinra@gmail.com, k.kinra@fct.unl.pt, cipriano@fct.unl.pt, maniltmohan@ma.iitr.ac.in, maniltmohan@gmail.com.}

			\noindent \textsuperscript{*}Corresponding author.

			\textit{Key words:} Stochastic convective Brinkman-Forchheimer equations, martingale solution, invariant measures, ergodicity, general domains. 
			
			Mathematics Subject Classification (2020): Primary 60H15; Secondary 35R60, 35Q30, 76D05.

}}}

\begin{document}
	
	
	\title[Martingale solution, invariant measure and ergodicity for SCBFE\MakeLowercase{s}]{Martingale solution, invariant measure and ergodicity for stochastic convective Brinkman-Forchheimer equations on general domains in $\mathbb{R}^d$
		\Addresses}
	\author[Kinra, Cipriano and Mohan]{Kush Kinra$^{\text{\textdagger},*}$, Fernanda Cipriano$^{\text{\textparagraph}}$ and Manil T. Mohan$^{\text{\textdaggerdbl}}$}

	\maketitle

\begin{abstract}
The convective Brinkman-Forchheimer  equations (CBFEs)
\[
\frac{\partial \u}{\partial t}
 - \mu \Delta\u
 + (\u\cdot\nabla)\u
 + \alpha\u
 + \beta|\u|^{r-1}\u
 + \nabla p = \mathbf{F}, 
 \qquad \nabla\cdot\u=0,
\]
with parameters $\mu,\alpha,\beta>0$ and $r\in[1,\infty)$ describe incompressible fluid motion in saturated porous media. 
In the stochastic setting, for $d=2,3$ and $r\in[3,\infty)$ (with $2\beta\mu\geq 1$ when $r=3$), strong pathwise solutions on general domains are already known, hence weak martingale solutions exist as well. In the same parameter regime, invariant probability measures on bounded domains have also been obtained. The present work complements and significantly extends these results.
More precisely, on general domains in $\mathbb{R}^d$ (bounded or unbounded), for all $d\in\{2,3\}$, we prove the existence of a weak martingale solution to the stochastic CBFEs  for every exponent $r\in[1,\infty)$, which includes the regimes where no strong solution theory is available. In this more general framework,  the monotonicity technique does not apply, so we follow a different methodology, which relies on a Faedo-Galerkin approximation, tightness arguments, stochastic compactness via Jakubowski's generalization of the Skorokhod representation theorem for nonmetric spaces, and the martingale representation theorem. For $d=2$, $r\in[1,\infty)$, and for $d=3$, $r\in[3,\infty)$, we further show that the martingale solutions satisfy the energy equality (It\^o's formula) and possess $\mathbb{H}$-valued continuous trajectories almost surely. 
In this regularity regime (excluding $2\beta\mu < 1$ when $r=3$), we establish pathwise uniqueness and thereby, via the Yamada-Watanabe argument, obtain the existence of strong solutions and uniqueness in law, thereby recovering, in particular,  the known results. Finally, for $d=2$, $r\in[1,\infty)$, and for $d=3$, $r\in[3,\infty)$ (with $2\beta\mu\geq 1$ when $r=3$), we prove the existence of an invariant probability measure for the associated Markov semigroup, while for $d=2,3$ with $r\in[3,\infty)$ (and with $2\beta\mu\geq 1$ for $r=3$), we show that at most one invariant measure can exist. 
Our results broaden the stochastic well-posedness and ergodic theory of the stochastic CBFEs beyond previously known parameter ranges and beyond bounded spatial domains.
\end{abstract}

	\section{Introduction}\label{sec1}\setcounter{equation}{0}
\subsection{The underlying system}	The convective Brinkman-Forchheimer equations (CBFEs) in two- and three-dimensional smooth domains (bounded or unbounded) is considered in this work. The CBFEs describe the motion of incompressible fluid flows in a saturated porous medium. This model  is recognized to be more accurate when the flow velocity is too large for Darcy's law to be valid alone, and in addition, the porosity is not too small, so that  the term \emph{non-Darcy models} is used in the literature  for these types of fluid flow models (see \cite{Gautam+Kinra+Mohan_AMOP,PAM} for a discussion). Let {$\mathcal{O}\subseteq\R^d$ ($d=2,3$}) be a general domain with smooth boundary $\partial\mathcal{O}$. Let $\u(x,t) \in \R^d$ represent the velocity field at time $t$ and position $x$, $p(t,x)\in\R$ denote the pressure field, $\f(x,t)\in\R^d$ stand for an external forcing. The deterministic CBFEs or incompressible Navier-Stokes equations with damping are given by (see \cite{KT2} for Brinkman-Forchheimer equations with fast growing nonlinearities)
	\begin{equation}\label{1}
		\left\{
		\begin{aligned}
			\frac{\partial \u}{\partial t}-\mu \Delta\u+(\u\cdot\nabla)\u+\alpha\u+\beta|\u|^{r-1}\u+\nabla p&= \mathbf{F}, && \text{ in } \ \mathcal{O}\times(0,T), \\ \nabla\cdot\u&=0, && \text{ in } \ \mathcal{O}\times[0,T], \\
			\u&=\mathbf{0}, && \text{ on } \ \partial\mathcal{O}\times[0,T], \\
			\u(0)&=\u_0 && \text{ in } \ \mathcal{O}.
		\end{aligned}
		\right.
	\end{equation}
    To ensure the uniqueness of the pressure $p$, we may additionally impose the condition $	\int_{\mathcal{O}}p(x,t)\d x=0$ in  $[0,T]$.  
The constant $\mu$ denotes the positive Brinkman coefficient (effective viscosity), while the positive constants $\alpha$ and $\beta$ correspond to the Darcy coefficient (related to the permeability of the porous medium) and the Forchheimer coefficient (associated with the porosity of the material), respectively. When $\alpha = \beta = 0$, the system reduces to the classical $d$-dimensional Navier-Stokes equations (NSEs). Hence, system \eqref{1} may also be viewed as the NSEs with damping. The absorption exponent satisfies $r \in [1,\infty)$, where $r = 3$ is known as the critical exponent, and $r > 3$ corresponds to a fast-growing nonlinearity (cf. \cite{KT2}). It was shown in \cite[Proposition 1.1]{KWH} that the critical homogeneous CBFEs possess the same scaling as the NSEs only in the case $\alpha = 0$; for other values of $\alpha$ and $r$, the system does not exhibit scale invariance. The particular case $r = 3$ with $\alpha = 0$ is often referred to as the NSEs with an absorption term (\cite{SNA}) or the tamed NSEs (\cite{MRXZ}).
When $\mathcal{O} = \mathbb{R}^d$, the boundary condition $\u = \mathbf{0}$ on $\partial\mathcal{O} \times [0,T]$ is replaced by a decay condition at infinity, that is,
\begin{align*}
    |\u(x,t)| \to 0 \quad \text{as} \quad |x| \to \infty, \quad \text{for all } 0 \le t \le T.
\end{align*}
Global well-posedness results for the deterministic CBFEs \eqref{1} can be found in \cite{SNA,CLF,KT2,Gautam+Mohan_2025}, among other works and the references therein.

	\subsection{Literature} The global solvability of the stochastic counterpart of the problem \eqref{1} and related models (forced by Gaussian noise) in the whole space or on a torus is available in the works \cite{HBAM,ZBGD,WLMR,MRXZ1}, etc. By using classical Faedo-Galerkin approximations and compactness method, the existence of martingale solutions for  stochastic 3D NSEs with nonlinear damping subjected to bounded domains and multiplicative Gaussian noise is obtained in \cite{LHGH1}. For a sample literature on the weak martingale solution for 2D and 3D NSEs and related models perturbed by Gaussian noise, the interested readers are referred, for instance, to \cite{ZBEM,GDAN,YTFC}.

For $d=2,3$ with $r\in[3,\infty)$ ($2\beta\mu\geq 1$ for $r=3$),  the authors in \cite{KK+MTM-SCBF,MTM6} established the global existence and uniqueness of pathwise strong solutions satisfying the energy equality (It\^o's formula) for 2D and 3D stochastic convective Brinkman-Forchheimer equations (SCBFEs)  subjected to multiplicative Gaussian and pure jump noise, respectively  by exploiting a monotonicity property of the linear and nonlinear operators as well as a stochastic generalization of  the Minty-Browder technique.  The existence and uniqueness of local and global  pathwise mild solutions for SCBFEs perturbed by additive L\'evy noise in $\R^d$ ($d\in \{2,3\}$) by exploiting the contraction mapping principle is established in \cite{MTM9}. For additive rough Gaussian noise (taking values in Lebesgue spaces), analogous results on unbounded domains were established in \cite{KKMTM-DCDSB}. Moreover, \cite{KKMTM-DCDSB} also proves the existence of random attractors and invariant measures. Wentzell-Freidlin type large deviation principles for the 2D and 3D SCBFEs driven by multiplicative Gaussian noise and pure jump noise are obtained in \cite{MTM7} and \cite{MTM8}, respectively. In addition to the  results available in \cite{KK+MTM-SCBF,MTM6}, we further show in this paper that the martingale solution satisfying the energy equality  exists for $d=r=3$ with any $\mu,\beta>0$. 
	
\subsection{Main results}	In this work, we consider the stochastic convective Brinkman-Forchheimer equations (SCBFEs) perturbed by Gaussian noise consisting of the $\Q$-Wiener process. The SCBFEs driven by multiplicative Gaussian noise are given by 
		\begin{equation}\label{31}
			\left\{
			\begin{aligned}
				\d\u -\mu \Delta\u+(\u\cdot\nabla)\u + \alpha \u + \beta|\u|^{r-1}\u+\nabla p &=\mathbf{F} + \boldsymbol{\sigma}(t,\u)\d\W, && \text{ in } \ \mathcal{O}\times(0,T), \\ \nabla\cdot\u&=0, \ &&\text{ in }  \mathcal{O}\times[0,T], \\
				\u&=\mathbf{0},&& \text{ on } \ \partial\mathcal{O}\times[0,T], \\
				\u(0)&=\u_0, && \text{ in } \ \mathcal{O},
			\end{aligned}
			\right.
		\end{equation} 
		where $\W$ is a $\Q$-Wiener process  and $\boldsymbol{\sigma}(\cdot,\cdot)$ is noise coefficient. This work has three primary objectives:
	\begin{enumerate}
    \item [1.] For all the cases given in Table \ref{Table-1}, we first show the existence of a weak martingale solution $((\bar{\Omega},\bar{\mathscr{F}},\{\bar{\mathscr{F}}_t\}_{t\geq 0},\bar{\mathbb{P}}),\bar{\u},\bar{\W})$  to SCBFEs, where $(\bar{\Omega},\bar{\mathscr{F}},\{\bar{\mathscr{F}}_t\}_{t\geq 0},\bar{\mathbb{P}})$ is a filtered probability space, $\bar{\W}$ is a $\Q$-Wiener process, and $\bar{\u}=\{\bar{\u}(t)\}_{t\in[0,T]}$ is a stochastic process with trajectories in $\C([0,T];\H_w)\cap\mathrm{L}^2(0,T;\V)\cap\mathrm{L}^{r+1}(0,T;\wi\L^{r+1})$, $\bar{\mathbb{P}}$-a.s., satisfying an appropriate integral inequality. We use the classical Faedo-Galerkin approximation, a stochastic compactness method, Jakubowski's version of the Skorokhod theorem for nonmetric spaces and martingale representation theorem to obtain this result. 
            	\begin{table}[ht]
{	\begin{tabular}{|c|c|c|c|c|}
		\hline
		\textbf{Cases}& $d$ &$ r$& conditions on $\mu>0$, $\alpha>0$ \& $\beta>0$ \\
		\hline
		\textbf{I}& $d=2$ &$r\in[1,\infty)$&  for any   $\mu$, $\alpha$ and $\beta$  \\
		\hline
		\textbf{II}& $d=3$ &$r\in[1,\infty)$& for any $\mu$, $\alpha$ and $\beta$ \\
		\hline
	\end{tabular}
\vskip 0.1 cm
\caption{Values of $\mu, \alpha, \beta$ and $r$ for the existence of a weak martingale solution}
\label{Table-1}}
		\end{table}

	\item [2.] For all the cases given in Table \ref{Table-2}, we prove that the martingale solution  satisfies the energy equality (It\^o's formula) and hence the trajectories are continuous $\H$-valued function defined on $[0, T]$, $\bar{\mathbb{P}}$-a.s. 
      \begin{table}[ht]
{	\begin{tabular}{|c|c|c|c|c|}
		\hline
		\textbf{Cases}& $d$ &$ r$& conditions on $\mu>0$, $\alpha>0$ \& $\beta>0$ \\
		\hline
		\textbf{I}& $d=2$ &$r\in[1,\infty)$&  for any   $\mu$, $\alpha$ and $\beta$  \\
		\hline
		\textbf{II}& $d=3$ &$r\in[3,\infty)$& for any $\mu$, $\alpha$ and $\beta$ \\
		\hline
	\end{tabular}
\vskip 0.1 cm
\caption{Values of $\mu, \alpha, \beta$ and $r$ for the energy equality (It\^o's formula)}
\label{Table-2}}
		\end{table}\\
    Furthermore, for all the cases given in Table \ref{Table-3}, we establish the pathwise uniqueness of weak martingale solutions and use the classical Yamada-Watanabe argument to derive the existence of a strong solution and hence the uniqueness in law. 
        	\begin{table}[ht]
{	\begin{tabular}{|c|c|c|c|c|}
		\hline
		\textbf{Cases}& $d$ &$ r$& conditions on $\mu>0$, $\alpha>0$ \& $\beta>0$ \\
		\hline
		\textbf{I}& $d=2$ &$r\in[1,\infty)$&  for any   $\mu$, $\alpha$ and $\beta$  \\
		\hline
		\textbf{II}& $d=3$ &$r\in(3,\infty)$& for any $\mu$, $\alpha$ and $\beta$ \\
		\hline
		\textbf{III}& $d=3$ &$r=3$&for  $2\beta\mu\geq1$ \\
		\hline
	\end{tabular}
\vskip 0.1 cm
\caption{Values of $\mu, \alpha, \beta$ and $r$ for the pathwise uniqueness and strong solutions}
\label{Table-3}}
		\end{table}

        \item[3.]  	For all the cases given in Table \ref{Table-3alpha}, we show that there exists an invariant measure associated with the stochastic system \eqref{31}. 
        \begin{table}[ht]
{	\begin{tabular}{|c|c|c|c|c|}
		\hline
		\textbf{Cases}& $d$ &$ r$& conditions on $\mu>0$, $\alpha>0$ \& $\beta>0$ \\
		\hline
		\textbf{I}& $d=2$ &$r\in[1,\infty)$&   $2\alpha>L$  \\
		\hline
		\textbf{II}& $d=3$ &$r\in(3,\infty)$&  $2\alpha>L$ \\
		\hline
		\textbf{III}& $d=3$ &$r=3$&  $2\beta\mu\geq1$ and $2\alpha>L$ \\
		\hline
	\end{tabular}
\vskip 0.1 cm
\caption{Values of $\mu, \alpha, \beta$ and $r$ for the existence of invariant measures. Here $L\geq 0$ is the constant appearing in Hypothesis \ref{hyp}.}
\label{Table-3alpha}}
		\end{table}

        In addition, for all the cases given in Table \ref{Table-4}, we show that there exists at most one invariant measure associated with the stochastic system \eqref{31}.
        \begin{table}[ht]
{	\begin{tabular}{|c|c|c|c|c|}
		\hline
		\textbf{Cases}& $d$ &$ r$& conditions on $\mu>0$, $\alpha>0$ \& $\beta>0$  \\
		\hline
		\textbf{I}& $d=2,3$ &$r\in(3,\infty)$& either $2\alpha> 2\hat{\zeta}+L$ or $2\beta\mu\geq1$ and $2\alpha> \frac{1}{2\mu} + L$  \\
		\hline
		\textbf{II}& $d=2,3$ &$r=3$&  $2\beta\mu\geq1$ and $2\alpha>L$ \\
        \hline
	\end{tabular}
\vskip 0.1 cm
\caption{Values of $\mu, \alpha, \beta$ and $r$ for the uniqueness of invariant measures. Here $\hat{\zeta}=\frac{r-3}{2\mu(r-1)}\left(\frac{4}{\beta\mu (r-1)}\right)^{\frac{2}{r-3}}$.}
\label{Table-4}}
		\end{table} 
	\end{enumerate}
	We mainly follow the work \cite{ZBEM} to obtain the existence of a weak martingale solution and \cite{ZBEMMO} to establish the existence of an invariant probability measure. It is worth mentioning here that the current findings complement the findings of the works \cite{KKMTM-DCDSB} and \cite{KK+MTM-SCBF}. 
    
 \subsection{Comparison with existing works}   In the stochastic setting, the CBFEs  have been previously investigated mainly for $d=2,3$ with $r\in[3,\infty)$ (with $2\beta\mu \ge 1$ when $r=3$), where the existence of  strong solutions on general domains $\mathcal{O}\subseteq\R^d$ was established using the classical monotonicity method (cf. \cite{KK+MTM-SCBF}), see also the work \cite{LHGH1} which incorporated the existence of martingale and strong solutions on bounded domains only. This technique relies crucially on the monotone structure of the nonlinear damping term. Moreover, within this parameter regime, the existence of invariant probability measures has been established only for bounded domains, relying on the compactness of the Sobolev embedding.  The present work extends these results in several significant ways. First, we establish the existence of weak martingale solutions for all 
$r\in[1,\infty)$ on general domains, both bounded and unbounded, thereby covering the entire range of absorption exponents, including those regimes where strong solution theory is not yet available. Second, in the cases where strong solutions are known to exist, we refine the analysis by proving that martingale solutions satisfy the energy equality and enjoy enhanced regularity. This, in turn, enables us to prove pathwise uniqueness and invoke the Yamada-Watanabe theorem to obtain the existence of strong solutions and uniqueness in law on general domains. Finally, we broaden the ergodic theory of the system by proving the existence of invariant probability measures for a much wider class of domains and exponents, and we show the uniqueness of the invariant measure in the regime $r\geq 3$ (with $2\beta\mu \ge 1$ when $r=3$), thereby strengthening and generalizing the previously known results. Let us summarize the similarities and differences between the current work and the articles \cite{KKMTM-DCDSB} and \cite{KK+MTM-SCBF} in Table \ref{Table-5}.
	\begin{table}[ht]
		{	\begin{tabular}{|c|c|c|c|}
				\hline
				\backslashbox{\textbf{Results$\downarrow$}}{\textbf{Article}$\to$}& \cite{KKMTM-DCDSB}  & 	\cite{KK+MTM-SCBF}& Current work\\
				\hline
                	Additive  Gaussian noise& $\checkmark$ & $\checkmark$ & $\checkmark$\\
				\hline
				General multiplicative Gaussian  noise& -- &$\checkmark$ & $\checkmark$ \\
				\hline
				A weak martingale solution for the cases in Table \ref{Table-1}& --  & -- & $\checkmark$ \\
				\hline
				It\^o's formula for all the cases in Table \ref{Table-2}& -- &  -- &$\checkmark$  \\
                \hline
				It\^o's formula for all the cases in Table \ref{Table-3}& $\checkmark$ &  -- &$\checkmark$  \\
				\hline
				A unique strong solution for all the cases in Table \ref{Table-3} & $\checkmark$ & --  &$\checkmark$\\
				\hline
				Invariant measures on unbounded domains  & $\checkmark$& -- & $\checkmark$ \\
				\hline
			\end{tabular}
		\vskip 0.1 cm 
			\caption{Similarities and differences between the current work and previous works.}
			\label{Table-5}}
	\end{table}
\begin{remark}
    It is noteworthy that for NSEs, the authors of  \cite[Theorem~4.8 and Lemma~A.1]{ZBEMMO} assume $\mathbf{F}\in \mathrm{L}^{p}(0,T;\mathbb{V}')$ in order to derive  $p^{\mathrm{th}}$ moment bounds for all $p\geq 2$, in analogy with Proposition~\ref{prop1}, see also \cite[Proposition 2.3]{SSSP}. 
However, the weaker condition $\mathbf{F}\in \mathrm{L}^{2}(0,T;\mathbb{V}')$ already suffices to obtain 
$p^{\mathrm{th}}$ moment estimates for every $p\geq 2$, see the proof of Proposition \ref{prop1} below.

\end{remark}

\begin{remark}
   It is important to emphasize that, in order to apply the abstract result of \cite{Maslowski+Seidler_1999} for establishing the existence of an invariant measure, one must first verify that the semigroup $\{\mathrm{T}_t\}_{t\geq0}$ possesses the $bw$-Feller property. This means that for every bounded sequentially weakly continuous function $\varphi : \mathbb{H}\to\mathbb{R}$ and every $t>0$, the mapping $\mathrm{T}_t\varphi : \mathbb{H}\to\mathbb{R}$ remains bounded and sequentially weakly continuous. In particular, for $t>0$, if $\mathbf{u}_{0,m}\rightharpoonup \mathbf{u}_0$ in $\mathbb{H}$, then 
\begin{align*}
    \mathrm{T}_t\varphi(\mathbf{u}_{0,m}) \to \mathrm{T}_t\varphi(\mathbf{u}_0),
\end{align*}
see Definition~\ref{def-bw-feller} below.  
In \cite{KK+MTM-SCBF}, the authors employ the monotonicity method to obtain the existence of strong solutions to the stochastic system \eqref{1}. However, this approach depends on strong convergence of the initial data and therefore does not readily extend to establishing the $bw$-Feller property of $\{\mathrm{T}_t\}_{t\geq0}$.  
In contrast, in the present work, we construct strong solutions to \eqref{1} via the martingale solution framework. This method enables us to verify the $bw$-Feller property of the semigroup $\{\mathrm{T}_t\}_{t\geq0}$; see Proposition~\ref{prop-bw-feller} below.
\end{remark}

\subsection{Organization of the article} The organization of the paper is as follows. In Section \ref{sec2}, we define the linear and nonlinear operators, a compact operator, and provide the necessary function spaces needed to obtain the main results of this article. In Section \ref{sec3}, we first provide an abstract formulation of the SCBFEs \eqref{31}. We also present one of our main results in this section, the existence of a weak martingale solution, as well as the existence and uniqueness of strong solutions (Theorems \ref{thm3.4} and \ref{thm3.5}).  Furthermore, we provide the necessary functional tools  like deterministic compactness criterion, the Aldous condition,  Jakubowski's version of the Skorokhod theorem, etc. in the same section. One of the main results on the existence of weak martingale solution $((\bar{\Omega},\bar{\mathscr{F}},\{\bar{\mathscr{F}}_t\}_{t\geq 0},\bar{\mathbb{P}}),\bar{\u},\bar{\W})$ of the stochastic system \eqref{31} for all the cases given in Table \ref{Table-1} (Theorem \ref{thm3.4}) is established in Section \ref{sec4} using the classical Faedo-Galerkin approximation, a compactness method,  Jakubowski's version of the Skorokhod theorem for nonmetric spaces and the martingale representation theorem. Section \ref{sec5} is devoted for establishing the regularity properties of weak martingale solutions and the existence and uniqueness of strong solutions. For all the cases given in Table \ref{Table-2}, we address the energy equality (It\^o's formula) satisfied by the martingale solutions (Propositions \ref{prop5.1} and \ref{prop5.3}). Moreover, for all the cases given in Table \ref{Table-3}, we prove  the pathwise uniqueness in  Proposition \ref{prop5.4} and derive  the existence of a strong solution and uniqueness in law using  the classical Yamada-Watanabe argument. 
In final section, we demonstrate the existence of an invariant measure associated with underlying system for all the cases given in Table \ref{Table-3alpha} using the abstract result developed in \cite{Maslowski+Seidler_1999}. Also, we show the uniqueness of invariant measures for all the cases given in Table \ref{Table-4}.
	
	\section{Mathematical Formulation}\label{sec2}\setcounter{equation}{0}
	This section is devoted to the introduction of the necessary operators, function spaces, and auxiliary results required to establish the main results of the article.
\subsection{Function spaces}\label{Function-spaces} 
Let $\C_0^{\infty}(\mathcal{O};\R^d)$ denote the space of all infinitely differentiable, $\R^d$-valued functions with compact support in $\mathcal{O}\subseteq\R^d$. We define
\begin{align*}
	\mathcal{V} &:= \{\u \in \C_0^{\infty}(\mathcal{O};\R^d) : \nabla \cdot \u = 0\},\\
	\mathbb{H} &:= \text{closure of } \mathcal{V} \text{ in } \L^2(\mathcal{O}) = \mathrm{L}^2(\mathcal{O};\R^d),\\
	\mathbb{V} &:= \text{closure of } \mathcal{V} \text{ in } \H^1(\mathcal{O}) = \mathrm{H}^1(\mathcal{O};\R^d),\\
	\widetilde{\L}^{p} &:= \text{closure of } \mathcal{V} \text{ in } \L^p(\mathcal{O}) = \mathrm{L}^p(\mathcal{O};\R^d), \quad p \in (1,\infty), \quad p\neq 2,\\
	\mathbb{V}_s &:= \text{closure of } \mathcal{V} \text{ in } \H^s(\mathcal{O}) = \mathrm{H}^s(\mathcal{O};\R^d), \quad s > 1.
\end{align*}

Under suitable smoothness assumptions on the boundary $\partial\mathcal{O}$, the spaces $\H$, $\V$, and $\widetilde{\L}^p$ can be characterized as follows:
\begin{align*}
    \H = \left\{ \u \in \L^2(\mathcal{O}) : \nabla \cdot \u = 0, \; \u \cdot \n\big|_{\partial\mathcal{O}} = 0 \right\},
\quad 
\|\u\|_{\H}^2 := \int_{\mathcal{O}} |\u(y)|^2 \, \mathrm{d}y,
\end{align*}
where $\n$ denotes the outward unit normal vector on $\partial\mathcal{O}$ and $\u \cdot \n\big|_{\partial\mathcal{O}}$ is understood in the sense of trace (\cite[Lemma 1.3, Chapter 1]{Te})
\begin{align*}
    \V = \left\{ \u \in \H_0^1(\mathcal{O}) : \nabla \cdot \u = 0 \right\},
\quad
\|\u\|_{\V}^2 := \int_{\mathcal{O}} |\u(y)|^2 \, \mathrm{d}y + \int_{\mathcal{O}} |\nabla \u(y)|^2 \, \mathrm{d}y,
\end{align*}
and
\begin{align*}
    \widetilde{\L}^p = \left\{ \u \in \L^p(\mathcal{O}) : \nabla \cdot \u = 0, \; \u \cdot \n\big|_{\partial\mathcal{O}} = 0 \right\},
\quad
\|\u\|_{\widetilde{\L}^p}^p := \int_{\mathcal{O}} |\u(y)|^p \, \mathrm{d}y.
\end{align*}

Let $(\cdot,\cdot)$ denote the inner product in the Hilbert space $\H$, and let $\langle \cdot, \cdot \rangle$ represent the duality pairing between $\V$ and its dual $\V'$, as well as between $\widetilde{\L}^p$ and its dual $\widetilde{\L}^{p'}$, where $\frac{1}{p} + \frac{1}{p'} = 1$.  
For any general Hilbert space $\mathrm{H}$, we denote its inner product by $(\cdot,\cdot)_{\mathrm{H}}$.  
Note that the space $\H$ can be identified with its dual $\H'$.

Following \cite[Subsection~2.1]{RFHK}, the sum space $\V' + \widetilde{\L}^{p'}$ is well defined and forms a Banach space equipped with the norm
\begin{align}\label{22}
	\|\u\|_{\V' + \widetilde{\L}^{p'}} 
	&:= \inf \left\{ \|\u_1\|_{\V'} + \|\u_2\|_{\widetilde{\L}^{p'}} : 
	\u = \u_1 + \u_2, \; \u_1 \in \V', \; \u_2 \in \widetilde{\L}^{p'} \right\} \nonumber\\
	&= \sup \left\{ 
	\frac{|\langle \u, \g \rangle|}{
	\|\g\|_{\widetilde{\L}^p \cap \V}} : 
	\boldsymbol{0} \neq \g \in \widetilde{\L}^p \cap \V \right\},
\end{align}
where
\[
\|\cdot\|_{\widetilde{\L}^p \cap \V} := 
\max\left\{ \|\cdot\|_{\widetilde{\L}^p}, \, \|\cdot\|_{\V} \right\}
\]
defines a norm on the Banach space $\widetilde{\L}^p \cap \V$.  
Moreover, the norm $\max\{\|\u\|_{\widetilde{\L}^p}, \|\u\|_{\V}\}$ is equivalent to both 
$\|\u\|_{\widetilde{\L}^p} + \|\u\|_{\V}$ and 
$\sqrt{\|\u\|_{\widetilde{\L}^p}^2 + \|\u\|_{\V}^2}$ on $\widetilde{\L}^p \cap \V$.

Finally, the following continuous embeddings hold:
\[
\widetilde{\L}^p \cap \V \hookrightarrow \V \hookrightarrow \H \cong \H' 
\hookrightarrow \V' \hookrightarrow \widetilde{\L}^{p'} + \V'.
\]
For a detailed functional framework, we refer the reader to \cite{KWH,Te1}, among others.

\subsection{The Helmholtz-Hodge projection}
It is well known (see \cite{Farwig+Kozono+Sohr_2007,DFHM,HKTY}) that any vector field $\u \in\L^2(\mathcal{O}) \cap \L^p(\mathcal{O})$, with $2 \leq  p < \infty$, admits a unique decomposition $\u = \v + \nabla q$, where $\v \in \L^2(\mathcal{O})\cap \L^p(\mathcal{O})$ is divergence-free in the distributional sense and satisfies $\v \cdot \n = 0$ on $\partial\mathcal{O}$ (in the sense of trace), while $\nabla q \in \mathbb{L}^{p}(\mathcal{O})\cap \mathbb{L}^{2}(\mathcal{O})$ with $q \in \mathbb{L}_{\mathrm{loc}}^{p}(\mathcal{O})\cap \mathbb{L}_{\mathrm{loc}}^{2}(\mathcal{O})$. This representation is known as the \emph{Helmholtz-Hodge} (or \emph{Helmholtz-Weyl}) decomposition.  
For smooth vector fields, this decomposition is orthogonal in $\L^2(\mathcal{O})$. Using this decomposition, we define the projection operator
\begin{align*}
    \mathcal{P}_p : \L^2(\mathcal{O})\cap \L^p(\mathcal{O}) \to \H\cap\widetilde{\L}^p, 
\quad 
\mathcal{P}_p \u := \v.
\end{align*}
The operator $\mathcal{P}_p$ is a bounded linear projection satisfying $\mathcal{P}_p^2 = \mathcal{P}_p$.  
In particular, for $p = 2$, we denote $\mathcal{P} := \mathcal{P}_2 : \L^2(\mathcal{O}) \to \H$, which is the \emph{orthogonal Helmholtz-Hodge projection}.  Moreover, it has also been established (see \cite{Farwig+Kozono+Sohr_2007}) that, for $2\leq p<\infty$, the adjoint map  $(\mathcal{P}_p)^{*}=\mathcal{P}_{p'}$, where $\frac{1}{p}+\frac{1}{p'}=1$ and 
\begin{align*}
    \mathcal{P}_{p'} : \L^2(\mathcal{O})+ \L^{p'}(\mathcal{O}) \to \H+\widetilde{\L}^{p'}.
\end{align*}

\begin{remark}
    For, $1\leq r <\infty$, the projection $\mathcal{P}_{\frac{r+1}{r}}$ has a continuous linear extension (still denoted by the same) ${\mathcal{P}}_{\frac{r+1}{r}}: \H^{-1} + \L^{\frac{r+1}{r}} \to \V' + \wi\L^{\frac{r+1}{r}}$, see e.g. \cite[Proposition 3.1]{Kunstmann_2010}.
\end{remark}

\subsection{Linear operator}\label{opeA}
Let us  introduce  the linear operator defined by 
\begin{equation*}
	\A\u:=-\mathcal{P}_{\frac{r+1}{r}}\Delta\u,\ \u\in\V\cap\wi\L^{r+1}.
\end{equation*}
Remember that the operator $\A$ is a non-negative operator in $\H$ and \begin{align}\label{2.7a}
	\left<\A\u,\u\right>=\|\nabla\u\|_{\H}^2,\ \textrm{ for all }\ \u\in\V\cap\wi\L^{r+1}, \ \text{ so that }\ \|\A\u\|_{\V^{\prime}+\wi\L^{\frac{r+1}{r}}} \leq \|\u\|_{\V}.
\end{align}

\subsection{Bilinear operator}
	Let us define the \emph{trilinear form} $b(\cdot,\cdot,\cdot):\V\cap\wi\L^{r+1}\times\V\cap\wi\L^{r+1}\times\V\cap\wi\L^{r+1}\to\R$ by $$b(\u,\v,\w)=\int_{\mathcal{O}}(\u(x)\cdot\nabla)\v(x)\cdot\w(x)\d x=\sum_{i,j=1}^d\int_{\mathcal{O}}\u_i(x)\frac{\partial \v_j(x)}{\partial x_i}\w_j(x)\d x.$$ If $\u, \v$ are such that the linear map $b(\u, \v, \cdot) $ is continuous on $\V\cap\wi\L^{r+1}$, the corresponding element of $\V'+\wi\L^{\frac{r+1}{r}}$ is denoted by $\B(\u, \v)$. We also denote  $\B(\u) :=  \B(\u, \u)= \mathcal{P}_{\frac{r+1}{r}} [(\u\cdot\nabla)\u]$.
	An integration by parts yields  
	\begin{equation}\label{b0}
		\left\{
		\begin{aligned}
			b(\u,\v,\v) &= 0,\ \text{ for all }\ \u,\v \in\V,\\
			b(\u,\v,\w) &=  -b(\u,\w,\v),\ \text{ for all }\ \u,\v,\w\in \V.
		\end{aligned}
		\right.\end{equation}
	For $r\in[1,3]$,  using H\"older's inequality, we have 
	$
	\left|\langle \B(\u,\u),\v\rangle \right|=\left|b(\u,\v,\u)\right|\leq\|\u\|_{\wi\L^4}^2\|\v\|_{\V},
	$ for all $\v\in\V$ so that $$\|\B(\u)\|_{\V'}\leq\|\u\|_{\wi\L^4}^2, \ \text{ for all }\ \u\in\wi\L^4, $$ and we conclude that $\B(\cdot):\V\cap\widetilde{\L}^{4}\to\V'+\widetilde{\L}^{\frac{4}{3}}$. Furthermore, we have 
		\begin{align*}
			\|\B(\u)-\B(\v)\|_{\V'}&\leq \left(\|\u\|_{\widetilde{\L}^{4}}+\|\v\|_{\widetilde{\L}^{4}}\right)\|\u-\v\|_{\widetilde{\L}^{4}},
		\end{align*}
		hence $\B(\cdot):\V\cap\widetilde{\L}^{4}\to\V'+\widetilde{\L}^{\frac{4}{3}}$ is a locally Lipschitz operator. 
	An application of H\"older's inequality yields
	\begin{align*}
		|b(\u,\v,\w)|=|b(\u,\w,\v)|\leq \|\u\|_{\widetilde{\L}^{r+1}}\|\v\|_{\widetilde{\L}^{\frac{2(r+1)}{r-1}}}\|\w\|_{\V},
	\end{align*}
	for all $\u\in\V\cap\widetilde{\L}^{r+1}$, $\v\in\V\cap\widetilde{\L}^{\frac{2(r+1)}{r-1}}$ and $\w\in\V$, so that we obtain  
	\begin{align}\label{2p9}
		\|\B(\u,\v)\|_{\V'}\leq \|\u\|_{\widetilde{\L}^{r+1}}\|\v\|_{\widetilde{\L}^{\frac{2(r+1)}{r-1}}}.
	\end{align}
Using the interpolation inequality, we get 
	\begin{align}\label{212}
		\left|\langle \B(\u,\u),\v\rangle \right|=\left|b(\u,\v,\u)\right|
        \leq \|\u\|_{\widetilde{\L}^{r+1}}\|\u\|_{\widetilde{\L}^{\frac{2(r+1)}{r-1}}}\|\v\|_{\V}
    \leq\|\u\|_{\widetilde{\L}^{r+1}}^{\frac{r+1}{r-1}}\|\u\|_{\H}^{\frac{r-3}{r-1}}\|\v\|_{\V},
	\end{align}
	{for $r> 3$ } and  all $\v\in\V$. Thus, we can  deduce that 
	\begin{align}\label{2.9a}
		\|\B(\u)\|_{\V'}\leq\|\u\|_{\widetilde{\L}^{r+1}}^{\frac{r+1}{r-1}}\|\u\|_{\H}^{\frac{r-3}{r-1}}.
	\end{align}
	Using \eqref{2p9} and the interpolation inequality, for $\u,\v\in\V\cap\widetilde{\L}^{r+1}$, we also have  
	\begin{align}\label{lip}
		\|\B(\u)-\B(\v)\|_{\V'}&\leq \left(\|\u\|_{\H}^{\frac{r-3}{r-1}}\|\u\|_{\widetilde{\L}^{r+1}}^{\frac{2}{r-1}}+\|\v\|_{\H}^{\frac{r-3}{r-1}}\|\v\|_{\widetilde{\L}^{r+1}}^{\frac{2}{r-1}}\right)\|\u-\v\|_{\widetilde{\L}^{r+1}},
	\end{align}
	{for $r> 3$}. Therefore, the map $\B(\cdot):\V\cap\wi\L^{r+1}\to\V'+\wi\L^{\frac{r+1}{r}}$ is locally Lipschitz. 

\begin{remark}\label{rem-trilinear-ext}
    For $s>\frac{d}{2}+1$, $\u,\v\in\H$ and $\w\in\V_s$, we have 
    \begin{align*}
        |b(\u,\v,\w)|=|b(\u,\w,\v)| \leq \|\u\|_{\H}\|\v\|_{\H}\|\nabla\w\|_{\L^{\infty}} \leq C \|\u\|_{\H}\|\v\|_{\H}\|\w\|_{\V_s},
    \end{align*}
    for some positive constant $C$. Thus $b$ can be uniquely extended to the trilinear form (still denoted by the same) $b:\H\times\H \times \V_s \to \R$. In parallel, the operator $\B$ can be uniquely extended to a bounded bilinear operator (still denoted by the same) $\B:\H\times\H\to \V_s^{\prime}$ such that 
    \begin{align}\label{eqn-trilinear-ext}
        \|\B(\u,\v)\|_{\V^{\prime}_s}\leq C \|\u\|_{\H}\|\v\|_{\H}, \;\;\; \u,\v\in\H.
    \end{align}
\end{remark}

\subsection{Nonlinear operator}
Let us introduce the  nonlinear operator defined by
\begin{equation*}
	\mathcal{C}(\u):=-\mathcal{P}_{\frac{r+1}{r}}[|\u|^{r-1}\u],\ \u\in\V\cap\wi\L^{r+1}.
\end{equation*}
It can be easily seen that $\langle\mathcal{C}(\u),\u\rangle =\|\u\|_{\widetilde{\L}^{r+1}}^{r+1}$. An application of the Mean Value Theorem and the H\"older inequality provides the estimate (see \cite[Subsection 2.4, pp. 8]{Gautam+Mohan_2025})
\begin{align}\label{213}
	&\langle \mathcal{C}(\u)-\mathcal{C}(\v),\w\rangle \leq r\left(\|\u\|_{\widetilde{\L}^{r+1}}+\|\v\|_{\widetilde{\L}^{r+1}}\right)^{r-1}\|\u-\v\|_{\widetilde{\L}^{r+1}}\|\w\|_{\widetilde{\L}^{r+1}},
\end{align}
for all $\u,\v, \w\in\V\cap\widetilde{\L}^{r+1}$. 
Thus the operator $\mathcal{C}(\cdot):\V\cap\widetilde{\L}^{r+1}\to \V^{\prime}+\widetilde{\L}^{\frac{r+1}{r}}$ is locally Lipschitz. Furthermore, 	for any $r\in[1,\infty)$, we have (see \cite{KK+MTM-SCBF})
	\begin{align}\label{2.23}
		&\langle\mathcal{C}(\u)-\mathcal{C}(\v),\u-\v\rangle\geq \frac{1}{2}\||\u|^{\frac{r-1}{2}}(\u-\v)\|_{\H}^2+\frac{1}{2}\||\v|^{\frac{r-1}{2}}(\u-\v)\|_{\H}^2\geq \frac{1}{2^{r-1}}\|\u-\v\|_{\wi\L^{r+1}}^{r+1}\geq 0,
	\end{align}
	for $r\geq 1$ 	and all $\u,\v\in\V\cap\wi\L^{r+1}$.

	\subsection{A compact operator}\label{C_O}
In this section, we discuss the existence of an unbounded, self-adjoint operator whose inverse is compact. Further details can be found in \cite[Subsection 2.3]{ZBEM}.

Let us fix $s>2$. Since the embedding $\V_s\hookrightarrow\V$ is dense and continuous, in view of \cite[Lemma C.1]{ZBEM}, there exists a Hilbert space $\mathbb{U}$ such that $\mathbb{U}\subset\V_s$ and $\mathbb{U}\hookrightarrow\V_s$ is dense and  compact. In addition, there exists an onto and self-adjoint operator $\mathfrak{L}:\D(\mathfrak{L}) \subset \mathbb{U} \to\H$ such that $\D(\mathfrak{L})$ is dense in $\H$, $\mathfrak{L}^{-1}$ is compact and 
\begin{align}\label{eqn-compact-op-L}
	(\mathfrak{L}\u,\w)_{\H}=(\u,\w)_{\mathbb{U}}, \ \ \ \ \u\in\D(\mathfrak{L}), \ \w\in\mathbb{U}.
\end{align}
Therefore, there exists an orthonormal basis $\{\boldsymbol{e}_i\}_{i\in\N} \subset \D(\mathfrak{L})\subset \mathbb{U} $ of $\H$ such that 
\begin{align}\label{L3}
\mathfrak{L}\boldsymbol{e}_i=\nu_i\boldsymbol{e}_i,\ \ \ \ i\in\N.
\end{align}
 
Let us choose and fix $m\in\N$. We define an operator $\Pi_m$ from $\mathbb{U}'$ to $\mathrm{span}\{\boldsymbol{e}_1,\ldots,\boldsymbol{e}_m\}=:\H_{m}$  by 
\begin{align}\label{eqn-projection}
\Pi_m\u^*:=\sum_{i=1}^{m}\langle\u^*,\boldsymbol{e}_i\rangle_{\mathbb{U}'\times\mathbb{U}}\boldsymbol{e}_i, \ \ \ \ \ \ \u^*\in\mathbb{U}'.
\end{align}
The restriction of operator $\Pi_m$ to the space $\H$ will be considered and denoted still by the  same. Since, $\H\hookrightarrow\mathbb{U}'$, that is, every element $\u\in\H$ induces a functional $\u^*\in\mathbb{U}'$ by 
\begin{align}
\langle\u^*,\v\rangle_{\mathbb{U}'\times\mathbb{U}}:=(\u,\v)_{\H}, \ \ \ \ \v\in\mathbb{U}.
\end{align}
The restriction of $\Pi_m$ to $\H$ is given by  (still denoted by the  same symbol)
\begin{align}
	\Pi_m\u:=\sum_{i=1}^{m}(\u,\boldsymbol{e}_i)\boldsymbol{e}_i, \ \ \ \ \ \ \u\in\H.
\end{align}
Hence, in particular, $\Pi_m$ is the orthogonal projection from $\H$ onto $\H_m$.  In addition, we also have (see \cite[Lemma 2.3]{ZBEM})
\begin{align}\label{eqn-adjoint-projection}
    (\Pi_m\u^*, \v)_{\H} = \langle \u^*, \Pi_m\v\rangle_{\mathbb{U}' \times\mathbb{U}}, \;\;\; \u^*\in \mathbb{U}', \; \v \in \mathbb{U}.
\end{align}
\begin{lemma}[{\cite[Lemma 2.4]{ZBEM}}]\label{lem-conv-Pi_m}
	For every $\u\in\mathbb{U}$ and $s>2$, we have 
	\begin{itemize}
        \item [(i)] $ \|\Pi_m\u\|_{\mathbb{U}} \leq \|\u\|_{\mathbb{U}}$,
		\item [(ii)] $\lim\limits_{m\to\infty}\|\Pi_m\u-\u\|_{\mathbb{U}}=0$,
		\item [(iii)] $\lim\limits_{m\to\infty}\|\Pi_m\u-\u\|_{\mathbb{V}_s}=0$, 
		\item [(iv)] $\lim\limits_{m\to\infty}\|\Pi_m\u-\u\|_{\mathbb{V}}=0$,
        \item [(v)] $\lim\limits_{m\to\infty}\|\Pi_m\u-\u\|_{\wi\L^{r+1}}=0$.
	\end{itemize}
\end{lemma}

\subsection{A sequence of bounded domains} Let $\{\mathcal{O}_k\}_{k\in\N}$ be a sequence of bounded open subsets of $\mathcal{O}$ with regular boundaries $\partial\mathcal{O}_k$ such that $\mathcal{O}_k \subset \mathcal{O}_{k+1}$ and $\bigcup\limits_{k=1}^{\infty}\mathcal{O}_k =\mathcal{O}$. We will consider the following spaces of restrictions of the functions defined on $\mathcal{O}$ to subsets $\mathcal{O}_k$, that is,
\begin{align}
    \H_{\mathcal{O}_k} := \{\v|_{\mathcal{O}_k} : \v\in\H\}, \;\;\;\;\; \V_{\mathcal{O}_k} := \{\v|_{\mathcal{O}_k} : \v\in\V\}
\end{align}
with appropriate scalar product and norms, that is,
\begin{align*}
    (\v_1,\v_2)_{\H_{\mathcal{O}_k}} & := \int_{\mathcal{O}_k} \v_1\cdot \v_2 \d x, \qquad \v_1,\v_2 \in \H_{\mathcal{O}_k},\\
    (\v_1,\v_2)_{\V_{\mathcal{O}_k}} & := \int_{\mathcal{O}_k} \v_1\cdot \v_2 \d x + \int_{\mathcal{O}_k} \nabla \v_1 : \nabla \v_2 \d x,  \qquad  \v_1,\v_2 \in \V_{\mathcal{O}_k},
\end{align*}
and $\|\v\|^2_{\H_{\mathcal{O}_k}}: = (\v,\v)_{\H_{\mathcal{O}_k}}$ for $\v \in \H_{\mathcal{O}_k}$, and $\|\v\|^2_{\V_{\mathcal{O}_k}}: = (\v,\v)_{\V_{\mathcal{O}_k}}$ for $\v \in \V_{\mathcal{O}_k}$.  We will denote by $\H^{\prime}_{\mathcal{O}_k}$ and $\V^{\prime}_{\mathcal{O}_k}$ for the corresponding dual spaces. Since the sets $\mathcal{O}_k$ are bounded,
\begin{align}
    \mbox{ the embeddings $\V_{\mathcal{O}_k}\hookrightarrow \H_{\mathcal{O}_k}$ is compact. }
\end{align}

		\section{Stochastic Convective Brinkman-Forchheimer Equations} \label{sec3}\setcounter{equation}{0}
		In this section, we discuss  the existence of weak martingale solutions for the stochastic system \eqref{31}. Let us first provide an abstract formulation of the stochastic system \eqref{31}.  By taking the Helmholtz-Hodge projection $\mathcal{P}_{\frac{r+1}{r}}$ onto the first equation in \eqref{31}, we find
		\begin{equation}\label{32}
			\left\{
			\begin{aligned}
				\d\u(t)+ \mu \A\u(t)\d t + \B(\u(t)) \d t + \alpha \u(t)\d t + \beta\mathcal{C}(\u(t)) \d t & =  \mathbf{F}(t) + \boldsymbol{\sigma}(t,\u(t))\W(t), \\
				\u(0)&=\u_0,
			\end{aligned}
			\right.
		\end{equation}
	for $t\in(0,T)$. Strictly speaking, one should write $\mathcal{P}_{\frac{r+1}{r}}\mathbf{F}$ and $\mathcal{P}_{\frac{r+1}{r}}\boldsymbol{\sigma}$   for $\mathbf{F}$ and  $\boldsymbol{\sigma}$, respectively.

\subsection{Stochastic setting} Let $(\Omega,\mathscr{F},\mathbb{P})$ be a complete probability space equipped with an increasing family of sub-sigma fields $\{\mathscr{F}_t\}_{t\geq 0}$ of $\mathscr{F}$ satisfying the usual conditions.

\subsubsection{$\Q$-Wiener process} Firstly, we provide the definition and properties of $\Q$-Wiener processes. Let $\H$ be a separable Hilbert space. 
\begin{definition}
	A stochastic process $\{\W(t)\}_{t\geq 0}$ is said to be an \emph{$\H$-valued $\mathscr{F}_t$-adapted
		$\Q$-Wiener process} with covariance operator $\Q$ if
	\begin{enumerate}
		\item [$(i)$] for each non-zero $h\in \H$, $\|\Q^{1/2}h\|_{\H}^{-1} (\W(t), h)$ is a standard one dimensional Wiener process,
		\item [$(ii)$] for any $h\in \H, (\W(t), h)$ is a martingale adapted to $\mathscr{F}_t$.
	\end{enumerate}
\end{definition}
The stochastic process $\{\W(t)\}_{t\geq 0}$ is a $\Q$-Wiener process with covariance $\Q$ if and only if for arbitrary $t$, the  random field  $\W(\cdot)$ can be expressed as $\W(\cdot,x) =\sum_{k=1}^{\infty}\sqrt{\mu_k}\boldsymbol{q}_k(x)\beta_k(\cdot)$, where  $\beta_{k}(\cdot),k\in\mathbb{N}$, are independent one dimensional Brownian motions on $(\Omega,\mathscr{F},\mathbb{P})$ and $\{\boldsymbol{q}_k \}_{k=1}^{\infty}$ is an orthonormal basis of $\H$ such that $\Q \boldsymbol{q}_k=\mu_k \boldsymbol{q}_k$.  If $\W(\cdot)$ is a $\Q$-Wiener process  with $\Tr \Q=\sum_{k=1}^{\infty} \mu_k< +\infty$, then $\W(\cdot)$ is a Gaussian process on $\H$ and $ \E[\W(t)] = 0,$ $\textrm{Cov} [\W(t)] = t\Q,$ $t\geq 0.$ The space $\H_0=\Q^{1/2}\H$ is a Hilbert space equipped with the inner product $(\cdot, \cdot)_0$, $$(\u, \v)_0 =\sum_{k=1}^{\infty}\frac{1}{\mu_k}(\u,\boldsymbol{q}_k)(\v,\boldsymbol{q}_k)= (\Q^{-1/2}\u, \Q^{-1/2}\v),\ \text{ for all } \ \u, \v\in \H_0,$$ where $\Q^{-1/2}$ is the pseudo-inverse of $\Q^{1/2}$.

Let $\mathcal{L}(\H)$ denote the space of all bounded linear operators on $\H$ and $\mathcal{L}_{\Q}:=\mathcal{L}_{\Q}(\H)$ represent the space of all Hilbert-Schmidt operators from $\H_0:=\Q^{1/2}\H$ to $\H$.  Since $\Q$ is a trace class operator, the embedding of $\H_0$ in $\H$ is Hilbert-Schmidt and the space $\mathcal{L}_{\Q}$ is a Hilbert space equipped with the norm $ \left\|\Phi\right\|^2_{\mathcal{L}_{\Q}}=\Tr\left(\Phi {\Q}\Phi^*\right)=\sum_{k=1}^{\infty}\| {\Q}^{1/2}\Phi^*\boldsymbol{q}_k\|_{\H}^2 $ and inner product $ \left(\Phi,\Psi\right)_{\mathcal{L}_{\Q}}=\Tr\left(\Phi {\Q}\Psi^*\right)=\sum_{k=1}^{\infty}\left({\Q}^{1/2}\Psi^*\boldsymbol{q}_k,{\Q}^{1/2}\Phi^*\boldsymbol{q}_k\right) $. For more details, the interested readers are referred to see \cite{DaZ}.

Let us now provide the assumptions satisfied by the noise coefficient $\boldsymbol{\sigma}$.

\begin{hypothesis}\label{hyp}
We assume that  $\{\W(t)\}_{t\geq 0}$ is an $\H$-valued $\Q$-Wiener process 
on the stochastic basis $(\Omega,\mathscr{F},\{\mathscr{F}_t\}_{t\geq 0},\mathbb{P})$. The noise coefficient $\boldsymbol{\sigma}(\cdot,\cdot)$ satisfies: 
	\begin{itemize}
		\item [(H.1)] The function $\boldsymbol{\sigma}\in\C([0,T]\times(\V\cap\wi\L^{r+1});\mathcal{L}_{\Q}(\H))$.
		\item[(H.2)]  (Growth condition)
		There exists a positive
		constant $K$ such that for all $t\in[0,T]$ and $\u\in\H$,
		\begin{equation*}
		\|\boldsymbol{\sigma}(t,\u)\|_{\mathcal{L}_{\Q}}^2  	\leq K \left(1 +\|\u\|_{\H}^{2}\right).
		\end{equation*}
		
		\item[(H.3)]  (Lipschitz condition)
		There exists a constant $L\geq 0$ such that for any $t\in[0,T]$ and all $\u_1,\u_2\in\H$,
		\begin{align*}
			\|\boldsymbol{\sigma}(t,\u)-\boldsymbol{\sigma}(t,\v)\|_{\mathcal{L}_{\Q}}^2 	\leq L\|\u_1 -	\u_2\|_{\H}^2.
		\end{align*} 
        \item[(H.4)]   (A continuity condition)
		For every $\psi\in \mathcal{V}$, the mapping 
        \begin{align*}
            \u\mapsto \mathrm{Q}^{\frac{1}{2}} \boldsymbol{\sigma}(\cdot,{\u})^* \psi \in \H \text{ is continuous},
        \end{align*}
        if in the space $\H$ we consider the Fr\'echet topology inherited from the space $\mathbb{L}_{\mathrm{loc}}^{2}(\mathcal{O};\R^d)$.
	\end{itemize}
\end{hypothesis}	

\begin{remark}
    An example satisfying all the conditions of Hypothesis~\ref{hyp} can be constructed by following the procedure described in \cite[Section~6]{ZBEMMO}.
\end{remark}

\begin{remark}
    Hypothesis \ref{hyp} (H.4) is essential in unbounded domain case for the approach which we are following in this article. But if one makes use of the  monotonicity method, such assumption is not required (see \cite{KK+MTM-SCBF}). However, monotonicity method can only be used for the cases given in Table \ref{Table-4}. In addition, Hypothesis \ref{hyp} (H.4) is not required in the case of bounded domains.
\end{remark}

	\subsection{Weak martingale and strong solutions} Let us first  provide the definition of weak martingale solutions  for the SCBFEs \eqref{32}. 
	\begin{definition}[Weak martingale solution]\label{defn3.1}
		A weak martingale solution of the SCBFEs \eqref{32} is a system $((\bar{\Omega},\bar{\mathscr{F}},\{\bar{\mathscr{F}}_t\}_{t\geq 0},\bar{\mathbb{P}}),\bar{\u},\bar{\W}),$ where 
		\begin{enumerate}
			\item [(a)] $(\bar{\Omega},\bar{\mathscr{F}},\{\bar{\mathscr{F}}_t\}_{t\geq 0},\bar{\mathbb{P}})$ is a filtered probability space with a filtration $\{\bar{\mathscr{F}}_t\}_{t\geq 0}$, that is, a set of  sub $\sigma$-fields of $\mathscr{F}$ with $\mathscr{F}_s\subset\mathscr{F}_t\subset\mathscr{F}$ for $0\leq s<t<\infty$, 
			\item [(b)] $\bar{\W}$ is a $\Q$-Wiener process on  $(\bar{\Omega},\bar{\mathscr{F}},\{\bar{\mathscr{F}}_t\}_{t\geq 0},\bar{\mathbb{P}})$,
			\item [(c)] $\bar{\u}:[0,T]\times\bar\Omega\to\H$ is a predictable process with $\bar{\mathbb{P}}$-a.e. paths  $$\bar{\u}(\cdot,\omega)\in\C([0,T];\H_w)\cap\mathrm{L}^2(0,T;\V)\cap\mathrm{L}^{r+1}(0,T;\wi\L^{r+1})$$ such that for all $t\in[0,T]$ and all $\v\in\V\cap\wi{\L}^{r+1}$, the following identity holds $\bar{\mathbb{P}}$-a.s.:
			\begin{align}\label{3.2}
				(\bar\u(t),\v)&=(\u_0,\v)-\int_0^t\langle\mu \A\bar\u(s)+\B(\bar\u(s)) + \alpha \u(s) +\beta\mathcal{C}(\bar\u(s)) - \mathbf{F}(s) ,\v\rangle\d s 
                \nonumber\\ & \quad + \int_0^t\left(\boldsymbol{\sigma}(s,\bar\u(s))\d\bar\W(s),\v\right).
			\end{align}
		\end{enumerate}
	\end{definition}

	
Next, we present the definition of strong  solutions  for the SCBFEs \eqref{32}.

	\begin{definition}
	We say that the stochastic system \eqref{32}  has a \emph{strong solution} if and only if for every stochastic basis $(\Omega,\mathscr{F},\{\mathscr{F}\}_{t\geq 0},\mathbb{P})$ and $\Q$-Wiener processes $\{\W(t)\}_{t\geq 0}$ defined on this stochastic basis,  there exists a progressively measurable process $\u :[0,T]\times\Omega\to\H$ with $\mathbb{P}$-a.s. paths
		\begin{align*}
			\u(\cdot,\omega)\in\C([0,T];\H_w)\cap\mathrm{L}^2(0,T;\V)\cap\mathrm{L}^{r+1}(0,T;\wi\L^{r+1})
		\end{align*}
		such that for all $t \in[0,T]$  and all $\v\in\V\cap\wi\L^{r+1}$, the following identity holds $\mathbb{P}$-a.s.:
		\begin{align}
			 (\u(t),\v) 
              & =(\u_0,\v)-\int_0^t\langle\mu \A\u(s)+\B(\u(s)) + \alpha \u(s) +\beta\mathcal{C}(\u(s)) - \mathbf{F}(s),\v\rangle\d s 
              \nonumber\\ & \quad + \int_0^t\left(\boldsymbol{\sigma}(s,\u(s))\d\W(s),\v\right).
		\end{align}
	\end{definition}
	
For all the cases given in Table \ref{Table-3}, the existence of strong solutions for the SCBFEs  \eqref{32} satisfying the energy equality (It\^o formula, see \eqref{3p5} below)  is established in \cite{KK+MTM-SCBF,MTM6}, etc. via monotonicity methods. Let us now recall two important concepts of uniqueness of the solution, that is,  pathwise uniqueness and uniqueness in law (cf. \cite{IW}). 
	
	\begin{definition}
		We say that solutions of the stochastic system \eqref{32} are pathwise unique if and only if the following condition	holds:
		\begin{itemize}
			\item [] if $((\Omega,\mathscr{F},\{\mathscr{F}\}_{t\geq 0},\mathbb{P}),\W,\u^i)$, $i=1,2,$ are such solutions of the stochastic system \eqref{32} that $\u^i(0)=\u_0$, $i=1,2,$ then  for all $t\in[0,T]$, $\u^1(t)=\u^2(t)$, $\mathbb{P}$-a.s. 
		\end{itemize}
	\end{definition}
	\begin{definition}
		We say that solutions of the stochastic system \eqref{32} are unique in law if and only if  the following condition	holds:
		\begin{itemize}
			\item [] if $((\Omega^i,\mathscr{F}^i,\{\mathscr{F}^i\}_{t\geq 0},\mathbb{P}^i),\W^i,\u^i)$, $i=1,2,$ are such solutions of the stochastic system \eqref{32} that $\u^i(0)=\u_0$, $i=1,2,$ then $\mathscr{L}_{\mathbb{P}^1}(\u^1)=\mathscr{L}_{\mathbb{P}^2}(\u^2)$,  where $\mathscr{L}_{\mathbb{P}^i}(\u^i)$ denotes the law of $\u^i$ with respect to $(\Omega^i,\mathscr{F}^i,\{\mathscr{F}^i\}_{t\geq 0},\mathbb{P}^i)$.
		\end{itemize}
	\end{definition}
The pathwise uniqueness implies  uniqueness in law  (cf. \cite{Me}). For $\u_0\in\H$, the stochastic system \eqref{32} admits a  weak martingale  solution and if it has  pathwise  uniqueness property,  then for any given $(\Omega,\mathscr{F},\{\mathscr{F}\}_{t\geq 0},\mathbb{P})$ and any Brownian motion with covariance $\Q$ on this stochastic basis, the problem \eqref{1} has a unique strong solution (cf. \cite{Me,OMa,MRBS}, etc.).  

The main results of this work corresponding to martingale and strong solutions are the following: 
\begin{theorem}\label{thm3.4}
	Let $\u_0\in\H$ and  $\mathbf{F}\in \mathrm{L}^{2}(0,T;\V')$. Under Hypothesis \ref{hyp}, for all the cases given in Table \ref{Table-1}, the stochastic system \eqref{32} has a weak martingale solution $((\bar{\Omega},\bar{\mathscr{F}},\{\bar{\mathscr{F}}_t\}_{t\geq 0},\bar{\mathbb{P}}),\bar{\W},\bar{\u})$ in the sense of Definition \ref{defn3.1}. Furthermore, the solution satisfies the estimate for all $p\geq 1$:
	\begin{align}
    \E \left[ \sup_{t\in[0,T]} \|\bar{\u}(t)\|^{2p}_{\H}+ p\int_{0}^{T}\|\bar{\u}(s)\|_{\H}^{2p-2}\left(\mu\|\nabla\bar{\u}(s)\|^2_{\H}+\alpha\|\bar{\u}(s)\|^2_{\H}+2\beta\|\bar{\u}(s)\|^{r+1}_{\wi\L^{r+1}}\right)\d s\right] <\infty.
	\end{align}
\end{theorem}

\begin{theorem}\label{thm3.5}
	Let $\u_0\in\H$, $\mathbf{F}\in \mathrm{L}^{2}(0,T;\V')$  and  Hypothesis \ref{hyp} be satisfied. 
	\begin{enumerate}
    \item [(1)] For all the cases given in Table \ref{Table-2}, if $((\Omega,\mathscr{F},\{\mathscr{F}\}_{t\geq 0},\mathbb{P}),\W,\u)$ is a weak martingale solution of the stochastic system \eqref{32}, then for $\mathbb{P}$-almost all $\omega\in\Omega$, the trajectory $\u(\cdot,\omega)$ is continuous $\H$-valued function defined on $[0, T]$ satisfying the energy equality (It\^o's formula):
			\begin{align}\label{3p5}
			&	\|{\u}(t)\|_{\H}^2+2\mu \int_0^t\|\nabla \u(s)\|_{\H}^2\d s+ 2\alpha \int_0^T\|\u(t)\|_{\H}^2\d t  +2\beta\int_0^t\|{\u}(s)\|_{\widetilde{\L}^{r+1}}^{r+1}\d s\nonumber\\&=\|{\u_0}\|_{\H}^2 + 2 \int_0^{t}\left\langle \mathbf{F}(s),{\u}(s)\right\rangle \d s + 2 \int_0^{t}(\boldsymbol{\sigma}(s,{\u}(s))\d{\W}(s),{\u}(s))+  \int_0^{t}\|\boldsymbol{\sigma}(s,{\u}(s))\|_{\mathcal{L}_{\Q}}^2\d s,
		\end{align}
		for all $t\in[0,T]$, ${\mathbb{P}}$-a.s. 
		\item [(2)] For all the cases given in Table \ref{Table-3}, pathwise uniqueness of weak solutions of the stochastic system \eqref{32} holds. In addition, the classical Yamada-Watanabe argument ensures the existence of a unique strong solution and uniqueness in law.
	\end{enumerate}
\end{theorem}
We prove  Theorem \ref{thm3.4} in Section \ref{sec4} and Theorem \ref{thm3.5} in Section \ref{sec5}. 
In the subsections below, we furnish some preliminaries and important results to establish Theorem \ref{thm3.4}.

\subsection{Deterministic compactness criterion}
Let us introduce the following spaces, which are used frequently in the paper: 
\begin{itemize}
	\item $\C([0,T];\mathbb{U}'):=$ the space of continuous functions $\v:[0,T]\to\mathbb{U}'$ with the topology $\mathcal{T}_1$ induced by the norm $\|\v\|_{\C([0,T];\mathbb{U}')}:= \sup\limits_{t\in[0,T]} \|\v(t)\|_{\mathbb{U}'}$, 
	\item $\mathrm{L}_w^2(0,T;\V):=$ the space $\mathrm{L}^2(0,T;\V)$ with the weak topology $\mathcal{T}_2$, 
	\item $\mathrm{L}_w^{r+1}(0,T;\wi\L^{r+1}):=$ the space $\mathrm{L}^{r+1}(0,T;\wi\L^{r+1})$ with the weak topology $\mathcal{T}_3$, 
	\item $\mathrm{L}^2(0,T;\H_{\mathrm{loc}}):=$ the space of measurable functions $\v:[0,T]\to\H$ such that for all $k\in\N$ 
    \begin{align*}
        \mathfrak{q}_{k,T};= \|\v\|_{\mathrm{L}^2(0,T;\H_{\mathcal{O}_k})}:= \left(\int_0^T \int_{\mathcal{O}_k} |\v(x,t)|^2 \d x \d t\right)^{\frac12} < + \infty,
    \end{align*}
    with the topology $\mathcal{T}_4$ generated by the seminorms $\{\mathfrak{q}_{k,T}\}_{k\in\N}$. 
\end{itemize}
Furthermore, we consider 
\begin{itemize}
	\item $\C([0,T];\H_w):=$ the space of weakly continuous functions $\v:[0,T]\to\H$ endowed with the weakest topology $\mathcal{T}_5$ such that for all $\boldsymbol{\psi}\in\H,$ the mapping $\C([0,T];\H_w)\ni\v\mapsto(\v(\cdot),\boldsymbol{\psi})\in \C([0,T];\R)$ is  continuous. In particular, $\v_n\to\v$ in $\C([0,T];\H_w)$ if and only if for all $\boldsymbol{\psi}\in\H: (\v_n(\cdot),\boldsymbol{\psi})\to (\v(\cdot),\boldsymbol{\psi})$ in the space $\C([0,T];\R)$. 
\end{itemize}
A proof of the following result can be obtained from \cite[Theorem 2]{EMo}, which is a generalization of \cite[Lemma 3.3]{ZBEM}. 
		\begin{lemma}[Compactness criterion]\label{thm3.7}
            Let 
			\begin{align}\label{eqn-intersection-spaces}
				\mathscr{Y}:=\C([0,T]; \mathbb{U}^{\prime} ) \cap\mathrm{L}_w^{2}(0,T;\V)\cap\mathrm{L}_w^{r+1}(0,T;\wi\L^{r+1})\cap \mathrm{L}^{2}(0,T;\H_{\text{\rm loc}}) \cap\C([0,T];\H_w)
			\end{align}
		and let $\mathcal{T}$ be the supremum of of the corresponding topologies. Then a set $\mathcal{K}\subset\mathscr{Y}$ is $\mathcal{T}$-relatively compact if the following four conditions hold:
		\begin{enumerate}
			\item $\sup\limits_{\v\in\mathcal{K}}\sup\limits_{t\in[0,T]}\|\v(t)\|_{\H}<\infty$, 
			\item $\sup\limits_{\v\in\mathcal{K}}\int_0^T\|\v(t)\|_{\V}^{2}\d t<\infty$, that is, $\mathcal{K}$ is bounded in $\mathrm{L}^{2}(0,T;\V)$, 
			\item $\sup\limits_{\v\in\mathcal{K}}\int_0^T\|\v(t)\|_{\wi\L^{r+1}}^{r+1}\d t<\infty$, that is, $\mathcal{K}$ is bounded in $\mathrm{L}^{r+1}(0,T;\wi\L^{r+1})$, 
	       \item   $\lim\limits_{\delta\to 0}\sup\limits_{\v\in\mathcal{K}} \sup\limits_{\substack{s,t\in [0,T], \\ |t-s|\leq \delta}} \|\v(t) - \v (s)\|_{\mathbb{U}^{\prime}}  =  0$.
		\end{enumerate}
		\end{lemma}
	It should be noted that the space $\mathscr{Y}$ is not a Polish space.

 \subsection{Tightness criterion}

Let $(\mathcal{E}, \uprho)$ be a separable and complete metric space.
\begin{definition}
    Let $\v\in \C([0,T];\mathcal{E})$. The modulus of continuity of $\v$ on $[0,T]$ is defined by 
    \begin{align*}
        \mathfrak{m} (\v,\delta):= \sup\limits_{\substack{s,t\in [0,T], \\ |t-s|\leq \delta}} \uprho(\v(t)- \v(s)), \;\; \delta>0.
    \end{align*}
\end{definition}

Let $(\Omega,\mathscr{F},\mathbb{P})$ be a probability space with the filtration $\{\mathscr{F}_t\}_{t\geq 0}$ satisfying the usual conditions. Let $\{\y_n\}_{n\in\N}$ be a sequence of continuous $\mathscr{F}_t$-adapted and $\mathcal{E}$-valued processes. 
\begin{definition}
	 The sequence $\{\y_n\}_{n\in\mathbb{N}}$ satisfies condition $\mathbf{[T]}$  in the space $\mathcal{E}$ if and only if 
	\begin{enumerate}
		\item [$\mathbf{[T]}$] for every $\e,\zeta>0$, there exists a $\delta>0$ such that $$\sup_{n\in\N}\mathbb{P}\left\{\mathfrak{m} (\y_n,\delta)>\zeta\right\}\leq \e.$$ 
	\end{enumerate}
\end{definition}
\begin{definition}
	A sequence $\{\y_n\}_{n\in\mathbb{N}}$ satisfies the \emph{Aldous condition} in the space $\mathcal{E}$ if and only if
	\begin{enumerate}
	\item [$\mathbf{[A]}$] for every $\e,\zeta>0$, there exists a $\delta>0$ such that for every sequence of $\{\mathscr{F}_t\}_{t\geq 0}$
    -adapted stopping times  $\{\tau_n\}_{n\in\mathbb{N}}$ with $\tau_n\leq T$, one has 
	$$\sup_{n\in\N}\sup_{0\leq\theta\leq\delta}\mathbb{P}\left\{\uprho(\y_n(\tau_n+\theta),\y_n(\tau_n))\geq\zeta\right\}\leq\e.$$
	\end{enumerate}
\end{definition}
A proof of the following result can be found in \cite[Theorem 2.2.2]{AJMM}. 
\begin{lemma}
	 The condition $\mathbf{[A]}$ implies condition $\mathbf{[T]}$.
\end{lemma}
The following result provides a condition which guarantees that the sequence $\{\y_n\}_{n\in\N}$ satisfies condition $\mathbf{[A]}$ (see \cite[Lemma 9]{EMo}). 

\begin{lemma}\label{lem3.11}
	Let $(\mathbb{X},\|\cdot\|_{\mathbb{X}})$ be a separable Banach space and let $\{\y_n\}_{n\in\N}$  be a sequence of $\X$-valued processes. Assume that for every $\{\tau_n\}_{n\in\N}$ of $\mathscr{F}_t$-stopping times with $\tau_n\leq T$ and for every $n\in\N$ and $\theta\geq 0$, the following holds: 
	\begin{align}\label{37}
		\E\left[\|\y_n(\tau_n+\theta)-\y_n(\tau_n)\|_{\X}^{\xi}\right]\leq C\theta^{\eta},
	\end{align}
for some $\xi,\eta>0$ and some constant $C > 0$. Then the sequence $\{\y_n\}_{n\in\N}$ satisfies the \emph{Aldous condition} in the space $\X$.
\end{lemma}

From the deterministic compactness criterion in Lemma \ref{thm3.7}, we deduce the following corollary, which will serve as a tool to establish the tightness of the laws associated with the Galerkin approximations. To begin, recall that $\mathbb{U}$ is a Hilbert space satisfying
\begin{align*}
    \mathbb{U} \hookrightarrow \V \hookrightarrow \H
\end{align*}
and the embedding $\mathbb{U} \hookrightarrow \V$ is dense and compact. Moreover, we consider the following space
\begin{align*}
    \mathscr{Y}:=\C([0,T]; \mathbb{U}^{\prime} ) \cap\mathrm{L}_w^{2}(0,T;\V)\cap
    \mathrm{L}_w^{r+1}(0,T;\wi\L^{r+1})\cap \mathrm{L}^{2}(0,T;\H_{\text{\rm loc}}) \cap\C([0,T];\H_w)
\end{align*}
equipped with the topology $\mathcal{T}$, see \eqref{eqn-intersection-spaces}.


\begin{corollary}\label{cor3.12}
    Let $\{\v_n\}_{n\in\N}$ be a sequence of continuous $\{\mathscr{F}_t\}_{t\geq 0}$-adapted $\mathbb{U}^{\prime}$-valued processes such that 
	\begin{enumerate}
		\item [(a)] there exists a positive constant $M_1$ such that 
		\begin{align*}
\sup_{n\in\N}\E\left[\sup_{t\in[0,T]}\|\v_n(t)\|_{\H}\right]\leq M_1,
		\end{align*}
	\item [(b)] there exists a positive constant $M_2$ such that 
	\begin{align*}
\sup_{n\in\N}\E\left[\int_0^T\|\v_n(t)\|_{\V}^2\d t\right]\leq M_2,
	\end{align*}
\item [(c)] there exists a positive constant $M_3$ such that 
\begin{align*}
\sup_{n\in\N}\E\left[\int_0^T\|\v_n(t)\|_{\wi\L^{r+1}}^{r+1}\d t\right]\leq M_3,
\end{align*}
\item [(d)] $\{\v_n\}_{n\in\N}$ satisfies the \emph{Aldous condition} in $\mathbb{U}^{\prime}$. 
	\end{enumerate}
Let $\mathbb{P}_n$ be the law of $\v_n$ on $\mathscr{Y}$. Then for every $\e>0,$ there exists a compact subset $\mathcal{K}_\e$ of $\mathscr{Y}$  such that
\begin{align*}
	\mathbb{P}_n(\mathcal{K}_{\e})\geq 1-\e. 
\end{align*}
\end{corollary}
\begin{proof}
    For the proof, we refer readers to \cite[Corollary 3.9]{ZBEM}.
\end{proof}

\subsection{The Skorokhod Theorem}
We first remember the following Jakubowski's version of the Skorokhod theorem: 
\begin{theorem}[{\cite[Theorem 2]{AJa}}]\label{thm-JvST}
	Let $(\mathscr{X},\mathcal{T})$ be a topological space such that there exists a sequence $\{f_k\}$ of continuous functions $f_k :\mathscr{X} \to\R$  that separates points	of $\mathscr{X}$. Let $\{X_n\}$  be a sequence of $\mathscr{X}$-valued random variables. Suppose that for every $\e > 0$, there exists a compact subset $K_\e\subset \mathscr{X}$ such that
	\begin{align*}
		\inf_{n\in\N}\mathbb{P}\left(\left\{X_n\in K_{\e}\right\}\right)>1-\e. 
	\end{align*}
Then there exist a subsequence $\{X_{n_k}\}_{k\in\N}$, a sequence $\{Y_k\}_{k\in\N}$ of $\mathscr{X}$-valued random variables and a $\mathscr{X}$-valued random variable $Y$ defined on some probability space $(\Omega,\mathscr{F},\mathbb{P})$ such that
\begin{align*}
	\mathscr{L}(X_{n_k})=\mathscr{L}(Y_k), \ k=1,2,\ldots, 
\end{align*}
and   
\begin{align*}
	Y_k\xrightarrow{\mathcal{T}} Y \quad \mathbb{P} \text{-a.e.}\quad  \text{ as }\ k\to\infty. 
\end{align*}
\end{theorem}

\subsection{Convergences associated with nonlinear operators} The following convergence results will be used to establish the existence of martingale solutions by enabling us to pass to the limit in the finite-dimensional approximated system.
\begin{lemma}[{\cite[Corollary B.2]{ZBEM}}]\label{lem-convergence-B}
    Let $\v\in \mathrm{L}^2(0,T;\H)$ and $\{\v_n\}_{n\in\N}$ be a bounded sequence in $\mathrm{L}^2(0,T;\H)$ such that $\v_n\to \v$ in $\mathrm{L}^2(0,T;\H_{\mathrm{loc}})$. Then for all $t\in[0,T]$ and for all $\varphi\in\mathbb{U}$:
    \begin{align}
        \lim_{n\to \infty}  \int_0^t \langle\B({\v}_n(\tau)), \Pi_n\varphi\rangle\d \tau = \int_0^t \langle\B ({\v}(\tau)), \varphi\rangle\d \tau.
    \end{align}
\end{lemma}

\begin{lemma}\label{lem-convergence-C}
    Let $\v\in \mathrm{L}^2(0,T;\H)\cap\mathrm{L}^{r+1}(0,T;\wi\L^{r+1})$ and $\{\v_n\}_{n\in\N}$ be a bounded sequence in $\mathrm{L}^2(0,T;\H)\cap\mathrm{L}^{r+1}(0,T;\wi\L^{r+1})$ such that $\v_n\to \v$ in $\mathrm{L}^2(0,T;\H_{\mathrm{loc}})$. Then for all $t\in[0,T]$ and for all $\varphi\in\mathbb{U}$:
    \begin{align}
        \lim_{n\to \infty}  \int_0^t \langle\mathcal{C}({\v}_n(\tau)), \Pi_n\varphi\rangle\d \tau = \int_0^t \langle\mathcal{C} ({\v}(\tau)), \varphi\rangle\d \tau.
    \end{align}
\end{lemma}
\begin{proof}
Let us consider for $\varphi\in\mathbb{U}$
    \begin{align}
     &  \int_0^t [\langle\mathcal{C}({\v}_n(\tau)), \Pi_n\varphi\rangle - \langle\mathcal{C} ({\v}(\tau)), \varphi\rangle ]\d \tau 
     \nonumber\\ & = 
       \int_0^t \langle\mathcal{C}({\v}_n(\tau)), \Pi_n\varphi - \varphi\rangle \d \tau   +   \int_0^t \langle \mathcal{C}({\v}_n(\tau)) - \mathcal{C} ({\v}(\tau)), \varphi\rangle \d \tau
      =: \mathcal{C}_1^n (t)  +  \mathcal{C}_2^n (t).
    \end{align}
    Using Lemma \ref{lem-conv-Pi_m} and boundedness of sequence $\{\v_n\}_{n\in\N}$ in $\mathrm{L}^{r+1}(0,T;\wi\L^{r+1})$, we get
    \begin{align}
        |\mathcal{C}_1^n (t)|\leq  \|\Pi_n\varphi - \varphi\|_{\wi\L^{r+1}} \int_0^t \|\v_n(\tau)\|^r_{\wi\L^{r+1}}\d \tau  \to 0 \text{ as } n\to \infty. 
    \end{align}

 We first show that $|\mathcal{C}_2^n (t)|\to 0$ as $n\to\infty$ for all $\varphi\in\mathcal{V}$. Let us choose $\varphi\in\mathcal{V}$ and fix it. Then there exists  $R>0$ such that $\mathrm{supp}\;\varphi$ is a compact subset of $\mathcal{O}_{R}$. An application of the Mean Value Theorem and H\"older's inequality (see \cite[Subsection 2.4, pp. 8]{Gautam+Mohan_2025}) yields 
\begin{align}\label{eqn-312}
	& \hspace{-10mm} \left|\int_0^t \langle \mathcal{C}({\v}_n(\tau)) - \mathcal{C} ({\v}(\tau)), \varphi\rangle \d \tau \right|  \nonumber\\ &  \hspace{-10mm}
    \leq r  \|\varphi\|_{{\L}^{\infty}(\mathcal{O}_{R})} \int_0^t \left(\|\v_n(\tau)\|_{{\L}^{r+1}(\mathcal{O}_{R})} +\|\v(\tau)\|_{{\L}^{r+1}(\mathcal{O}_{R})}\right)^{r-1}\|\v_n(\tau) -\v(\tau)\|_{{\L}^{\frac{r+1}{2}}(\mathcal{O}_{R})} \d \tau 
    \nonumber\\ & \hspace{-10mm} \leq r \|\varphi\|_{{\L}^{\infty}(\mathcal{O}_{R})}
    \begin{cases}
     |\mathcal{O}_R|^{\frac{3-r}{2(r+1)}} \int_0^t  \left(\|\v_n(\tau)\|_{{\wi\L}^{r+1}} +\|\v(\tau)\|_{{\wi\L}^{r+1}}\right)^{r-1}\|\v_n(\tau) -\v(\tau)\|_{{\L}^{{2}}(\mathcal{O}_{R})} \d \tau \vspace{2mm}
     \\
   & \vspace{2mm} \hspace{-40mm} \text{ for } r\in [1,3),
     \\
    \int_0^t \left(\|\v_n(\tau)\|_{{\wi\L}^{r+1}} +\|\v(\tau)\|_{{\wi\L}^{r+1}}\right)^{r-1}\|\v_n(\tau) -\v(\tau)\|^{\frac{2}{r-1}}_{{\L}^{2}(\mathcal{O}_{R})} \|\v_n(\tau) -\v(\tau)\|^{\frac{r-3}{r-1}}_{{\wi\L}^{{r+1}}} \d \tau  \vspace{2mm} \\
    & \vspace{2mm}  \hspace{-40mm} \text{ for } r\geq3,
    \end{cases}
    \nonumber\\ & \hspace{-10mm} \to 0 \text{ as } n\to \infty.
\end{align}
If $\varphi\in {\mathbb{U}}$, then for every $\delta>0,$ there exists $\varphi_{\delta}\in \mathcal{V}$ such that $\|\varphi - \varphi_{\delta}\|_{\wi\L^{\frac{r+1}{r}}}\leq \delta$. Therefore, we have 
\begin{align}
   & \left| \int_0^t \langle \mathcal{C}({\v}_n(\tau)) - \mathcal{C} ({\v}(\tau)), \varphi\rangle \d \tau \right| 
   \nonumber\\ & \leq \left| \int_0^t \langle \mathcal{C}({\v}_n(\tau)) - \mathcal{C} ({\v}(\tau)), \varphi - \varphi_{\delta} \rangle \d \tau \right| + \left| \int_0^t \langle \mathcal{C}({\v}_n(\tau)) - \mathcal{C} ({\v}(\tau)), \varphi_{\delta}\rangle \d \tau \right|
   \nonumber\\ & \leq \left[ \int_0^t \|{\v}_n(\tau)\|_{\wi\L^{r+1}}^{r} + \|{\v}(\tau)\|_{\wi\L^{r+1}}^{r} \right] \| \varphi - \varphi_{\delta} \|_{\wi\L^{\frac{r+1}{r}}}  + \left| \int_0^t \langle \mathcal{C}({\v}_n(\tau)) - \mathcal{C} ({\v}(\tau)), \varphi_{\delta}\rangle \d \tau \right|
   \nonumber\\ & \leq K \delta   + \left| \int_0^t \langle \mathcal{C}({\v}_n(\tau)) - \mathcal{C} ({\v}(\tau)), \varphi_{\delta}\rangle \d \tau \right|,
\end{align}
where $K$ is a constant independent of $\delta$ and $n$. Passing to upper limit as $n\to \infty$ and using \eqref{eqn-312}, we find
\begin{align}
   & \limsup_{n\to\infty}\left| \int_0^t \langle \mathcal{C}({\v}_n(\tau)) - \mathcal{C} ({\v}(\tau)), \varphi\rangle \d \tau \right| \leq K \delta. 
\end{align}
Since, $\delta>0$ is arbitrary, we obtain that $|\mathcal{C}_2^n (t)|\to 0$ as $n\to\infty$ for all $\varphi\in\mathbb{U}$. Hence, the proof is completed.
\end{proof}

The following result can be obtained by arguing analogously to the proofs of Lemmas \ref{lem-convergence-B}–\ref{lem-convergence-C}.
\begin{corollary}\label{cor-convergence-B+C}
    Let $\v\in \mathrm{L}^2(0,T;\H)\cap\mathrm{L}^{2}(0,T;\V)\cap\mathrm{L}^{r+1}(0,T;\wi\L^{r+1})$ and $\{\v_n\}_{n\in\N}$ be a bounded sequence in $\mathrm{L}^2(0,T;\H)\cap\mathrm{L}^{2}(0,T;\V)\cap\mathrm{L}^{r+1}(0,T;\wi\L^{r+1})$ such that $\v_n\to \v$ in $\mathrm{L}^2(0,T;\H_{\mathrm{loc}})$. Then for all $t\in[0,T]$ and for all $\varphi\in\V\cap\wi\L^{r+1}$:
    \begin{align}
        \lim_{n\to \infty}  \int_0^t \langle\B({\v}_n(\tau)),  \varphi\rangle\d \tau = \int_0^t \langle\B ({\v}(\tau)), \varphi\rangle\d \tau,
    \end{align}
    and
    \begin{align}
        \lim_{n\to \infty}  \int_0^t \langle\mathcal{C}({\v}_n(\tau)),  \varphi\rangle\d \tau = \int_0^t \langle\mathcal{C} ({\v}(\tau)), \varphi\rangle\d \tau.
    \end{align}
\end{corollary}

\section{Existence of a Weak Martingale Solution} \label{sec4}\setcounter{equation}{0} In this section, we prove Theorem \ref{thm3.4} using the classical Faedo-Galerkin approximation, a stochastic compactness method,   Jakubowski's version of the Skorokhod theorem for nonmetric spaces and martingale representation theorem. 
\subsection{Faedo-Galerkin approximation} Let us define $\mathrm{A}^n=\Pi_n\A$, $\B^n(\u_n)=\Pi_n\B(\u_n)$, $\mathcal{C}^n(\u_n)=\Pi_n\mathcal{C}(\u_n)$, $\boldsymbol{\sigma}^n(\cdot,\u_n)=\Pi_n\boldsymbol{\sigma}(\cdot,\u_n)$  and $\W^n=\Pi_n\W$, where $\Pi_n$ is the projection operator from $\mathbb{U}^{\prime}$ to $\H_n$ defined in Subsection \ref{C_O}. For each $n\in\N$, we search for approximate solutions of the form 
\begin{align}
	\u_n(x,t)=\sum_{k=1}^n g^n_k(t)\boldsymbol{e}_k(x), \;\; \boldsymbol{e}_k\in\H_n,
\end{align} 
where $\{\boldsymbol{e}_k\}_{k\in\N}$ is defined in \eqref{L3} and  the coefficients $g^n_1,\ldots,g^n_n$ are solutions of the following $n$ stochastic ordinary differential equations:
\begin{equation}\label{eqn-finite-dim-system}
\left\{
\begin{aligned}
\d\left(\u_n(t),\boldsymbol{e}_j\right)&=-\langle\mu \A^n\u_n(t)+\B^n(\u_n(t))+ \alpha\u_n(t)+\beta\mathcal{C}^n(\u_n(t)) - \mathbf{F}^n(t),\boldsymbol{e}_j\rangle\d t
\\ & \quad +\left(\boldsymbol{\sigma}^n(t,\u_n(t))\d \W^n(t),\boldsymbol{e}_j\right),\\
(\u_n(0),\boldsymbol{e}_j)&=(\u_0^n,\boldsymbol{e}_j),
\end{aligned}
\right.
\end{equation}
for all $0\leq t \leq T$ and $j=1,\ldots,n$, where $\u_0^n=\Pi_n\u_0$ and $\mathbf{F}^n=\Pi_n\mathbf{F}$. Since $\B^n(\cdot)$ and $\mathcal{C}^n(\cdot)$ are  locally Lipschitz (see \eqref{lip} and \eqref{213}), and  $\boldsymbol{\sigma}^n(\cdot,\u_n)$  is globally Lipschitz, the system \eqref{eqn-finite-dim-system} has a unique $\H_n$-valued local strong solution $\u_n(\cdot)$ and $\u_n\in
\mathrm{L}^2(\Omega;\C([0,T^*_n];\H_n))$ with
$\{\mathscr{F}_t\}_{t\geq 0}$-adapted continuous sample paths {(see  \cite{chow})}.  Now we discuss the a-priori energy estimates satisfied by the solution to the system \eqref{eqn-finite-dim-system}, which also implies  that $T^*_n$ can be extended to $T$.

\begin{proposition}[Energy estimates]\label{prop1}
Let $p\geq 1$. Under Hypothesis \ref{hyp}, let $\u_n(\cdot)$ be the unique solution of the system of stochastic	ODEs \eqref{eqn-finite-dim-system} with $\u_0\in \H$ and $\mathbf{F}\in \mathrm{L}^{2}(0,T;\V^{\prime})$. Then, we have 
	\begin{align}\label{ener1}
		& \sup_{n\geq 1} \E \left[ \sup_{t\in[0,T]} \|\u^n(t)\|^{2p}_{\H}+ p\int_{0}^{T}\|\u^n(s)\|_{\H}^{2p-2}\left(\mu\|\nabla\u^n(s)\|^2_{\H}+\alpha\|\u^n(s)\|^2_{\H}+2\beta\|\u^n(s)\|^{r+1}_{\wi\L^{r+1}}\right)\d s\right] \nonumber\\&
		\leq C\left( \|\u_0\|_{\H},  \|\mathbf{F}\|_{\mathrm{L}^{2}(0,T;\V^{\prime})}, \mu,\alpha, p , K ,T\right).
	\end{align}
\end{proposition}

\begin{proof}
		\noindent\textbf{Step (1):} Let us first define a sequence of stopping times $\tau_N$ by
	\begin{align}\label{stopm}
		\tau_N:=\inf_{t\geq 0}\left\{t:\|\u^n(t)\|_{\H}\geq N\right\},
	\end{align}
	for $N\in\mathbb{N}$. Applying the finite-dimensional It\^{o} formula to the process
	$\|\u^n(\cdot)\|_{\H}^2$, we obtain for all $t\in[0,T]$, $\mathbb{P}$-a.s.,
	\begin{align}\label{4.99}
		&	\|\u^n(\t)\|_{\H}^2
		  \nonumber\\ &=
		\|\u^n(0)\|_{\H}^2-2\int_0^{\t}\langle\mu \A\u^n(s)+\B^n(\u^n(s))+\alpha\u^n(s)+\beta\mathcal{C}^n(\u^n(s)) - \mathbf{F}^n(s),\u^n(s)\rangle\d s
		\nonumber\\&\quad+\int_0^{\t}\|\upsigma^n(s,\u^n(s))\|^2_{\mathcal{L}_{\Q}}\d
		s 
		+2\int_0^{\t}\left(\upsigma^n(s,\u^n(s))\d\W^n(s),\u^n(s)\right).
	\end{align}
	Note that $\langle\B^n(\u^n),\u^n\rangle=\langle\B(\u^n),\u^n\rangle=0$. 
	Taking expectation in \eqref{4.99}, and using the Cauchy-Schwarz and Young inequalities, and Hypothesis \ref{hyp}  (H.2), we find
	\begin{align}\label{4.13}
		&\E\left[\|\u^n(\t)\|_{\H}^2 + \mu \int_0^{\t} \|\nabla\u^n(s)\|_{\H}^2\d s + \alpha \int_0^{\t} \|\u^n(s)\|_{\H}^2\d s+2\beta\int_0^{\t} \|\u^n(s)\|_{\widetilde{\L}^{r+1}}^{r+1}\d s\right]
        \nonumber\\&\leq
		\E\left[	\|\u^n(0)\|_{\H}^2\right]+\E\left[\int_0^{\t}\|\upsigma^n(s,\u^n(s))\|^2_{\mathcal{L}_{\Q}}\d s \right]
         + \frac{1}{\min\{\mu,\alpha\}}\E\left[\int_0^{\t}\|\mathbf{F}(s)\|^2_{\V^{\prime}} \d
		s \right]
        \nonumber\\&\leq 	\|\u_0\|_{\H}^2 +K\E\left[\int_0^{\t}\left[1+\|\u^n(s)\|^2_{\H}\right]\d
		s \right]  + \frac{1}{\min\{\mu,\alpha\}}\int_0^{T}\|\mathbf{F}(s)\|^2_{\V^{\prime}} \d s
		\nonumber\\&\leq 	\|\u_0\|_{\H}^2 + \frac{1}{\min\{\mu,\alpha\}}\int_0^{T}\|\mathbf{F}(s)\|^2_{\V^{\prime}} \d s + KT + K\int_0^{t}\E\left[\|\u^n(s\land \tau_N)\|^2_{\H} \right]\d s
        \nonumber\\&\leq 	\|\u_0\|_{\H}^2 + \frac{1}{\min\{\mu,\alpha\}}\int_0^{T}\|\mathbf{F}(s)\|^2_{\V^{\prime}} \d s + KT 
          + K\int_0^{t}\E\left[\|\u^n(s\land \tau_N)\|_{\H}^2 + \mu \int_0^{s\land \tau_N} \|\nabla\u^n(\tau)\|_{\H}^2\d \tau 
          \right. 
          \nonumber\\ & \quad \left. + \alpha \int_0^{s\land \tau_N} \|\u^n(\tau)\|_{\H}^2\d \tau + 2\beta\int_0^{s\land \tau_N} \|\u^n(\tau)\|_{\widetilde{\L}^{r+1}}^{r+1}\d \tau \right]\d s,
	\end{align}
	where we have used the fact that final term in the right hand side of the equality \eqref{4.99} is a martingale with zero expectation. Applying Gr\"onwall's inequality in (\ref{4.13}), we get 
	\begin{align}\label{2.7}
		& \E\left[\|\u^n(\t)\|_{\H}^2 + \mu \int_0^{\t} \|\nabla\u^n(s)\|_{\H}^2\d s + \alpha \int_0^{\t} \|\u^n(s)\|_{\H}^2\d s+2\beta\int_0^{\t} \|\u^n(s)\|_{\widetilde{\L}^{r+1}}^{r+1}\d s\right]
        \nonumber\\ & \leq
		\left( \|\u_0\|_{\H}^2  +  \frac{1}{\min\{\mu,\alpha\}}\int_0^{T}\|\mathbf{F}(s)\|^2_{\V^{\prime}} \d s +KT\right)e^{KT},
	\end{align}
	for all $0\leq t \leq T$.  Note that for the indicator function $\chi$, 
	$$\E\left[{\chi}_{\{\tau_N^n<t\}}\right]=\mathbb{P}\Big\{\omega\in\Omega:\tau_N^n(\omega)<t\Big\},$$
	and using \eqref{stopm}, we obtain 
	\begin{align}\label{sq11}
		\E\left[\|\u^n(\t)\|_{\H}^2\right]&=
		\E\left[\|\u^n(\tau_N^n)\|_{\H}^2{\chi}_{\{\tau_N^n<t\}}\right]+\E\left[\|\u^n(t)\|_{\H}^2
		{\chi}_{\{\tau_N^n\geq t\}}\right]\nonumber\\&\geq
		\E\left[\|\u^n(\tau_N^n)\|_{\H}^2{\chi}_{\{\tau_N^n<t\}}\right]\geq
		N^2\mathbb{P}\Big\{\omega\in\Omega:\tau_N^n<t\Big\}.
	\end{align}
	Using the energy estimate (\ref{2.7}), we find 
	\begin{align}\label{sq12}
		\mathbb{P}\Big\{\omega\in\Omega:\tau_N<t\Big\}&\leq
		\frac{1}{N^2}\E\left[\|\u^n(\t)\|_{\H}^2\right]
        \nonumber\\ & \leq
		\frac{1}{N^2}
		\left( \|\u_0\|_{\H}^2  +  \frac{1}{\min\{\mu,\alpha\}}\int_0^{T}\|\mathbf{F}(s)\|^2_{\V^{\prime}} \d s +KT\right)e^{KT}.
	\end{align}
	Hence, we have
	\begin{align}\label{sq13}
		\lim_{N\to\infty}\mathbb{P}\Big\{\omega\in\Omega:\tau_N<t\Big\}=0, \ \textrm{
			for all }\ 0\leq t \leq T,
	\end{align}
	and $\t\to t$ as $N\to\infty$. 
	Taking limit $N\to\infty$ in
	(\ref{2.7}) and using the \emph{monotone convergence theorem}, we get 
	\begin{align}\label{4.16}
		& \E\left[\|\u^n(\t)\|_{\H}^2 + \mu \int_0^{\t} \|\nabla\u^n(s)\|_{\H}^2\d s + \alpha \int_0^{\t} \|\u^n(s)\|_{\H}^2\d s+2\beta\int_0^{\t} \|\u^n(s)\|_{\widetilde{\L}^{r+1}}^{r+1}\d s\right]
        \nonumber\\ & \leq
		\left( \|\u_0\|_{\H}^2  +  \frac{1}{\min\{\mu,\alpha\}}\int_0^{T}\|\mathbf{F}(s)\|^2_{\V^{\prime}} \d s +KT\right)e^{KT},
	\end{align}
	for $0\leq t\leq T$. Substituting (\ref{4.16}) in (\ref{4.13}), we arrive at
	\begin{align}\label{4.16a}
		&	\E\left[\|\u^n(t)\|_{\H}^2 + \mu \int_0^{t}\|\nabla\u^n(s)\|_{\H}^2\d s + \alpha \int_0^{t}\|\u^n(s)\|_{\H}^2\d s +2\beta\int_0^t\|\u^n(s)\|_{\widetilde{\L}^{r+1}}^{r+1}\d s\right]
		\nonumber\\ & \leq
		\left( \|\u_0\|_{\H}^2  +  \frac{1}{\min\{\mu,\alpha\}}\int_0^{T}\|\mathbf{F}(s)\|^2_{\V^{\prime}} \d s +KT\right)e^{KT},
	\end{align}
	for all $0\leq t \leq T$. 
	\vskip 0.2cm
	\noindent\textbf{Step (2):} \textit{Proof of \eqref{ener1}}.	An application of thr finite-dimensional It\^o formula to the process $\|\u^n(\cdot)\|^{2p}_{\H}$, and the Cauchy-Schwarz and Young inequalities yields
	\begin{align}\label{S4}
		&\|\u^n(\t)\|^{2p}_{\H}+ 2p\int_{0}^{\t}\|\u^n(s)\|_{\H}^{2p-2}\left(\mu\|\nabla\u^n(s)\|^2_{\H}+\alpha\|\u^n(s)\|^2_{\H}+\beta\|\u^n(s)\|^{r+1}_{\wi\L^{r+1}}\right)\d s
        \nonumber\\&
		=\underbrace{ \|\Pi_n\u_0\|^{2p}_{\H}}_{\leq \|\u_0\|^{2p}_{\H}} +  p \int_{0}^{\t}\|\u^n( s)\|_{\H}^{2p-2}\|\upsigma^n(s,\u^n(s))\|^2_{\mathcal{L}_{\Q}} \d s
		\nonumber\\ & \quad 
        +  2 p\int_{0}^{\t}\|\u^n( s)\|_{\H}^{2p-2}\left(\upsigma^n(s,\u^n(s))\d\W^n(s) , \u^n( s)\right) 
        +  2 p\int_{0}^{\t}\|\u^n( s)\|_{\H}^{2p-2}\left<\mathbf{F}(s) , \u^n( s)\right> \d s
        \nonumber\\&\quad+ 2 p(p-1) \int_{0}^{\t}\|\u^n( s)\|^{2p-4}_{\H} \|(\upsigma^n(s,\u^n(s)))^{*}\u^n( s)\|^2_{\H}\d s
        \nonumber\\&
		\leq \|\u_0\|^{2p}_{\H} +  p \int_{0}^{\t}\|\u^n( s)\|_{\H}^{2p-2}\|\upsigma^n(s,\u^n(s))\|^2_{\mathcal{L}_{\Q}} \d s
		\nonumber\\ & \quad 
        +  2 p\int_{0}^{\t}\|\u^n( s)\|_{\H}^{2p-2}\left(\upsigma^n(s,\u^n(s))\d\W^n(s) , \u^n( s)\right) 
        \nonumber\\ &\quad  +  2 p\int_{0}^{\t}\|\u^n( s)\|_{\H}^{2p-2}\|\mathbf{F}(s)\|_{\V^\prime} \|\u^n( s)\|_{\V} \d s
        \nonumber\\&\quad+ 2 p(p-1) \int_{0}^{\t}\|\u^n( s)\|^{2p-4}_{\H} \|(\upsigma^n(s,\u^n(s)))^{*}\u^n( s)\|^2_{\H}\d s
        \nonumber\\&
		\leq \|\u_0\|^{2p}_{\H} +  p \int_{0}^{\t}\|\u^n( s)\|_{\H}^{2p-2}\|\upsigma^n(s,\u^n(s))\|^2_{\mathcal{L}_{\Q}} \d s
		\nonumber\\ & \quad 
        +  2 p\int_{0}^{\t}\|\u^n( s)\|_{\H}^{2p-2}\left(\upsigma^n(s,\u^n(s))\d\W^n(s) , \u^n( s)\right) 
         \nonumber\\ & \quad +   p\int_{0}^{\t}\|\u^n( s)\|_{\H}^{2p-2}\left\{ \frac{1}{\min\{\mu,\alpha\}}\|\mathbf{F}(s)\|^2_{\V^\prime} +  \min\{\mu,\alpha\}\|\u^n( s)\|^2_{\V} \right\}\d s
        \nonumber\\&\quad+ 2 p(p-1) \int_{0}^{\t}\|\u^n( s)\|^{2p-4}_{\H} \|(\upsigma^n(s,\u^n(s)))^{*}\u^n( s)\|^2_{\H}\d s.
	\end{align}
	 Taking supremum from $0$
	to $\T$ before taking expectation in \eqref{S4}, we obtain
	\begin{align}\label{S44}
		&\E \left[ \sup_{t\in[0,\T]} \|\u^n(t)\|^{2p}_{\H}+ p\int_{0}^{\T}\|\u^n(s)\|_{\H}^{2p-2}\left(\mu\|\nabla\u^n(s)\|^2_{\H}+\alpha\|\u^n(s)\|^2_{\H}+2\beta\|\u^n(s)\|^{r+1}_{\wi\L^{r+1}}\right)\d s\right]\nonumber\\&
		\leq  \|\u_0\|^{2p}_{\H}  +  p \E \left[\int_{0}^{\T}\|\u^n( s)\|_{\H}^{2p-2}\|\upsigma^n(s,\u^n(s))\|^2_{\mathcal{L}_{\Q}} \d s\right]
		 \nonumber\\&\quad+ 2 p(p-1) \E \left[\int_{0}^{\T}\|\u^n( s)\|^{2p-4}_{\H} \|(\upsigma^n(s,\u^n(s)))^{*}\u^n( s)\|^2_{\H}\d s\right]
		 \nonumber\\ & \quad +  2 p \E\left[ \sup_{t\in[0,\T]} \left|\int_{0}^{t}\|\u^n( s)\|_{\H}^{2p-2}\left(\upsigma^n(s,\u^n(s))\d\W^n(s) , \u^n( s)\right)\right| \right]
         \nonumber\\ & \quad + \frac{p}{\min\{\mu,\alpha\}} \E \left[ \int_{0}^{\T}\|\u^n( s)\|_{\H}^{2p-2} \|\mathbf{F}(s)\|^2_{\V^\prime} \d s \right]
		 \nonumber\\ & =    \|\u_0\|^{2p}_{\H} +\sum_{i=1}^{4}Q_{i}(t,n).
	\end{align}
	Let us now estimate each term of \eqref{S44} separately. Using Hypothesis \ref{hyp} (H.2), we get 
	\begin{align}\label{S5}
		| Q_1(t,n) + Q_2(t,n) | \leq C \E \left[ \int_{0}^{\T} \left\{ 1+ \|\u^n( s)\|_{\H}^{2p}  \right\}\d s \right].
	\end{align}
	Using Burkholder-Davis-Gundy (in short BDG) inequality, (\cite[Theorem 1]{BD} and \cite[Theorem 1.1]{DLB}), Hypothesis \ref{hyp} (H.2), and H\"{o}lder's  and Young's inequalities, we have
	\begin{align}\label{S55}
		\left|Q_3(s,n)\right|  
		& \leq C \mathbb{E}\left[\left(\int_{0}^{\T}\|\u^n( s)\|^{4p-4}_{\H} \|(\upsigma^n(s,\u^n(s)))^{*}\u^n( s)\|^2_{\H}\d s\right)^{\frac{1}{2}}\right]
		\nonumber\\ & \leq C \mathbb{E}\left[\left(\int_{0}^{\T}\|\u^n( s)\|^{4p-2}_{\H} \|(\upsigma^n(s,\u^n(s)))^{*}\|^2_{\mathcal{L}_{\Q}}\d s\right)^{\frac{1}{2}}\right] 
		\nonumber\\ & \leq C \mathbb{E}\left[\left(\int_{0}^{\T}\|\u^n( s)\|^{4p-2}_{\H} \left\{1+ \|\u^n( s)\|^{2}_{\H}\right\}\d s\right)^{\frac{1}{2}}\right] 
		\nonumber\\ & \leq C \mathbb{E}\left[\sup_{t\in[0,\T]}\|\u^n(t)\|^{p}_{\H} \left(\int_{0}^{\T}\|\u^n( s)\|^{2p-2}_{\H} \left\{1+ \|\u^n( s)\|^{2}_{\H}\right\}\d s\right)^{\frac{1}{2}}\right]
		\nonumber\\ & \leq  \mathbb{E}\left[\frac14\sup_{t\in[0,\T]}\|\u^n(t)\|^{2p}_{\H} \right] + C \mathbb{E}\left[ \int_{0}^{\T} \left\{1+ \|\u^n( s)\|^{2p}_{\H}\right\}\d s \right].
	\end{align}
    Using H\"older's inequality, we find 
    \begin{align}\label{S555}
      |Q_4(t,n)| & \leq \frac{p}{\min\{\mu,\alpha\}} \E \left[ \sup_{t\in[0,\T]}  \|\u^n( s)\|_{\H}^{2p-2}   \int_{0}^{T} \|\mathbf{F}(s)\|^2_{\V^\prime} \d s \right]
      \nonumber\\ & \leq  \mathbb{E}\left[\frac14\sup_{t\in[0,\T]}\|\u^n(t)\|^{2p}_{\H} \right] + C\|\mathbf{F}\|^{2p}_{\mathrm{L}^2(0,T;\V^{\prime})}.
    \end{align}
	By combining \eqref{S44}-\eqref{S555}, we reach at
	\begin{align}\label{S66}
		&\E \left[ \frac12 \sup_{t\in[0,\T]} \|\u^n(t)\|^{2p}_{\H}+ p\int_{0}^{\T}   \|\u^n(s)\|_{\H}^{2p-2}\left(\mu\|\nabla\u^n(s)\|^2_{\H}+\alpha\|\u^n(s)\|^2_{\H}+2\beta\|\u^n(s)\|^{r+1}_{\wi\L^{r+1}}\right)\d s\right]\nonumber\\&
		\leq  \|\u_0\|^{2p}_{\H} + C\|\mathbf{F}\|^{2p}_{\mathrm{L}^2(0,T;\V^{\prime})} + CT + C \mathbb{E}\left[ \int_{0}^{\T}  \|\u^n( s)\|^{2p}_{\H} \d s \right].
	\end{align}
Applying Gr\"onwall's inequality in \eqref{S66}, we obtain 
\begin{align}\label{4.25}
	& \E \left[ \sup_{t\in[0,\T]} \|\u^n(t)\|^{2p}_{\H}+ p\int_{0}^{\T}\|\u^n(s)\|_{\H}^{2p-2}\left(\mu\|\nabla\u^n(s)\|^2_{\H}+\alpha\|\u^n(s)\|^2_{\H}+2\beta\|\u^n(s)\|^{r+1}_{\wi\L^{r+1}}\right)\d s\right]
    \nonumber\\ & \leq\left(2  \|\u_0\|_{\H}^{2p} + C\|\mathbf{F}\|^{2p}_{\mathrm{L}^2(0,T;\V^{\prime})} +CT\right)e^{CT}.
\end{align}
Now passing $N\to\infty$ in \eqref{4.25} and using the monotone convergence theorem, we finally obtain \eqref{ener1}. 
\end{proof}


\begin{remark}
    Applying the finite-dimensional It\^{o} formula to the process
	$\|\u_n(\cdot)\|_{\H}^2$, we obtain 
	\begin{align}\label{4.9}
	&	\|\u_n(t)\|_{\H}^2
        \nonumber\\&=
		\|\u_0^n\|_{\H}^2-2\mu\int_0^{t}\|\nabla\u_n(s)\|_{\H}^2\d s - 2\alpha \int_0^{t}\|\u_n(s)\|_{\H}^2\d s -2\beta\int_0^{t}\|\u_n(s)\|_{\wi\L^{r+1}}^{r+1}\d s \nonumber\\&\quad + 2 \int_0^t\left\langle \mathbf{F}(s),\u_n(s) \right\rangle \d s
        +2\int_0^{t}\left(\boldsymbol{\sigma}^n(s,\u_n(s))\d\W^n(s),\u_n(s)\right)
		+\int_0^{t}\|\boldsymbol{\sigma}^n(s,\u_n(s))\|^2_{\mathcal{L}_{\Q}}\d
		s,
	\end{align}
	where we have used the fact that $\langle\B^n(\u_n),\u_n\rangle=\langle\B(\u_n),\u_n\rangle=0$. 
    
	In \eqref{4.9},  we estimate $\int_0^t\left\langle \mathbf{F}^n(s),\u_n(s) \right\rangle \d s$ using H\"older's and Young's inequalities, take supremum over $0$ to $T$,   raise to the power $p=\frac{q}{2}$ and then take expectation to obtain 
	\begin{align}\label{4p34}
		&  \mu^{\frac{q}{2}}\E\left[\left(\int_0^{T}\|\nabla\u_n(t)\|_{\H}^2\d t\right)^{\frac{q}{2}}\right] 
        +2^{\frac{q}{2}}\beta^{\frac{q}{2}}\E\left[\left(\int_0^{T}\|\u_n(t)\|_{\wi\L^{r+1}}^{r+1}\d t\right)^{\frac{q}{2}}\right]\nonumber\\&\leq C_q\Bigg\{ \|\u_0\|_{\H}^{q} + \left( \int_0^T\|\mathbf{F}(s)\|^2_{\V^{\prime}}\d s \right)^{\frac{q}{2}}
		+\E\left[\left(\int_0^{T}\|\boldsymbol{\sigma}^n(t,\u_n(t))\|^2_{\mathcal{L}_{\Q}}\d
		t \right)^{\frac{q}{2}}\right] \nonumber\\&\quad+\E\left[\sup_{t\in[0,T]}\left|\int_0^{t}\left(\boldsymbol{\sigma}^n(s,\u_n(s))\d\W^n(s),\u_n(s)\right)\right|^{\frac{q}{2}}\right]\Bigg\}=:C_q\sum_{i=1}^4 J_i. 
	\end{align}
	Using Hypothesis \ref{hyp} (H.2), we estimate $J_3$ as 
	\begin{align}\label{4p35}
		\E\left[\left(\int_0^{T}\|\boldsymbol{\sigma}^n(t,\u_n(t))\|^2_{\mathcal{L}_{\Q}}\d
		t \right)^{\frac{q}{2}}\right]\leq 2^{\frac{q}{2}-1} T^{\frac{q}{2}}K^{\frac{q}{2}}\left\{1+\E\left[\sup_{t\in[0,T]}\|\u_n(t)\|_{\H}^q\right]\right\}.
	\end{align}
	Applying BDG, H\"{o}lder's  and Young's inequalities we estimate $J_4$ as 
	\begin{align}\label{4p36}
		&\E\left[\sup_{t\in[0,T]}\left|\int_0^{t}\left(\boldsymbol{\sigma}^n(s,\u_n(s))\d\W(s),\u_n(s)\right)\right|^{\frac{q}{2}}\right]\nonumber\\&\leq C_q \E\left[\int_0^{T}\|\boldsymbol{\sigma}^n(t,\u_n(t))\|_{\mathcal{L}_{\Q}}^2\|\u_n(t)\|_{\H}^2\d t\right]^{\frac{q}{4}}\nonumber\\&\leq \E\left[\sup_{t\in[0,T]}\|\u_n(t)\|_{\H}^{q}\right]+C_q\E\left[\left(\int_0^T\|\boldsymbol{\sigma}^n(t,\u_n(t))\|_{\mathcal{L}_{\Q}}^2\right)^{\frac{q}{2}}\right]\nonumber\\&\leq C_{q,T}K^{\frac{q}{2}}\left\{1+\E\left[\sup_{t\in[0,T]}\|\u_n(t)\|_{\H}^q\right]\right\}.
	\end{align}
	Substituting \eqref{4p35}-\eqref{4p36} in \eqref{4p34}, we deduce 
	\begin{align*}
		&   \sup_{n\geq 1}\E\left[\left(\int_0^{T}\|\nabla\u_n(t)\|_{\H}^2\d t\right)^{\frac{q}{2}}\right]
        +\sup_{n\geq 1}\E\left[\left(\int_0^{T}\|\u_n(t)\|_{\wi\L^{r+1}}^{r+1}\d t\right)^{\frac{q}{2}}\right]\nonumber\\&\leq C_{\mu,\alpha, \beta,q,T,K}\left\{ \|\u_0\|_{\H}^{q} + \left( \int_0^T\|\mathbf{F}(s)\|^2_{\V^{\prime}}\d s \right)^{\frac{q}{2}} + 1 + \E\left[\sup_{t\in[0,T]}\|\u_n(t)\|_{\H}^q\right]\right\}\nonumber\\&\leq C\left(\|\u_0\|_{\H}, \|\mathbf{F}\|_{\mathrm{L}^2(0,T;\V^{\prime})} ,\mu, \alpha,\beta,q,T,K\right), 
	\end{align*}
so that we get 
\begin{align}\label{4p32}
	&   \sup_{n\geq 1}\E\left[\left(\int_0^{T}\|\nabla\u_n(t)\|_{\H}^2\d t\right)^{p}\right]
    +\sup_{n\geq 1}\E\left[\left(\int_0^{T}\|\u_n(t)\|_{\wi\L^{r+1}}^{r+1}\d t\right)^{p}\right]\nonumber\\&\leq C\left(\|\u_0\|_{\H}, \|\mathbf{F}\|_{\mathrm{L}^{2}(0,T;\V^{\prime})} ,\mu, \alpha,\beta,p,T,K\right),
\end{align}
for some $p>1$. 
\end{remark}

\subsection{Tightness} In order to prove tightness, we first consider the space 
\begin{align}\label{4.32}
	\mathscr{Y}:=\C([0,T]; \mathbb{U}^{\prime}) \cap\mathrm{L}_w^2(0,T;\V)\cap\mathrm{L}_w^{r+1}(0,T;\wi\L^{r+1})\cap \mathrm{L}^2(0,T;\H_{\mathrm{loc}})\cap\C([0,T];\H_w).
\end{align}
For each $n\in\N$, the solution $\u_n(\cdot)$ of the Faedo-Galerkin approximation defines a measure $\mathscr{L}(\u_n)$ on  $(\mathscr{Y},\mathcal{T})$, where $\mathcal{T}=\mathcal{T}_1\vee\mathcal{T}_2\vee\mathcal{T}_3\vee\mathcal{T}_4\vee\mathcal{T}_5$. Using Corollary \ref{cor3.12}, our aim is to show that the set of measures $\{\mathscr{L}(\u_n): n \in\N\}$ is tight on $(\mathscr{Y},\mathcal{T})$. The estimate  \eqref{ener1} in Proposition \ref{prop1} and \eqref{4p32} play a crucial role. Nevertheless, in order to prove tightness, it is sufficient to use inequality \eqref{ener1}. 

\begin{lemma}\label{lem4.2}
	The set of measures $\{\mathscr{L}(\u_n):n\in\N\}$ is tight on $(\mathscr{Y},\mathcal{T})$. 
\end{lemma}
\begin{proof}
 We apply Corollary \ref{cor3.12} to obtain the proof. The estimate \eqref{ener1} (for $p=2$) implies the conditions (a), (b) and (c) of Corollary \ref{cor3.12}. Therefore, it is sufficient to prove that the sequence $\{\u_n\}_{n\in\N}$ satisfies the Aldous condition $\mathbf{[A]}$ in the space $\mathbb{U}^{\prime}$.
	  We use Lemma \ref{lem3.11} to obtain the required result. Let $\{\tau_n\}_{n\in\N} $ be a sequence of stopping times such that $0 \leq\tau_n\leq T$. From \eqref{eqn-finite-dim-system}, we have 
	\begin{equation}\label{434}
		\begin{aligned}
			\u_n(t)&=\u_0^n-\mu\int_0^t \A^n\u_n(s)\d s-\int_0^t\B^n(\u_n(s))\d s - \alpha\int_0^t \u_n(s)\d s -\beta\int_0^t\mathcal{C}^n(\u_n(s))\d s\\&\quad + \int_0^t \mathbf{F}^n(s)\d s +\int_0^t\boldsymbol{\sigma}^n(s,\u_n(s))\d\W^n(s) \nonumber\\&=:J_1^n+\sum_{i=2}^7 J_i^n(t),\ t\in[0,T].
		\end{aligned}
	\end{equation}
	For $\theta>0$, we verify that each term $J_i^n, i=1,\ldots, 6$ satisfies the condition \eqref{37} in Lemma \ref{lem3.11}. For $J_1^n$, it is clear that the condition \eqref{37} is satisfied. For $J_2^n$, we use the continuous embedding $\V^{\prime}\hookrightarrow \mathbb{U}^{\prime}$, H\"older's inequality, \eqref{ener1} to estimate it as 
	\begin{align}
	&\E\left[\|J_2^n(\tau_n+\theta)-J_2^n(\tau_n)\|_{\mathbb{U}'}\right] \leq C \E\left[\|J_2^n(\tau_n+\theta)-J_2^n(\tau_n)\|_{\V'}\right]=C\mu\E\left[\left\|\int_{\tau_n}^{\tau_n+\theta}\A^n\u_n(s)\d s\right\|_{\V'}\right]\nonumber\\&\leq C\mu\E\left[\int_{\tau_n}^{\tau_n+\theta}\|\u_n(s)\|_{\V}\d s\right]\leq C\mu\theta^{1/2}\E\left[\left(\int_{\tau_n}^{\tau_n+\theta}\|\u_n(s)\|_{\V}^2\d s\right)^{1/2}\right]\nonumber\\&\leq C\left(\|\u_0\|_{\H}, \|\mathbf{F}\|_{\mathrm{L}^{2}(0,T;\V^{\prime})} ,\mu,\alpha,\beta,p,T,K\right)\theta^{1/2} =: c_2 \theta^{1/2},
	\end{align}
	so that $J_2^n$ satisfies the condition \eqref{37} with $\xi=1$ and $\eta=\frac{1}{2}$. Next, let $s> \frac{d}{2}+1$ and recall that $\B:\H\times\H\to \V_s$ is bilinear and continuous (Remark \ref{rem-trilinear-ext}), and the embedding $\V'_s\hookrightarrow \mathbb{U}'$ is continuous. We consider $J_3^n$ and estimate it using \eqref{eqn-trilinear-ext} and \eqref{ener1}  as 
	\begin{align}
		& \E\left[\|J_3^n(\tau_n+\theta)-J_3^n(\tau_n)\|_{\mathbb{U}'}\right] \leq C	\E\left[\|J_3^n(\tau_n+\theta)-J_3^n(\tau_n)\|_{\V'_s}\right]
        \nonumber\\ & = C\E\left[\left\|\int_{\tau_n}^{\tau_n+\theta}\B^n(\u_n(s))\d s\right\|_{\V'_s}\right] 
        \leq  C\E\left[\int_{\tau_n}^{\tau_n+\theta}\left\|\B(\u_n(s))\right\|_{\V'_s}\d s\right] 
        \nonumber\\& \leq C\E\left[\int_{\tau_n}^{\tau_n+\theta}\|\u_n(s)\|_{\H}^2\d s\right] \leq C\E\left[\sup_{s\in[0,T]}\|\u_n(s)\|_{\H}^2\right] \theta
        \nonumber\\&\leq C\left(\|\u_0\|_{\H}, \|\mathbf{F}\|_{\mathrm{L}^{2}(0,T;\V^{\prime})} ,\mu,\alpha,p,T,K\right)\theta  =: c_3 \theta,
	\end{align}
	which implies that the condition \eqref{37} is satisfied  with $\xi=1$ and $\eta=1$. Similarly, $J_4^n$ is  estimated using the continuous embedding $\H\hookrightarrow \mathbb{U}'$ and \eqref{ener1} as
\begin{align}
		& \E\left[\|J_4^n(\tau_n+\theta)-J_4^n(\tau_n)\|_{\mathbb{U}'}\right] \leq C	\E\left[\|J_4^n(\tau_n+\theta)-J_4^n(\tau_n)\|_{\H}\right]
         = C\alpha \E\left[\left\|\int_{\tau_n}^{\tau_n+\theta}\u_n(s)\d s\right\|_{\H}\right] 
        \nonumber\\ & \leq  C \alpha \E\left[\int_{\tau_n}^{\tau_n+\theta}\left\|\u_n(s)\right\|_{\H}\d s\right] 
     \leq C\alpha\E\left[\sup_{s\in[0,T]}\|\u_n(s)\|_{\H}\right] \theta
        \nonumber\\&\leq C\left(\|\u_0\|_{\H}, \|\mathbf{F}\|_{\mathrm{L}^{2}(0,T;\V^{\prime})} ,\mu,\alpha,\beta,p,T,K\right)\theta =: c_4 \theta,
	\end{align}
which implies that the condition \eqref{37} is satisfied  with $\xi=1$ and $\eta=1$. Now, we estimate $J_5^n$ using the continuous embedding $\L^{\frac{r+1}{r}}\hookrightarrow \mathbb{U}'$ and \eqref{ener1} as
	\begin{align}
		&\E\left[\|J_5^n(\tau_n+\theta)-J_5^n(\tau_n)\|_{\mathbb{U}^\prime} \right]
		= \E\left[\left\|\int_{\tau_n}^{\tau_n+\theta}\mathcal{C}^n(\u_n(s))\d s\right\|_{\mathbb{U}^\prime}\right] \leq  \E\left[\int_{\tau_n}^{\tau_n+\theta}\left\| \mathcal{C}(\u_n(s)) \right\|_{\mathbb{U}^\prime} \d s \right]
        \nonumber\\ & \leq  \E\left[\int_{\tau_n}^{\tau_n+\theta}\left\| \mathcal{C}(\u_n(s)) \right\|_{\L^{\frac{r+1}{r}}} \d s \right]
		\leq C\E\left[\int_{\tau_n}^{\tau_n+\theta}\|\u_n(s)\|_{\wi\L^{r+1}}^r \d s \right]
	\nonumber\\ & 	\leq C\theta^{\frac{1}{r+1}}\E\left[\left(\int_{\tau_n}^{\tau_n+\theta}\|\u_n(s)\|_{\wi\L^{r+1}}^{r+1}\d s\right)^{\frac{r}{r+1}}\right] \leq C\theta^{\frac{1}{r+1}}\left\{\E\left[\int_{0}^{T}\|\u_n(t)\|_{\wi\L^{r+1}}^{r+1}\d t\right]\right\}^{\frac{r}{r+1}} 
    \nonumber\\ &  \leq C\left(\|\u_0\|_{\H}, \|\mathbf{F}\|_{\mathrm{L}^{2}(0,T;\V^{\prime})} ,\mu,\alpha,\beta,p,T,K\right)\theta^{\frac{1}{r+1}} =: c_5 \theta^{\frac{1}{r+1}} ,
	\end{align}
	and thus the  condition \eqref{37} holds with $\xi=1$ and $\eta=\frac{1}{r+1}$. 
    Similarly, $J_6^n$ can be estimated using the continuous embedding $\V'\hookrightarrow \mathbb{U}'$ and \eqref{ener1} as
\begin{align}
		& \E\left[\|J_6^n(\tau_n+\theta)-J_6^n(\tau_n)\|_{\mathbb{U}'}\right] \leq C	\E\left[\|J_6^n(\tau_n+\theta)-J_6^n(\tau_n)\|_{\V'}\right]
         = C \E\left[\left\|\int_{\tau_n}^{\tau_n+\theta}\mathbf{F}^n(s)\d s\right\|_{\V'}\right] 
        \nonumber\\ & \leq  C  \E\left[\int_{\tau_n}^{\tau_n+\theta}\left\|\mathbf{F}^n(s)\right\|_{\V'}\d s\right] 
     \leq C \|\mathbf{F}\|_{\mathrm{L}^2(0,T;\V^{\prime})} \theta^{\frac12} =: c_6 \theta,
	\end{align}
which implies that the condition \eqref{37} is satisfied  with $\xi=1$ and $\eta=\frac12$.  Using It\^o's isometry,  Hypothesis \ref{hyp} (H.2) and \eqref{ener1}, we estimate $J_7^n$ as 
	\begin{align}
	&	\E\left[\|J_7^n(\tau_n+\theta)-J_7^n(\tau_n)\|_{\mathbb{U}'}^2\right]\leq C\E\left[\|J_7^n(\tau_n+\theta)-J_7^n(\tau_n)\|_{\H}^2\right]\nonumber\\&\leq C\E\left[\left\|\int_{\tau_n}^{\tau_n+\theta}\boldsymbol{\sigma}^n(s,\u_n(s))\d\W^n(s)\right\|_{\H}^2\right] \leq C\E\left[\int_{\tau_n}^{\tau_n+\theta}\|\boldsymbol{\sigma}(s,\u_n(s))\|_{\mathcal{L}_{\Q}}^2\d s\right]\nonumber\\&\leq CK\E\left[\int_{\tau_n}^{\tau_n+\theta}\left(1+\|\u_n(s)\|_{\H}^2\right)\d s\right]\leq CK\theta\E\left[1+\sup_{t\in[0,T]}\|\u_n(t)\|_{\H}^2\right]\nonumber\\&\leq C\left(\|\u_0\|_{\H}, \|\mathbf{F}\|_{\mathrm{L}^{2}(0,T;\V^{\prime})} ,\mu,\alpha,p,T,K\right)\theta := c_7 \theta,
	\end{align}
	therefore the condition \eqref{37} is satisfied with $\xi=2$ and $\eta=1$. Hence, an application of Lemma \ref{lem3.11} implies that the sequence $\{\u_n\}_{n\in\N}$ satisfies the Aldous condition in $\mathbb{U}'$.  This completes the proof.
\end{proof}

Next, we obtain the following result as an application of Theorem \ref{thm-JvST} which will be used to construct a martingale solution of underlying system.

\begin{lemma}\label{lem-strong-convergence-Y}
    There exists a subsequence $\{\u_{n_k}\}_{k\in\N}$ of $\{\u_n\}_{n\in\N}$, a probability space $(\bar{\Omega},\bar{\mathscr{F}},\bar{\mathbb{P}})$ and $\mathscr{Y}$-valued random variables $\u^*$, $\bar{\u}_k$, $k\in \N$ such that the variables $\u_{n_k}$ and $\bar{\u}_k$ have the same laws on $\mathscr{Y}$ and $\bar{\u}_k$ converges to $\u^*,$ $\bar{\mathbb{P}}$-a.s. on $\bar{\Omega}$.
\end{lemma}

\begin{proof}
    From Lemma \ref{lem4.2}, we know that the set of measures $\{\mathscr{L}(\u_n)\}_{n\in\N}$ is tight on the space $(\mathscr{Y},\mathcal{T})$.  Since $\mathrm{L}^2 (0, T ; \H_{\mathrm{loc}} )$ and $\C([0, T ]; \mathbb{U}')$ are separable and completely metrizable spaces, we conclude that on each of these spaces, there exists a countable family of continuous real valued mappings separating points. For the space $\mathrm{L}^2_w(0,T;\V)$, it is enough to take $$f_m(\u):=\int_0^T[(\u(t),\v_m(t)) + (\nabla\u(t),\nabla\v_m(t))]\d t\in\R, \ \u\in\mathrm{L}^2_{w}(0,T;\V),\ m\in\N,$$ where $\{\v_m\}_{m\in\N}$ is a dense subset of $\mathrm{L}^2(0, T; \V)$. Then the sequence $\{f_m\}_{m\in\N}$ consists of continuous real-valued mappings that separate points in the space $\mathrm{L}^2_w(0, T; \V)$. For the space $\mathrm{L}^{r+1}_w(0,T;\wi\L^{r+1})$, it is enough to consider $$f_m(\u):=\int_0^T\langle|\u(t)|^{r-1}\u(t),\v_m(t)\rangle\d t\in\R, \ \u\in\mathrm{L}^{r+1}_{w}(0,T;\wi\L^{r+1}),\ m\in\N,$$ where $\{\v_m\}_{m\in\N}$ is a dense subset of $\mathrm{L}^{r+1}(0, T; \wi\L^{r+1})$. Then $\{f_m\}_{m\in\N}$ forms a sequence of continuous real-valued mappings that separates points in the space $\mathrm{L}^{r+1}_w(0, T; \wi\L^{r+1})$.  Let $\widetilde\H\subset\H$ be a countable and dense subset of $\H$. Then  for each
	$\h\in\wi\H,$ the mapping
	$$\C([0, T]; \H_w) \ni \u\mapsto(\u(\cdot),\h)\in\C([0, T]; \R)$$
	is continuous. Since $\C([0,T];\R) $ is a separable complete metric space, there exists a sequence $\{g_{\ell}\}_{\ell\in\N}$ of real-valued continuous functions defined on $\C([0, T ]; \R)$ separating points of this space. Then the mappings $f_{\h,\ell} , \h \in\wi\H, \ell \in\N$ defined by
	$$f_{\h,\ell}(\u):=g_{\ell}((\u,\h)), \ \u\in\C([0,T];\H_w) $$ form a countable family of continuous mappings on $\C([0,T];\H_w) $  separating points of this space. Hence by an application of Theorem \ref{thm-JvST}, one can conclude the proof.
\end{proof}

Let us now prove the main result on the existence of martingale solutions of SCBFEs. The construction of a martingale solution is based on the idea followed in the article \cite{ZBEM}.

\begin{proof}[Proof of Theorem \ref{thm3.4}] We prove our main result in several steps. 

\vskip 0.2 cm
\noindent\textbf{Step (1).} We first recall that, by Lemma \ref{lem-strong-convergence-Y}, there exists a subsequence $\{\u_{n_k}\}_{k\in\N}$ of $\{\u_n\}_{n\in\N}$, a probability space $(\bar{\Omega},\bar{\mathscr{F}},\bar{\mathbb{P}})$ and $\mathscr{Y}$-valued random variables $\u^*$, $\bar{\u}_k$, $k\in \N$ such that 
\begin{align}\label{eqn-covergence-Y}
    \mbox{the variables $\u_{n_k}$ and $\bar{\u}_k$ have the same laws on $\mathscr{Y}$, and $\bar{\u}_k\to \u^*$  in $\mathscr{Y},$ \;\; $\bar{\mathbb{P}}$-a.s.}
\end{align}
 Since the random variables  $\bar{\omega}\mapsto \bar{\u}_n(\cdot,\bar{\omega})$ and $\omega\mapsto \u_n(\cdot,\omega)$ are identically distributed, we have the following estimates from \eqref{ener1}: 
	\begin{align}\label{ener3}
		&\sup_{n\geq 1}\bar{\E}\left[\sup_{t\in[0,T]}\|\bar{\u}_n(t)\|_{\H}^2\right]+\mu \sup_{n\geq 1}\bar{\E}\left[\int_0^{T}\|\nabla\bar{\u}_n(t)\|_{\H}^2\d t\right] 
        + \beta\sup_{n\geq 1}\bar{\E}\left[\int_0^T\|\bar{\u}_n(t)\|_{\widetilde{\L}^{r+1}}^{r+1}\d t\right]\nonumber\\&\leq C\left(\|\u_0\|_{\H}, \|\mathbf{F}\|_{\mathrm{L}^{2}(0,T;\V^{\prime})},\mu, \alpha ,T,K\right). 
	\end{align}
	Furthermore, from \eqref{ener1} and \eqref{4p32}, we infer  
	\begin{align}\label{ener4}
		&\sup_{n\geq 1}\bar{\E}\left[\sup_{t\in[0,T]}\|\bar{\u}_n(t)\|_{\H}^{2p}+\left(\int_0^T\|\nabla\bar{\u}_n(t)\|_{\H}^2\d t\right)^p 
+\left(\int_0^T\|\bar{\u}_n(t)\|_{\wi\L^{r+1}}^{r+1}\d t\right)^p\right]\nonumber\\&\leq C\left(\|\u_0\|_{\H}, \|\mathbf{F}\|_{\mathrm{L}^{2}(0,T;\V^{\prime})} ,\mu,\alpha,\beta,p,T,K\right),
	\end{align}
for some $p\geq 1$.	Using the energy estimates \eqref{ener3} and \eqref{ener4}, an application of the Banach-Alaoglu theorem yields the existence of a subsequence $\{\bar{\u}_n\}$ (still denoted by the same symbol) such that
\begin{align*}
   \bar{\u}_n  \stackrel{w^\ast}{\rightharpoonup} \u^* \text{ in }\mathrm{L}^{2p}(\bar{\Omega},\bar{\mathscr{F}},\bar{\mathbb{P}};\mathrm{L}^{\infty}(0,T;\H))
\end{align*}
and
\begin{align*}
   \bar{\u}_n  \stackrel{w}{\rightharpoonup} \u^* \text{ in } \mathrm{L}^{2p}(\bar{\Omega},\bar{\mathscr{F}},\bar{\mathbb{P}};\mathrm{L}^{2}(0,T;\V))\cap\mathrm{L}^{(r+1)p}(\bar{\Omega},\bar{\mathscr{F}},\bar{\mathbb{P}};\mathrm{L}^{r+1}(0,T;\wi\L^{r+1})).
\end{align*}
 Moreover, we have 
	\begin{align}\label{ener5}
		&\bar{\E}\left[\sup_{t\in[0,T]}\|{\u}^*(t)\|_{\H}^2\right]+\mu \bar{\E}\left[\int_0^{T}\|\nabla{\u}^*(t)\|_{\H}^2\d t\right] 
        +\beta\bar{\E}\left[\int_0^T\|{\u}^*(t)\|_{\widetilde{\L}^{r+1}}^{r+1}\d t\right]\nonumber\\&\leq C\left(\|\u_0\|_{\H}, \|\mathbf{F}\|_{\mathrm{L}^{2}(0,T;\V^{\prime})} ,\mu,\beta,p,T,K\right)
	\end{align}
and 
	\begin{align}\label{ener6}
	&\bar{\E}\left[\sup_{t\in[0,T]}\|{\u}^*(t)\|_{\H}^{2p}+\left(\int_0^T\|\nabla{\u}^*(t)\|_{\H}^2\d t\right)^p
    +\left(\int_0^T\|{\u}^*(t)\|_{\wi\L^{r+1}}^{r+1}\d t\right)^p\right]\nonumber\\&\leq C\left(\|\u_0\|_{\H}, \|\mathbf{F}\|_{\mathrm{L}^{2}(0,T;\V^{\prime})} ,\mu,\alpha,\beta,p,T,K\right). 
\end{align}

\vskip 0.2 cm
\noindent\textbf{Step (2).}  For each $n\in\N$, we consider a process $\bar{\mathcal{X}}_n$ with trajectories in $\C([0,T];\H)$ defined by 
\begin{align}
  \bar{\mathcal{X}}_n (t) & :=  \bar{\u}_n(t) - \bar{\u}_n(0) 
  +  \mu\int_0^t \A^n\bar{\u}_n(s)\d s + \int_0^t\B^n(\bar{\u}_n(s))\d s + \alpha\int_0^t \bar{\u}_n(s)\d s \nonumber\\ & \quad +\beta\int_0^t\mathcal{C}^n(\bar{\u}_n(s))\d s - \int_0^t \Pi_n\mathbf{F}(s) \d s,
\end{align}
for all $t\in[0,T]$.   We can verify that  $\bar{\mathcal{X}}_n$ is a square integrable martingale with respect to the filtration $\bar{\mathcal{F}}_n= \{\bar{\mathscr{F}}_{n,t}\}_{t\in[0,T]}$, where $\bar{\mathscr{F}}_{n,t} = \sigma\{\bar{\u}_n(s) : s\leq t\}$ with the quadratic variation
\begin{align}\label{eqn-quad-var}
    \langle\!\langle \bar{\mathcal{X}}_n \rangle\!\rangle_{t} = \int_0^t \Pi_n \boldsymbol{\sigma}(\tau,\bar{\u}_n(\tau))\Pi_n \mathrm{Q}\Pi_n\boldsymbol{\sigma}(\tau,\bar{\u}_n(\tau))^* \Pi_n \d \tau, \;\;\; t\in [0,T].
\end{align}

Indeed, since $\bar{\u}_k$ and $\u_{n_k}$ have the same laws, for all $s,t\in[0,T]$, 
$s\leq t$,
all functions $\f$ bounded and continuous on $\C([0,t];\mathbb{U}')$, and all $\varphi$, $\psi\in \mathbb{U}$, we have
\begin{align}\label{eqn-quad-1}
    \bar{\mathbb{E}} \left[ \langle \bar{\mathcal{X}}_n(t) - \bar{\mathcal{X}}_n (s), \psi \rangle \f (\bar{\u}_n|_{[0,s]})  \right] =0
\end{align}
and 
\begin{align}\label{eqn-quad-2}
   & \bar{\mathbb{E}} \bigg[ \bigg( \langle \bar{\mathcal{X}}_n(t), \psi \rangle \langle \bar{\mathcal{X}}_n(t), \varphi \rangle - \langle \bar{\mathcal{X}}_n(s), \psi \rangle \langle \bar{\mathcal{X}}_n(s), \varphi \rangle 
    \nonumber\\ & - \int_s^t \big( \mathrm{Q}^{\frac{1}{2}} \Pi_n \boldsymbol{\sigma}(\tau,\bar{\u}_n(\tau))^*\Pi_n \psi, \; \mathrm{Q}^{\frac{1}{2}} \Pi_n \boldsymbol{\sigma}(\tau,\bar{\u}_n(\tau))^*\Pi_n \varphi\big) \d \tau \bigg) \cdot \f (\bar{\u}_n|_{[0,s]}) \bigg] =0.
\end{align}
Our aim is to pass limit in \eqref{eqn-quad-1} and \eqref{eqn-quad-2}. Let $\bar{\mathcal{X}}$ be a $\mathbb{U}'$-valued process defined by
\begin{align}\label{eqn-limit-process}
  \bar{\mathcal{X}} (t) & :=  {\u^*}(t) - {\u^*}(0) +  \mu\int_0^t \A{\u^*}(s)\d s + \int_0^t\B({\u^*}(s))\d s + \alpha\int_0^t {\u^*}(s)\d s   
  \nonumber\\ & \quad  +\beta\int_0^t\mathcal{C}({\u^*}(s))\d s
   - \int_0^t \mathbf{F}(s) \d s,
\end{align}
for all $t\in[0,T]$.
\vskip 0.2 cm
\noindent\textbf{Step (3).} {\it We claim that for all $s,t\in[0,T]$ such that $s\leq t$ and for all $\varphi\in\mathbb{U}$, we have 
\begin{align}
    {\rm (a)} & \;\;\; \lim_{n\to \infty} (\bar{\u}_n(t), \varphi)_{\H} = ({\u^*}(t), \varphi)_{\H}, \;\; \bar{\mathbb{P}}\text{-a.s.}, \label{Con1}\\
    \rm (b) &  \;\;\; \lim_{n\to \infty}  \int_s^t \langle\A\bar{\u}_n(\tau), \Pi_n\varphi\rangle\d \tau  = \int_s^t \langle\A{\u^*}(\tau), \varphi\rangle\d \tau, \;\; \bar{\mathbb{P}}\text{-a.s.},\label{Con2}\\
    {\rm (c)} &  \;\;\; \lim_{n\to \infty}  \int_s^t \langle\B(\bar{\u}_n(\tau)), \Pi_n\varphi\rangle\d \tau = \int_s^t \langle\B ({\u^*}(\tau)), \varphi\rangle\d \tau, \;\; \bar{\mathbb{P}}\text{-a.s.},\label{Con3}\\   
    {\rm (d)} &  \;\;\; \lim_{n\to \infty}  \int_s^t (\bar{\u}_n(\tau), \varphi)\d \tau  = \int_s^t ({\u^*}(\tau), \varphi) \d \tau, \;\; \bar{\mathbb{P}}\text{-a.s.},\label{Con4}\\
    {\rm (e)} & \;\;\; \lim_{n\to \infty}  \int_s^t \langle\mathcal{C}(\bar{\u}_n(\tau)), \Pi_n\varphi\rangle\d \tau = \int_s^t \langle\mathcal{C} ({\u^*}(\tau)), \varphi\rangle\d \tau, \;\; \bar{\mathbb{P}}\text{-a.s.}\label{Con5}
\end{align}}
Let us choose and fix $s,t\in [0,T]$ such that $s\leq t$ and $\varphi\in\mathbb{U}$. In view of \eqref{eqn-covergence-Y}, we have
\begin{align}
    \bar{\u}_n \to \u^* \text{ in } \C([0,T]; \mathbb{U}^{\prime}) \cap\mathrm{L}_w^2(0,T;\V)\cap\mathrm{L}_w^{r+1}(0,T;\wi\L^{r+1})\cap \mathrm{L}^2(0,T;\H_{\mathrm{loc}})\cap\C([0,T];\H_w), 
\end{align}
$\bar{\mathbb{P}}$-a.s. Assertion $\mathrm{(a)}$ follows from the convergence $\bar{\u}_n \to \u^*$ in $\C([0,T];\H_w)$, $\bar{\mathbb{P}}$-a.s. 
Assertions $\mathrm{(b)}$ and $\mathrm{(d)}$ follow from the convergence $\bar{\u}_n \to \u^*$ in $\mathrm{L}_w^2(0,T;\V)$, $\bar{\mathbb{P}}$-a.s. and Lemma \ref{lem-conv-Pi_m}. Assertion $\mathrm{(c)}$ follows from the convergence $\bar{\u}_n \to \u^*$ in $\mathrm{L}_w^2(0,T;\V)\cap \mathrm{L}^2(0,T;\H_{\mathrm{loc}})$, $\bar{\mathbb{P}}$-a.s. and Lemma \ref{lem-convergence-B}. Assertion $\mathrm{(e)}$ follows from the convergence $\bar{\u}_n \to \u^*$ in $\mathrm{L}_w^{r+1}(0,T;\wi\L^{r+1})\cap \mathrm{L}^2(0,T;\H_{\mathrm{loc}})$, $\bar{\mathbb{P}}$-a.s. and Lemma \ref{lem-convergence-C}.

\vskip 0.2 cm
\noindent\textbf{Step (4).} {\it In this step, we show that for all $s,t\in[0,T]$ with $s\leq t$ and for all $\varphi\in \mathbb{U}:$
\begin{align}\label{eqn-step-4}
    \lim_{n\to\infty} \bar{\mathbb{E}} \left[ \langle \bar{\mathcal{X}}_n(t) - \bar{\mathcal{X}}_n (s), \varphi \rangle \f (\bar{\u}_n|_{[0,s]})  \right] = \bar{\mathbb{E}} \left[ \langle \bar{\mathcal{X}}(t) - \bar{\mathcal{X}} (s), \varphi \rangle \f (\u^*|_{[0,s]})  \right].
\end{align}}

We fix $s,t\in[0,T]$ with $s\leq t$ and $\varphi\in \mathbb{U}$. In view of \eqref{eqn-adjoint-projection}, we write 
\begin{align*}
    \langle \bar{\mathcal{X}}_n(t) - \bar{\mathcal{X}}_n (s), \psi \rangle & = (\bar{\u}_n(t), \Pi_n\varphi)_{\H} - (\bar{\u}_n(s), \Pi_n\varphi)_{\H} +  \mu \int_s^t \langle\A\bar{\u}_n(\tau), \Pi_n\varphi\rangle\d \tau
    \nonumber\\ & \quad + \int_s^t \langle\B(\bar{\u}_n(\tau)), \Pi_n\varphi\rangle\d \tau + \alpha \int_s^t (\bar{\u}_n(\tau), \varphi)\d \tau + \beta \int_s^t \langle\mathcal{C}(\bar{\u}_n(\tau)), \Pi_n\varphi\rangle\d \tau.
\end{align*}
Making use of convergences \eqref{Con1}-\eqref{Con2}, we find 
\begin{align}\label{eqn-Xn-conv}
    \lim_{n\to\infty}  \langle \bar{\mathcal{X}}_n(t) - \bar{\mathcal{X}}_n (s), \psi \rangle   = \langle \bar{\mathcal{X}}(t) - \bar{\mathcal{X}} (s), \psi \rangle, \;\;\; \bar{\mathbb{P}}\text{-a.s.}
\end{align}

Due to continuity property of $\f$, we note that $\bar{\mathbb{P}}$-a.s., $\lim\limits_{n\to\infty} \f (\bar{\u}_n|_{[0,s]}) = \f (\bar{\u}^*|_{[0,s]})$, and we also have $\sup\limits_{n\in\N}\|\f (\bar{\u}_n|_{[0,s]})\|_{\mathrm{L}^{\infty}}< + \infty$. We now define
\begin{align*}
    \h_n(\omega) :=  \bigg( \langle \bar{\mathcal{X}}_n(t, \omega), \varphi \rangle  - \langle \bar{\mathcal{X}}_n(s,\omega), \varphi \rangle \bigg)  \f \big(\bar{\u}_n|_{[0,s]}(\omega)\big), \;\;\text{for a.e. }\omega\in \bar{\Omega}.
\end{align*}
We are now going to show that the functions $\{\h_n\}_{n\in\N}$ are uniformly integrable. We first show that 
\begin{align}\label{eqn-unif-1}
    \sup_{n\in\N} \bar{\mathbb{E}} \big[ |\h_n|^2 \big] < +\infty.
\end{align}
Indeed, in view of the continuous embedding $\mathbb{U}\hookrightarrow\H$ and the Cauchy-Schwarz inequality, we have for each $n\in\N$
\begin{align}\label{eqn-unif-2}
    \bar{\mathbb{E}} \big[ |\h_n|^2 \big] \leq C \|\f\|^2_{\mathrm{L}^{\infty}}\|\varphi\|_{\mathbb{U}}^2\bar{\mathbb{E}} \big[ \|\bar{\mathcal{X}}_n(t)\|_{\H}^2 + \|\bar{\mathcal{X}}_n(s)\|_{\H}^2 \big].
\end{align}
Since $\bar{\mathcal{X}}_n(\cdot)$ is a continuous martingale with the quadratic variation defined in \eqref{eqn-quad-var}, an application of the  BDG inequality, Hypothesis \ref{hyp} (H.2) and \eqref{ener3} yield
\begin{align}\label{eqn-unif-3}
    \bar{\mathbb{E}} \bigg[ \sup_{t\in[0,T]}\|\bar{\mathcal{X}}_n(t)\|_{\H}^2  \bigg] & \leq C \bar{\mathbb{E}} \bigg[ \bigg( \int_0^T\|\Pi_n\boldsymbol{\sigma} (\tau, \bar{\u}_n(\tau))\Pi_n\|_{\mathcal{L}_{\mathrm{Q}}}^2 \d \tau \bigg)^{\frac12} \bigg]
    \nonumber\\ & \leq C \bar{\mathbb{E}} \bigg[ \bigg( \int_0^T\|\boldsymbol{\sigma} (\tau, \bar{\u}_n(\tau))\|_{\mathcal{L}_{\mathrm{Q}}}^2 \d \tau \bigg)^{\frac12} \bigg]
    \nonumber\\ & \leq C \bar{\mathbb{E}} \bigg[ \bigg( \int_0^T K \left[1 +  \|\bar{\u}_n(\tau)\|_{\H}^2\right] \d \tau \bigg)^{\frac12} \bigg] <\infty.
\end{align}
This implies from \eqref{eqn-unif-2} and \eqref{eqn-unif-3} that \eqref{eqn-unif-1} holds. Since the sequence $\{ \h_n\}_{n\in\N}$ is uniformly integrable, the  $\bar{\mathbb{P}}$-a.s.  convergence given in \eqref{eqn-Xn-conv} allows apply the Vitali convergence theorem  to obtain \eqref{eqn-step-4}.

\vskip 0.2 cm
\noindent\textbf{Step (5).} {\it In this step, we show that for all $s,t\in[0,T]$ with $s\leq t$ and for all $\psi, \varphi \in \mathbb{U}:$
\begin{align}\label{eqn-step-5}
    & \lim_{n\to\infty} \bar{\mathbb{E}} \left[ \big\{ \langle \bar{\mathcal{X}}_n(t), \psi \rangle \langle \bar{\mathcal{X}}_n(t), \varphi \rangle  - \langle \bar{\mathcal{X}}_n(s), \psi \rangle \langle \bar{\mathcal{X}}_n(s), \varphi \rangle\big\} \f (\bar{\u}_n|_{[0,s]})  \right] 
    \nonumber\\ & \quad = \bar{\mathbb{E}} \left[ \big\{ \langle \bar{\mathcal{X}}(t), \psi \rangle \langle \bar{\mathcal{X}}(t), \varphi \rangle  - \langle \bar{\mathcal{X}}(s), \psi \rangle \langle \bar{\mathcal{X}}(s), \varphi \rangle\big\} \f (\bar{\u}|_{[0,s]})  \right].
\end{align}}

We fix $s,t\in [0,T]$ with $s\leq t$ and $\psi, \varphi \in \mathbb{U}$ and write
\begin{align*}
  \g_n(\omega) & :=   \big\{ \bar{\mathcal{X}}_n(t,\omega), \psi \rangle \langle \bar{\mathcal{X}}_n(t,\omega), \varphi \rangle  - \langle \bar{\mathcal{X}}_n(s,\omega), \psi \rangle \langle \bar{\mathcal{X}}_n(s, \omega), \varphi \rangle\big\}  \f \big(\bar{\u}_n|_{[0,s]}(\omega)\big), && \omega\in \bar{\Omega},\\
  \g(\omega) & :=   \big\{ \bar{\mathcal{X}}(t,\omega), \psi \rangle \langle \bar{\mathcal{X}}(t,\omega), \varphi \rangle  - \langle \bar{\mathcal{X}}(s,\omega), \psi \rangle \langle \bar{\mathcal{X}}(s, \omega), \varphi \rangle\big\}  \f \big({\u^*}|_{[0,s]}(\omega)\big),  &&\omega\in \bar{\Omega}.
\end{align*}
In view of \eqref{Con1}-\eqref{Con5}, we have that $\lim\limits_{n\to\infty}\g_n(\omega) = \g(\omega)$, $\bar{\mathbb{P}}$-a.s. Next, let us show that the functions $\{\g_{n}\}_{n\in\N}$ are uniformly integrable. For this purpose, it is enough to prove that for some $\eta>1$,
\begin{align}\label{eqn-unif-4}
    \sup_{n\in\N} \bar{\mathbb{E}} \big[ |\g_n|^\eta \big] < +\infty.
\end{align}
Again, in view of continuous embedding $\mathbb{U}\hookrightarrow\H$ and the Cauchy-Schwarz inequality, we have for each $n\in\N$
\begin{align}\label{eqn-unif-5}
    \bar{\mathbb{E}} \big[ |\g_n|^\eta \big] \leq C^{2r} \|\f\|^\eta_{\mathrm{L}^{\infty}} \|\psi\|_{\mathbb{U}}^\eta \|\varphi\|_{\mathbb{U}}^\eta \bar{\mathbb{E}} \big[ \|\bar{\mathcal{X}}_n(t)\|_{\H}^{2\eta} + \|\bar{\mathcal{X}}_n(s)\|_{\H}^{2\eta} \big].
\end{align}
Since $\bar{\mathcal{X}}_n(\cdot)$ is a continuous martingale with the quadratic variation defined in \eqref{eqn-quad-var}, an application of the BDG inequality, Hypothesis \ref{hyp} (H.2) and \eqref{ener4} yield
\begin{align}\label{eqn-unif-6}
    \bar{\mathbb{E}} \bigg[ \sup_{t\in[0,T]}\|\bar{\mathcal{X}}_n(t)\|_{\H}^{2\eta}  \bigg] & \leq C \bar{\mathbb{E}} \bigg[ \bigg( \int_0^T\|\Pi_n\boldsymbol{\sigma} (\tau, \bar{\u}_n(\tau))\Pi_n\|_{\mathcal{L}_{\mathrm{Q}}}^2 \d \tau \bigg)^{\eta} \bigg]
    \nonumber\\ & \leq C \bar{\mathbb{E}} \bigg[ \bigg( \int_0^T\|\boldsymbol{\sigma} (\tau, \bar{\u}_n(\tau))\|_{\mathcal{L}_{\mathrm{Q}}}^2 \d \tau \bigg)^{\eta} \bigg]
    \nonumber\\ & \leq C \bar{\mathbb{E}} \bigg[ \bigg( \int_0^T K \left[1 +  \|\bar{\u}_n(\tau)\|_{\H}^2\right] \d \tau \bigg)^{\eta} \bigg] <\infty.
\end{align}
Combining \eqref{eqn-unif-5} and \eqref{eqn-unif-6} yields \eqref{eqn-unif-4}. Because the sequence $\{ \g_n\}_{n\in\N}$ is uniformly integrable, the Vitali convergence theorem can be applied to conclude \eqref{eqn-step-4}.

\vskip 0.2 cm
\noindent\textbf{Step (6)}: (\textit{Convergence of the quadratic variation.}) {\it  In this step, we show that for all $s,t\in[0,T]$ with $s\leq t$ and for all $\psi, \varphi \in \mathbb{U}$, we otain
\begin{align}\label{Conv-quad-vari}
    & \lim_{n\to\infty}\bar{\mathbb{E}} \bigg[ \bigg(\int_s^t \big(\mathrm{Q}^{\frac{1}{2}}\Pi_n \boldsymbol{\sigma}(\tau,\bar{\u}_n(\tau))^*\Pi_n \psi, \; \mathrm{Q}^{\frac{1}{2}} \Pi_n \boldsymbol{\sigma}(\tau,\bar{\u}_n(\tau))^*\Pi_n \varphi\big) \d \tau \bigg)  \f (\bar{\u}_n|_{[0,s]}) \bigg] 
    \nonumber\\ & = \bar{\mathbb{E}} \bigg[ \bigg(\int_s^t \big(\boldsymbol{\sigma}(\tau,{\u^*}(\tau))^* \psi,  \boldsymbol{\sigma}(\tau,{\u^*}(\tau))^* \varphi\big) \d \tau \bigg)  \f ({\u^*}|_{[0,s]}) \bigg].
\end{align}
}
We fix $s,t\in [0,T]$ with $s\leq t$ and $\psi, \varphi \in \mathbb{U}$ and define
\begin{align*}
 & \mathfrak{F}_n(\omega)
 \nonumber\\ & :=  \bigg(\int_s^t \big( \mathrm{Q}^{\frac{1}{2}} \Pi_n \boldsymbol{\sigma}(\tau,\bar{\u}_n(\tau, \omega))^*\Pi_n \psi, \; \mathrm{Q}^{\frac{1}{2}} \Pi_n \boldsymbol{\sigma}(\tau,\bar{\u}_n(\tau,\omega))^*\Pi_n \varphi\big) \d \tau \bigg)  \f (\bar{\u}_n|_{[0,s]}(\omega)), \;\; \text{for a.e. }\omega\in \bar{\Omega}.
\end{align*}
Proceeding as in the previous steps, we show that $\mathfrak{F}_n$  is $\bar{\mathbb{P}}$-a.s convergent  and uniformly integrabile. For uniform integrability, it is enough to demonstrate  that for some $\eta>1$
\begin{align}\label{eqn-unif-7}
    \sup_{n\in\N} \bar{\mathbb{E}} \big[ |\mathfrak{F}_n|^\eta \big] < +\infty.
\end{align}
Let us consider
\begin{align}
  \bar{\mathbb{E}} \left[ |\mathfrak{F}_n |^{\eta} \right] & =  \bar{\mathbb{E}} \left[ \left|  \bigg(\int_s^t \big( \mathrm{Q}^{\frac{1}{2}}\Pi_n \boldsymbol{\sigma}(\tau,\bar{\u}_n(\tau)^*\Pi_n \psi, \; \mathrm{Q}^{\frac{1}{2}} \Pi_n \boldsymbol{\sigma}(\tau,\bar{\u}_n(\tau))^*\Pi_n \varphi\big) \d \tau \bigg) \cdot \f (\bar{\u}_n|_{[0,s]})\right|^{\eta}\right]
   \nonumber\\ &
   \leq \|\f\|_{\mathrm{L}^{\infty}}^{\eta}   \bar{\mathbb{E}} \left[\bigg(\int_s^t \big\|\mathrm{Q}^{\frac{1}{2}}\Pi_n \boldsymbol{\sigma}(\tau,\bar{\u}_n(\tau))^*\Pi_n \psi\big\|_{\H} \big\|\mathrm{Q}^{\frac{1}{2}} \Pi_n \boldsymbol{\sigma}(\tau,\bar{\u}_n(\tau))^*\Pi_n \varphi\big\|_{\H} \d \tau \bigg)^{\eta}\right]
   \nonumber\\ &
   \leq C^{2\eta} \|\f\|_{\mathrm{L}^{\infty}}^{\eta} \|\psi\|_{\mathbb{U}}^{\eta}  \|\varphi\|_{\mathbb{U}}^{\eta} \bar{\mathbb{E}} \left[ \bigg(\int_s^t \big\|\boldsymbol{\sigma}(\tau,\bar{\u}_n(\tau))^*\big\|_{\mathcal{L}_{\mathrm{Q}}}^2 \d \tau \bigg)^{\eta}\right]
   \nonumber\\ &
   \leq C^{2\eta} K^{\eta} \|\f\|_{\mathrm{L}^{\infty}}^{\eta} \|\psi\|_{\mathbb{U}}^{\eta}  \|\varphi\|_{\mathbb{U}}^{\eta}  \bar{\mathbb{E}} \left[ \bigg(\int_s^t \left[1+ \big\|\bar{\u}_n(\tau)\big\|_{\H}^2\right] \d \tau \bigg)^{\eta}\right]
   \nonumber\\ &
   \leq C^{2\eta} K^{\eta} \|\f\|_{\mathrm{L}^{\infty}}^{\eta} \|\psi\|_{\mathbb{U}}^{\eta}  \|\varphi\|_{\mathbb{U}}^{\eta} (t-s)^{\eta-1} \bar{\mathbb{E}} \left[\int_s^t \left[1+  \big\|\bar{\u}_n(\tau)\big\|_{\H}^2\right]^{\eta} \d \tau \right], 
\end{align}
which implies \eqref{eqn-unif-7} in view of \eqref{ener4}.

We now proceed to establish the $\bar{\mathbb{P}}$-a.s. convergence of $\mathfrak{F}_n$. Let us choose and fix $\omega\in \bar{\Omega}$ such that 
\begin{align}
   (i) \; &   \bar{\u}_n(\cdot,\omega)\to \u^* (\cdot,\omega) \text{  in } \mathrm{L}^2(0,T ; \H_{\mathrm{loc}}). \label{prop-fix-omega-1} \\
    (ii) \;   & \u^*(\cdot,\omega)\in \mathrm{L}^2(0,T ; \H) \text{ and the sequence } \{\bar{\u}_n(\cdot,\omega)\}_{n\in\N} \text{ is bounded in } \mathrm{L}^2(0,T ; \H). \label{prop-fix-omega-2}
\end{align}
Our aim is to show that 
\begin{align}
    &  \int_s^t \big( \mathrm{Q}^{\frac{1}{2}} \Pi_n \boldsymbol{\sigma}(\tau,\bar{\u}_n(\tau,\omega))^*\Pi_n \psi,\; \mathrm{Q}^{\frac{1}{2}} \Pi_n \boldsymbol{\sigma}(\tau,\bar{\u}_n(\tau,\omega))^*\Pi_n \varphi\big) \d \tau  
    \nonumber\\ &  \to \int_s^t \big( \mathrm{Q}^{\frac{1}{2}}\boldsymbol{\sigma}(\tau,{\u^*}(\tau,\omega))^* \psi, \; \mathrm{Q}^{\frac{1}{2}}  \boldsymbol{\sigma}(\tau,{\u^*}(\tau,\omega))^* \varphi\big) \d \tau, \quad  \text{as } n\to \infty.
\end{align}
For this purpose, it is enough to ensure that 
\begin{align}\label{Conv-noise-coe}
    \mathrm{Q}^{\frac{1}{2}} \Pi_n \boldsymbol{\sigma}(\tau,\bar{\u}_n(\tau,\omega))^*\Pi_n \psi \to \mathrm{Q}^{\frac{1}{2}} \boldsymbol{\sigma}(\tau,{\u^*}(\tau,\omega))^* \psi \text{ in } \; \mathrm{L}^2(s,t;\H), \quad \text{ for all }
     \psi\in \mathbb{U}.
\end{align}
Therefore, we consider 
\begin{align}
    & \int_{s}^{t} \|\mathrm{Q}^{\frac{1}{2}}\Pi_n \boldsymbol{\sigma}(\tau,\bar{\u}_n(\tau,\omega))^*\Pi_n \psi - \mathrm{Q}^{\frac{1}{2}} \boldsymbol{\sigma}(\tau,\bar{\u}_n(\tau,\omega))^*\Pi_n \psi \|_{\H}^2 \d \tau
    \nonumber\\ & 
    \leq \int_{s}^{t} \big[ \|\mathrm{Q}^{\frac{1}{2}}\Pi_n \boldsymbol{\sigma}(\tau,\bar{\u}_n(\tau,\omega))^*\Pi_n \psi - \mathrm{Q}^{\frac{1}{2}}\boldsymbol{\sigma}(\tau,{\u^*}(\tau,\omega))^* \Pi_n \psi\|_{\H} 
    \nonumber\\ & \qquad + \|\mathrm{Q}^{\frac{1}{2}}\boldsymbol{\sigma}(\tau,\bar{\u}_n(\tau,\omega))^*\Pi_n \psi - \mathrm{Q}^{\frac{1}{2}}\boldsymbol{\sigma}(\tau,\bar{\u}_n(\tau,\omega))^* \psi\|_{\H} 
     \nonumber\\ & \qquad  + \|\mathrm{Q}^{\frac{1}{2}} \boldsymbol{\sigma}(\tau,\bar{\u}_n(\tau,\omega))^* \psi - \mathrm{Q}^{\frac{1}{2}}\boldsymbol{\sigma}(\tau,{\u^*}(\tau,\omega))^* \psi\|_{\H} \big]^2 \d \tau
     \nonumber\\ & 
    \leq 4\int_{s}^{t}  \|\mathrm{Q}^{\frac{1}{2}}\Pi_n \boldsymbol{\sigma}(\tau,\bar{\u}_n(\tau,\omega))^*\Pi_n \psi - \mathrm{Q}^{\frac{1}{2}}\boldsymbol{\sigma}(\tau,{\u^*}(\tau,\omega))^* \Pi_n \psi\|^2_{\H} \d \tau
    \nonumber\\ &  \quad  + 4\int_{s}^{t}  \|\mathrm{Q}^{\frac{1}{2}}\boldsymbol{\sigma}(\tau,\bar{\u}_n(\tau,\omega))^*\Pi_n \psi - \mathrm{Q}^{\frac{1}{2}}\boldsymbol{\sigma}(\tau,\bar{\u}_n(\tau,\omega))^* \psi\|^2_{\H}  \d \tau
     \nonumber\\ & \quad  + 4\int_{s}^{t}  \| \mathrm{Q}^{\frac{1}{2}}\boldsymbol{\sigma}(\tau,\bar{\u}_n(\tau,\omega))^* \psi - \mathrm{Q}^{\frac{1}{2}} \boldsymbol{\sigma}(\tau,{\u^*}(\tau,\omega))^* \psi\|^2_{\H}  \d \tau
     \nonumber \\ & =: L_1^n + L_2^n + L_3^n.
\end{align}
In view of the properties of projection $\Pi_n$ in  Lemma \ref{lem-conv-Pi_m}, Hypothesis \ref{hyp} (H.2) and \eqref{prop-fix-omega-1}, we immediately have
\begin{align*}
    \lim_{n\to\infty} [L_1^n + L_2^n] = 0.
\end{align*}
    Now, it remains to show that $\lim\limits_{n\to\infty} L_3^n=0.$  Let us choose $\varphi\in\mathcal{V}$ and fix it. Then there exists  $R>0$ such that $\mathrm{supp}\;\varphi$ is a compact subset of $\mathcal{O}_{R}$. Since, by \eqref{prop-fix-omega-1}, we have $\bar{\u}_n(\cdot,\omega)\to \u^* (\cdot,\omega)$ in  $\mathrm{L}^2(0,T ; \H_{\mathcal{O}_R})$, it follows that, along a subsequence (not relabeled),
    \begin{align*}
        \bar{\u}_n(\tau,\omega)\to \u^* (\tau,\omega) \text{ in }  \H_{\mathcal{O}_R} \text{ for a.e. } \tau\in [0,T] \; \text{ as }  n\to \infty.
    \end{align*}
    Hence by Hypothesis \ref{hyp} (H.4)
\begin{align*}
        \mathrm{Q}^{\frac{1}{2}} \boldsymbol{\sigma}(\tau,\bar{\u}_n(\tau,\omega))^*  \varphi \to \mathrm{Q}^{\frac{1}{2}} \boldsymbol{\sigma}(\tau,\u^* (\tau,\omega))^* \varphi  \text{ in }  \H \text{ for a.e. } \tau\in [0,T] \; \text{ as }  n\to \infty.
    \end{align*}
Consequently, by the Vitali convergence theorem, we have 
\begin{align}\label{J_3}
    \lim_{n\to\infty} L^n_3 = 0, \;\;\text{ for }  \varphi\in \mathcal{V}.
\end{align}

If $\psi\in \H$, then for every $\delta>0$, ther exists  $\psi_{\delta}\in \mathcal{V}$ such that $\|\psi -\psi_{\delta}\|_{\H}\leq \delta$. Consider for $\psi\in \H$,
\begin{align*}
    & \int_{s}^{t}  \|\mathrm{Q}^{\frac{1}{2}} \boldsymbol{\sigma}(\tau,\bar{\u}_n(\tau,\omega))^* \psi - \mathrm{Q}^{\frac{1}{2}} \boldsymbol{\sigma}(\tau,{\u^*}(\tau,\omega))^* \psi\|^2_{\H}  \d \tau
    \nonumber\\ & \leq 2 \int_{s}^{t}  \| \mathrm{Q}^{\frac{1}{2}} \boldsymbol{\sigma}(\tau,\bar{\u}_n(\tau,\omega))^* (\psi-\psi_{\delta}) - \mathrm{Q}^{\frac{1}{2}} \boldsymbol{\sigma}(\tau,{\u^*}(\tau,\omega))^* (\psi-\psi_{\delta})\|^2_{\H}  \d \tau
    \nonumber\\ & \qquad + 2 \int_{s}^{t}  \| \mathrm{Q}^{\frac{1}{2}} \boldsymbol{\sigma}(\tau,\bar{\u}_n(\tau,\omega))^* \psi_{\delta} - \mathrm{Q}^{\frac{1}{2}} \boldsymbol{\sigma}(\tau,{\u^*}(\tau,\omega))^* \psi_{\delta}\|^2_{\H}  \d \tau
    \nonumber\\ & \leq C^{*} \|\psi-\psi_{\delta}\|^2_{\H}  + 2 \int_{s}^{t}  \|\mathrm{Q}^{\frac{1}{2}} \boldsymbol{\sigma}(\tau,\bar{\u}_n(\tau,\omega))^* \psi_{\delta} - \mathrm{Q}^{\frac{1}{2}} \boldsymbol{\sigma}(\tau,{\u^*}(\tau,\omega))^* \psi_{\delta}\|^2_{\H}  \d \tau,
\end{align*}
where $C^*>0$ is a constant independent of $n$ and $\delta$. Taking $\limsup$ on both side and using \eqref{J_3}, we find
\begin{align*}
    & \limsup_{n\to\infty} \int_{s}^{t}  \| \mathrm{Q}^{\frac{1}{2}}\boldsymbol{\sigma}(\tau,\bar{\u}_n(\tau,\omega))^* \psi -  \mathrm{Q}^{\frac{1}{2}} \boldsymbol{\sigma}(\tau,{\u^*}(\tau,\omega))^* \psi\|^2_{\H}  \d \tau \leq C^{*} \delta^2.
\end{align*}
Since $\delta>0$ is arbitrary, we infer 
\begin{align}\label{J_31}
    \lim_{n\to\infty} L^n_3 = 0, \;\;\text{ for } \psi\in \H.
\end{align}
This concludes the proof of \eqref{Conv-noise-coe} and hence \eqref{Conv-quad-vari} holds.

\vskip 0.2 cm
\noindent\textbf{Step (7).} Making use of \eqref{eqn-step-4}, \eqref{eqn-step-5} and \eqref{Conv-quad-vari} in \eqref{eqn-quad-1}-\eqref{eqn-quad-2}, we conclude that 
for all $s,t\in[0,T]$, $s\leq t$,  for every bounded continuous function $\f$ on
 $\C([0,t];\mathbb{U}')$, and all $\varphi$, $\psi\in \mathbb{U}$, we have
\begin{align}\label{eqn-quad-3}
    \bar{\mathbb{E}} \left[ \langle \bar{\mathcal{X}}(t) - \bar{\mathcal{X}} (s), \psi \rangle \f ({\u^*}|_{[0,s]})  \right] =0
\end{align}
and 
\begin{align}\label{eqn-quad-4}
   & \bar{\mathbb{E}} \bigg[ \bigg( \langle \bar{\mathcal{X}}(t), \psi \rangle \langle \bar{\mathcal{X}}(t), \varphi \rangle - \langle \bar{\mathcal{X}}(s), \psi \rangle \langle \bar{\mathcal{X}}(s), \varphi \rangle 
    \nonumber\\ & - \int_s^t \big( \mathrm{Q}^{\frac12} \boldsymbol{\sigma}(\tau,{\u^*}(\tau))^* \psi, \mathrm{Q}^{\frac12}\boldsymbol{\sigma}(\tau,{\u^*}(\tau))^* \varphi\big) \d \tau \bigg)   \f ({\u^*}|_{[0,s]}) \bigg] =0,
\end{align}
where $\bar{\mathcal{X}}$ is a $\mathbb{U}'$-valued process defined in \eqref{eqn-limit-process}. 

Now we use the idea similar to the one used in \cite[Section 8.4]{DaZ} and \cite{ZBEM}. Consider the operator $\mathfrak{L}: \D(\mathfrak{L}) \subset \mathbb{U}\to \H$ defined by \eqref{eqn-compact-op-L}, the inverse $\mathfrak{L}^{-1}:\H\to \mathbb{U}$ and its dual $(\mathfrak{L}^{-1})^{\prime}: \mathbb{U}^{\prime} \to \H^{\prime}$. From \eqref{eqn-quad-3} and \eqref{eqn-quad-4} with $\psi= \mathfrak{L}^{-1} \zeta $ and $\varphi= \mathfrak{L}^{-1} \phi$, where $\zeta, \phi \in \H$, we infer  that $(\mathfrak{L}^{-1})^{\prime}\bar{\mathcal{X}}(t)$, $t\in[0,T]$ is a continuous square integrable martingale in $\H^{\prime}\cong \H$ with respect to the filtration $\bar{\mathcal{F}} = \{\bar{\mathscr{F}}_{t}\}_{t\in[0,T]}$, where $\bar{\mathscr{F}}_{t}= \sigma \{\u^*(s): s\leq t\}$ with the quadratic variation 
\begin{align}\label{eqn-quad-var-L}
    \langle\!\langle (\mathfrak{L}^{-1})^{\prime}\bar{\mathcal{X}} \rangle\!\rangle_{t} = \int_0^t  (\mathfrak{L}^{-1})^{\prime}\boldsymbol{\sigma}(\tau,{\u^*}(\tau)) \mathrm{Q} \big( (\mathfrak{L}^{-1})^{\prime} \boldsymbol{\sigma}(\tau,{\u^*}(\tau))\big)^*  \d \tau, \;\;\; t\in [0,T].
\end{align}
Particularly, the continuity of the process $(\mathfrak{L}^{-1})^{\prime}\bar{\mathcal{X}}$ follows from the fact that $\u^* \in \C([0,T];\mathbb{U}^{\prime})$. By the martingale representation theorem \cite[Section 8.2]{DaZ}, there exist:
\begin{itemize}
    \item a stochastic basis $(\bar{\bar{\Omega}},\bar{\bar{\mathscr{F}}}, \{\bar{\bar{\mathscr{F}}}\}_{t\in[0,T]} ,\bar{\bar{\mathbb{P}}})$,
    \item a $\mathrm{Q}$-Wiener process $\{\bar{\bar{\mathrm{W}}}(t)\}_{t\geq 0}$ defined on this basis,
    \item and a progressively measurable process $\bar{\bar{\u}}$ such that 
\begin{align}\label{eqn-MRT}
   & (\mathfrak{L}^{-1})^{\prime}\bar{\bar{\u}}(t) - (\mathfrak{L}^{-1})^{\prime}\bar{\bar{\u}}(0) +  \mu (\mathfrak{L}^{-1})^{\prime} \int_0^t \A{\bar{\bar{\u}}}(s)\d s + (\mathfrak{L}^{-1})^{\prime}\int_0^t\B({\bar{\bar{\u}}}(s))\d s 
   \nonumber\\ & \quad+ \alpha(\mathfrak{L}^{-1})^{\prime}\int_0^t {\bar{\bar{\u}}}(s)\d s   +\beta (\mathfrak{L}^{-1})^{\prime} \int_0^t\mathcal{C}(\bar{\bar{\u}}(s))\d s
    \nonumber\\ & =  (\mathfrak{L}^{-1})^{\prime} \int_0^t\mathbf{F}(s)\d s +  \int_0^t (\mathfrak{L}^{-1})^{\prime} \boldsymbol{\sigma}(s,\bar{\bar{\u}}(s))\d \bar{\bar{\mathrm{W}}}(s).
\end{align}
\end{itemize}
However,
    \begin{align*}
        \int_0^t (\mathfrak{L}^{-1})^{\prime} \boldsymbol{\sigma}(s,\bar{\bar{\u}}(s))\d \bar{\bar{\mathrm{W}}}(s) = (\mathfrak{L}^{-1})^{\prime}  \int_0^t \boldsymbol{\sigma}(s,\bar{\bar{\u}}(s))\d \bar{\bar{\mathrm{W}}}(s).
    \end{align*}
    Hence for all $t\in[0,T]$ and all $\varphi\in\mathbb{U}$
    \begin{align}\label{eqn-weak-SCBF-U}
        & (\bar{\bar{\u}}(t), \varphi)_{\H} - (\bar{\bar{\u}}(0),\varphi)_{\H} +  \mu  \int_0^t \langle\A{\bar{\bar{\u}}}(s), \varphi\rangle \d s + \int_0^t \langle \B({\bar{\bar{\u}}}(s)), \varphi\rangle \d s 
   \nonumber\\ & \quad+ \alpha \int_0^t ({\bar{\bar{\u}}}(s), \varphi )_{\H}\d s   +\beta  \int_0^t \langle \mathcal{C}(\bar{\bar{\u}}(s)), \varphi\rangle \d s
    \nonumber\\ & = \left\langle  \int_0^t\mathbf{F}(s)\d s, \varphi\right\rangle + \left\langle \int_0^t \boldsymbol{\sigma}(s,\bar{\bar{\u}}(s))\d \bar{\bar{\mathrm{W}}}(s), \varphi\right\rangle.
    \end{align}
Since $\mathbb{U}$ is densely embedded in $\V\cap\wi\L^{r+1}$, one can obtain that \eqref{eqn-weak-SCBF-U} holds for all $\varphi\in\V\cap\wi\L^{r+1}$. 
Hence, we deduce that the system $((\bar{\bar{\Omega}},\bar{\bar{\mathscr{F}}},\{\bar{\bar{\mathscr{F}}}_t\}_{t\geq 0},\bar{\bar{\mathbb{P}}}),\bar{\bar{\u}},\bar{\bar{\W}})$ is a martingale solution of the stochastic system \eqref{32}, which completes the proof. 
\end{proof}  

\section{Existence and Uniqueness of Strong Solutions}\label{sec5}\setcounter{equation}{0}  
For all the cases given in Table \ref{Table-2},  the weak martingale solution of SCBFEs obtained in  Theorem \ref{thm3.4} has stronger regularity properties. We prove that $\bar{\mathbb{P}}$-a.s., the trajectories are continuous $\H$-valued function defined on $[0, T ]$. Moreover, for all the cases given in Table \ref{Table-3}, by showing the pathwise uniqueness of weak martingale solutions, we use the classical Yamada-Watanabe argument to show the existence of a unique strong solution and uniqueness in law. The existence and uniqueness of a strong solution using global monotonicity property of the linear and nonlinear operators and a stcohastic generalization of the Minty-Browder technique is established  in the work \cite{KK+MTM-SCBF}. 
\begin{proposition}\label{prop5.1}
	 Let $d=2$, $r\in[1,3]$, $\u_0\in\H$, $\mathbf{F}\in \mathrm{L}^{2}(0,T;\V')$ and Hypothesis \ref{hyp} be satisfied. Let  $$((\bar{\Omega},\bar{\mathscr{F}},\{\bar{\mathscr{F}}_t\}_{t\geq 0},\bar{\mathbb{P}}),\bar{\u},\bar{\W})$$ be a weak martingale solution for the stochastic system \eqref{32} such that 
	\begin{align}\label{5p1}
		\bar{\E}\left[\sup_{t\in[0,T]}\|\bar{\u}(t)\|_{\H}^2+  \mu \int_0^T\|\nabla\bar{\u}(t)\|_{\H}^2\d t + \alpha \int_0^T\|\bar{\u}(t)\|_{\H}^2\d t + 2\beta\int_0^T\|\bar{\u}(t)\|_{\widetilde{\L}^{r+1}}^{r+1}\d t\right]<+\infty. 
	\end{align}
Then, for $\bar{\mathbb{P}}$-almost all $\omega\in\bar{\Omega}$, the trajectory $\bar{\u}(\cdot,\omega)$ is continuous $\H$-valued function defined on $[0, T ]$. Moreover, $\bar{\u}(\cdot)$ satisfies the following It\^o formula (energy equality): 
	\begin{align}\label{5p2}
	&	\|\bar{\u}(t)\|_{\H}^2+2\mu \int_0^t\|\nabla\bar{\u}(s)\|_{\H}^2\d s +2\alpha \int_0^t\|\bar{\u}(s)\|_{\H}^2\d s +2\beta\int_0^t\|\bar{\u}(s)\|_{\widetilde{\L}^{r+1}}^{r+1}\d s\nonumber\\& = \|{\u_0}\|_{\H}^2 + \int_0^t \left\langle  \mathbf{F}(s),\bar{\u}(s) \right\rangle \d s + 2 \int_0^{t}(\boldsymbol{\sigma}(s,\bar{\u}(s))\d\bar{\W}(s),\bar{\u}(s))+  \int_0^{t}\|\boldsymbol{\sigma}(s,\bar{\u}(s))\|_{\mathcal{L}_{\Q}}^2\d s,
\end{align}
for all $t\in[0,T]$, $\bar{\mathbb{P}}$-a.s. 
\end{proposition}
\begin{proof}
	If $\bar{\u}$ is a weak martingale solution of the stochastic system \eqref{32}, then in particular $\bar{\u}\in\C([0,T];\H_w)\cap\mathrm{L}^2(0,T;\V)\cap\mathrm{L}^{r+1}(0,T;\wi\L^{r+1}),$  $\bar{\mathbb{P}}$-a.s. and 
	\begin{align}
		\bar{\u}(t)&=\u_0-\int_0^t[\mu\A\bar{\u}(s)+\B(\bar{\u}(s))+\alpha\bar{\u}(s)+\beta\mathcal{C}(\bar{\u}(s))-\mathbf{F}(s)]\d s+\int_0^t\boldsymbol{\sigma}(s,\bar{\u}(s))\d\bar{\W}(s), \ \text{ in }\ \V',
	\end{align}
since $\wi\L^{\frac{r+1}{r}}\subset\V'$ for $r\in[1,3]$. Let us consider the following  Stokes equations
\begin{align}\label{5p3}
	\bar{\y}(t)=-\mu\int_0^t\A\bar{\y}(s)\d s+\int_0^t\boldsymbol{\sigma}(s,\bar{\u}(s))\d\bar{\W}(s), 
\end{align}
in $\V'$. Since $\A:\V\to\V'$ and Hypothesis \ref{hyp} is also satisfied, by the standard existence results for the stochastic Stokes system (cf. \cite{Me,SPJZ}), we infer that the system \eqref{5p3} has a unique progressively measurable solution $\bar{\y}$ such that $\bar{\y}\in\C([0,T];\H)\cap\mathrm{L}^2(0,T;\V)$, $\bar{\mathbb{P}}$-a.s.  and
\begin{align}
	\bar{\E}\left[\sup_{t\in[0,T]}\|\bar{\y}(t)\|_{\H}^2+\mu\int_0^T\|\bar{\y}(t)\|_{\V}^2\d t\right]\leq C(K,T)<\infty. 
\end{align}
For $d=2$ and $r\in[1,3]$, an application of the Gagliardo-Nirenberg interpolation inequality yields 
\begin{align}
	\int_0^T\|\bar{\y}(t)\|_{\wi\L^{r+1}}^{r+1}\d t&\leq C\int_0^T\|\bar{\y}(t)\|_{\H}^{2}\|\bar{\y}(t)\|_{\V}^{r-1}\d t\leq CT^{\frac{3-r}{2}}\sup_{t\in[0,T]}\|\bar{\y}(t)\|_{\H}^2\left(\int_0^T\|\bar{\y}(t)\|_{\V}^2\d t\right)^{\frac{r-1}{2}}. 
\end{align}
Let us define 
\begin{align*}
	\v(t):=\bar{\u}(t)-\bar{\y}(t),\ t\in[0,T].
\end{align*}
For $\bar{\mathbb{P}}$-a.a. $\omega\in\bar{\Omega},$ the function $\v = \v(\cdot,\omega)$ is a weak solution of the following deterministic equation (cf. \cite{KKMTM-DCDSB}):
\begin{equation}\label{5p5}
	\left\{
	\begin{aligned}
		\frac{\d \v(t)}{\d  t}&=-\mu\A\v(t)-\B(\v(t)+\bar{\y}(t))- \alpha(\v(t)+\bar{\y}(t)) -\beta\mathcal{C}(\v(t)+\bar{\y}(t)) + \mathbf{F}(t),\\
		\v(0)&=\u_0. 
	\end{aligned}
\right.
\end{equation}
Let $\omega\in\bar{\Omega}$ be such that $\bar{\u}(\cdot,\omega)\in\C([0,T];\H_w)\cap\mathrm{L}^2(0,T;\V)\cap\mathrm{L}^{r+1}(0,T;\wi\L^{r+1}),$  and $\bar{\y}(\cdot,\omega)\in \C([0,T];\H)\cap\mathrm{L}^2(0,T;\V)
\cap\mathrm{L}^{r+1}(0,T;\wi\L^{r+1})$,
then we have 
$\v(\cdot,\omega)\in  
\C([0,T];\H_w)\cap\mathrm{L}^2(0,T;\V)\cap\mathrm{L}^{r+1}(0,T;\wi\L^{r+1})$.
Applying the same arguments as in 
\cite[Theorem 3.7]{KKMTM-DCDSB}, we can establish the energy equality for 
$\v$ and conclude that 
$\v\in \C([0,T];\H)$. Hence, we get
$$\bar{\u}=\v+\bar{\y}\in \C([0,T];\H).$$

We prove the  It\^o formula in the following way. Since ${\v}(\cdot)$ satisfies the energy equality (cf. \cite[Theorem 3.7]{KKMTM-DCDSB}), we have 
\begin{align}
	 \|{\v}(t)\|_{\H}^2
     = \|\u_0\|_{\H}^2-2\int_0^t\langle\mu\A{\v}(s)+\B({\v}(s)+\bar{\y}(s))+\alpha{\v}(s) +\beta\mathcal{C}({\v}(s)+\bar{\y}(s)) -\mathbf{F}(s),{\v}(s)\rangle\d s,
\end{align}
for all $t\in[0,T]$. An application of It\^o's formula to the process $\|\bar{\y}(\cdot)\|_{\H}^2$ yields, $\bar{\mathbb{P}}$-a.s.,
\begin{align}
	\|\bar{\y}(t)\|_{\H}^2&=-2\mu\int_0^t\langle\A\bar{\y}(s),\bar{\y}(s)\rangle\d s+2\int_0^t(\boldsymbol{\sigma}(s,\bar{\u}(s))\d\bar{\W}(s),\bar{\y}(s)) +\int_0^t\|\boldsymbol{\sigma}(s,\bar{\u}(s))\|_{\mathcal{L}_{\Q}}^2\d s,
\end{align}
for all $t\in[0,T]$. Using It\^o's product formula, we also have, $\bar{\mathbb{P}}$-a.s.,
\begin{align}
	({\v}(t),\bar{\y}(t))&=-\int_0^t\langle\mu\A{\v}(s)+\B({\v}(s)+\bar{\y}(s)) + \alpha ({\v}(s)+\bar{\y}(s))+\beta\mathcal{C}({\v}(s)+\bar{\y}(s)) - \mathbf{F}(s),\bar{\y}(s)\rangle\d s\nonumber\\&\quad -\mu\int_0^t\langle\A\bar{\y}(s),{\v}(s)\rangle\d s+\int_0^t(\boldsymbol{\sigma}(s,\bar{\u}(s))\d\bar{\W}(s),{\v}(s)), 
\end{align}
for all $t\in[0,T]$. Using the fact that $\bar{\u}(t)={\v}(t)+\bar{\y}(t), \ t\in[0,T],$ and combining the above equations, we obtain, $\bar{\mathbb{P}}$-a.s.,
\begin{align}
	\|\bar{\u}(t)\|_{\H}^2&=\|{\v}(t)\|_{\H}^2+\|\bar{\y}(t)\|_{\H}^2+2({\v}(t),\bar{\y}(t))\nonumber\\&=\|\u_0\|_{\H}^2-2\mu\int_0^t\langle\mu\A\bar{\u}(s)+\B(\bar{\u}(s)) + \alpha \bar{\u}(s) +\beta\mathcal{C}(\bar{\u}(s)) - \mathbf{F}(s),\bar{\u}(s)\rangle\d s
    \nonumber\\ & \quad +\int_0^t\|\boldsymbol{\sigma}(s,\bar{\u}(s))\|_{\mathcal{L}_{\Q}}^2\d s +2\int_0^t(\boldsymbol{\sigma}(s,\bar{\u}(s))\d\bar{\W}(s),\bar{\u}(s)), 
\end{align}
for all $t\in[0,T]$, which completes the proof of \eqref{5p2}. 
\end{proof}
\begin{remark}
	The above mentioned method will not work for $d=2,3$ and $r\in(3,\infty)$ (for $d=3$ case $r=3$ also),  since we need $\bar{\y}\in\mathrm{L}^{r+1}(0,T;\wi\L^{r+1})$, $\bar{\mathbb{P}}$-a.s. for the solvabililty of the system \eqref{5p5}. Under Hypothesis \ref{hyp}, it is not possible to obtain $\bar{\y}\in\mathrm{L}^{r+1}(\Omega;\mathrm{L}^{r+1}(0,T;\wi\L^{r+1}))$. 
\end{remark}

For the cases $d=2,3$ with $r\in[3,\infty)$, we provide the next result from \cite{KK+MTM-SCBF}.

\begin{proposition}\label{prop5.3}
	 Let $\u_0\in\H$, $\mathbf{F}\in \mathrm{L}^{2}(0,T;\V')$ and Hypothesis \ref{hyp} be satisfied. Let  $$((\bar{\Omega},\bar{\mathscr{F}},\{\bar{\mathscr{F}}_t\}_{t\geq 0},\bar{\mathbb{P}}),\bar{\u},\bar{\W})$$ be a martingale solution for the stochastic system \eqref{32} such that \eqref{5p1} be satisfied.
	Then, for $\bar{\mathbb{P}}$-a.a. $\omega\in\bar{\Omega}$, the trajectory $\bar{\u}(\cdot,\omega)$ is continuous $\H$-valued function defined on $[0, T ]$ and $\bar{\u}(\cdot)$ satisfies   It\^o's formula \eqref{5p2}.
\end{proposition}

\begin{proof}
    For the proof, we refer readers to {\cite[Theorem 3.10, \textbf{Step (4)}]{KK+MTM-SCBF}}.
\end{proof}

Let us now prove the pathwise uniqueness of weak martingale solutions of the stochastic system \eqref{32} for all the cases given in Table \ref{Table-3}.

\begin{proposition}\label{prop5.4}
	Let Hypothesis \ref{hyp} be satisfied. If $\bar{\u}_1$ and $\bar{\u}_2$ are two weak martingale solutions of the stochastic system \eqref{32} defined on the same filtered probability space $(\bar{\Omega},\bar{\mathscr{F}},\{\bar{\mathscr{F}}_t\}_{t\geq 0},\bar{\mathbb{P}})$ with the same initial data ($\u_0$) and same external forcing ($\mathbf{F}$), then for all the cases given in Table \ref{Table-3}, $\bar{\u}_1(t) = \bar{\u}_2(t)$ $\bar{\mathbb{P}}$-a.s. for all $t \in [0, T ]$.
\end{proposition}

\begin{proof}
	Let $\bar{\u}_1(\cdot)$ and $\bar{\u}_2(\cdot)$ be two weak martingale solutions of the stochastic system \eqref{32} defined on $(\bar{\Omega},\bar{\mathscr{F}},\{\bar{\mathscr{F}}_t\}_{t\geq 0},\bar{\mathbb{P}})$. For $N>0$, let us define 
	\begin{align*}
		\tau_N^1=\inf_{0\leq t\leq T}\Big\{t:\|\bar{\u}_1(t)\|_{\H}\geq N\Big\},\ \tau_N^2=\inf_{0\leq t\leq T}\Big\{t:\|\bar{\u}_2(t)\|_{\H}\geq N\Big\}\text{ and }\tau_N=\tau_N^1\wedge\tau_N^2.
	\end{align*}
	Then, one can show that $\tau_N\to T$ as $N\to\infty$, $\bar{\mathbb{P}}$-a.s. (cf. \cite{KK+MTM-SCBF,MTM6}).
	
	We define $\bar{\z}(\cdot)=\bar{\u}_1(\cdot)-\bar{\u}_2(\cdot)$, and $\widehat{\boldsymbol{\sigma}}(\cdot)=\boldsymbol{\sigma}(\cdot,\bar{\u}_1(\cdot))-\boldsymbol{\sigma}(\cdot,\bar{\u}_2(\cdot))$. Then, $\bar{\z}(\cdot)$ satisfies the system:
	\begin{equation}
		\left\{
		\begin{aligned}
			\d\bar{\z}(t)&=-\left[\mu\A\bar{\z}(t)+\B(\bar{\u}_1(t))-\B(\bar{\u}_2(t)) + \alpha\bar{\z}(t) +\beta(\mathcal{C}(\bar{\u}_1(t))-\mathcal{C}(\bar{\u}_2(t)))\right]\d t +\widehat{\boldsymbol{\sigma}}(t)\d\bar{\W}(t),\\
			\bar{\z}(0)&=\boldsymbol{0},
		\end{aligned}
		\right.
	\end{equation}
for a.e. $t\in[0,T]$ in $\V'+\wi\L^{\frac{r+1}{r}}$. 
\vskip 0.2 cm
\noindent\textbf{Case (1):} {$d=2$ and $r\in[1,3]$.}
Let us first consider the case $d=2$ and $r\in[1,3]$.	We apply It\^o's formula  to the process $\|\bar{\z}(\cdot)\|_{\H}^2$ to find, $\bar{\mathbb{P}}$-a.s.,
	\begin{align}\label{5p13}
	&	\|\bar{\z}(t)\|_{\H}^2+2\mu\int_0^t\|\nabla\bar{\z}(s)\|_{\H}^2\d s +2\alpha\int_0^t\|\bar{\z}(s)\|_{\H}^2\d s
    \nonumber\\ &=\|\bar{\z}(0)\|_{\H}^2 -2\int_0^t\langle\B(\bar{\u}_1(s))-\B(\bar{\u}_2(s)),\bar{\z}(s)\rangle\d s - 2\beta\int_0^t\langle\mathcal{C}(\bar{\u}_1(s))-\mathcal{C}(\bar{\u}_2(s)),\bar{\z}(s)\rangle\d s\nonumber\\&\quad +\int_0^{t} \|\widehat{\boldsymbol{\sigma}}(s)\|^2_{\mathcal{L}_{\Q}}\d
		s + 2\int_0^{t}\left(\widehat{\boldsymbol{\sigma}}(s)\d\W(s),\bar{\z}(s)\right),
	\end{align}
	where $\int_0^{t}\left(\widehat{\boldsymbol{\sigma}}(s)\d\W(s),\bar{\z}(s)\right)$ is a local martingale.
Using H\"older's, Ladyzhenskaya's  and Young's inequalities, we estimate $\langle\B(\bar{\u}_1)-\B(\bar{\u}_2),\bar{\z}\rangle$ as 
\begin{align}\label{5p14}
	\langle\B(\bar{\u}_1)-\B(\bar{\u}_2),\bar{\z}\rangle&=\langle\B(\bar{\z},\bar{\u}_2),\bar{\z}\rangle=-\langle\B(\bar{\z},\bar{\z}),\bar{\u}_2\rangle\leq\|\nabla\bar{\z}\|_{\H}\|\bar{\z}\|_{\wi\L^4}\|\bar{\u}_2\|_{\wi\L^4}\nonumber\\&\leq 2^{1/4}\|\nabla\bar{\z}\|_{\H}^{3/2}\|\bar{\z}\|_{\H}^{1/2}\|\bar{\u}_2\|_{\wi\L^4}\leq\frac{\mu}{2}\|\nabla\bar{\z}\|_{\H}^2 + \frac{27}{16\mu^3}\|\bar{\u}_2\|_{\wi\L^4}^4\|\bar{\z}\|_{\H}^2. 
\end{align}
	Applying (\ref{2.23}) and \eqref{5p14} in \eqref{5p13}, we get,  $\bar{\mathbb{P}}$-a.s.,
	\begin{align}\label{5p15}
		&\|\bar{\z}(t)\|_{\H}^2 + \mu\int_0^t\|\nabla\bar{\z}(s)\|_{\H}^2\d s + 2\alpha \int_0^t\|\bar{\z}(s)\|_{\H}^2\d s +\frac{\beta}{2^{r-2}}\int_0^t\|\bar{\z}(s)\|_{\wi\L^{r+1}}^{r+1}\d s\nonumber\\&\leq\|\bar{\z}(0)\|_{\H}^2  +\frac{27}{8\mu^3}\int_0^t\|\bar{\u}_2(s)\|_{\L^4}^4\|\bar{\z}(s)\|_{\H}^2\d s +\int_0^{t} \|\widehat{\boldsymbol{\sigma}}(s)\|^2_{\mathcal{L}_{\Q}}\d
		s + 2\int_0^{t}\left(\widehat{\boldsymbol{\sigma}}(s)\d\W(s),\bar{\z}(s)\right).
	\end{align}
	Let us apply  It\^o's formula to the process $e^{-\varrho(t)}\|\bar{\z}(t)\|_{\H}^2$, where 
	\begin{align*}
		\varrho(t)=\frac{27}{8\mu^3}\int_0^t\|\bar{\u}_2(s)\|_{\L^4}^4\d s\ \text{ so that }\ \varrho'(t)= \frac{27}{8\mu^3}\|\bar{\u}_2(t)\|_{\L^4}^4, \ \text{ for a.e. } t,
	\end{align*} to obtain, $\bar{\mathbb{P}}$-a.s.,
	\begin{align}\label{5p16}
		&e^{-\varrho(\s)}\|\bar{\z}(\t)\|_{\H}^2+\mu\int_0^{\s}e^{-\varrho(s)}\|\nabla\bar{\z}(s)\|_{\H}^2\d s 
        +2\alpha \int_0^{\s} e^{-\varrho(s)}\|\bar{\z}(s)\|_{\H}^2\d s 
       \nonumber\\ & \quad+\frac{\beta}{2^{r-2}}\int_0^{\t}e^{-\varrho(s)}\|\bar{\z}(s)\|_{\wi\L^{r+1}}^{r+1}\d s\nonumber\\&\leq \|\bar{\z}(0)\|_{\H}^2+\int_0^{\s}e^{-\varrho(s)} \|\widehat{\boldsymbol{\sigma}}(s)\|^2_{\mathcal{L}_{\Q}}\d
		s + 2\int_0^{\s} e^{-\varrho(s)} \left(\widehat{\boldsymbol{\sigma}}(s)\d\W(s),\bar{\z}(s)\right),
	\end{align}
	where $\int_0^{\s} e^{-\varrho(s)} \left(\widehat{\boldsymbol{\sigma}}(s)\d\W(s),\bar{\z}(s)\right)$ is a local martingale. 
	Taking expectation in \eqref{5p16}, and  then using Hypothesis \ref{hyp} (H.3), we deduce
	\begin{align}\label{5p17}
		&\bar{\E}\bigg[e^{-\varrho(\s)}\|\bar{\z}(\s)\|_{\H}^2+\mu\int_0^{\s}e^{-\varrho(s)}\|\nabla\bar{\z}(s)\|_{\H}^2\d s
        +2\alpha\int_0^{\s}e^{-\varrho(s)}\|\bar{\z}(s)\|_{\H}^2\d s 
        \nonumber\\ & \quad+\frac{\beta}{2^{r-2}}\int_0^{\s}e^{-\varrho(s)}\|\bar{\z}(s)\|_{\wi\L^{r+1}}^{r+1}\d s\bigg]\nonumber\\&\leq \bar{\E}\left[\|\bar{\z}(0)\|_{\H}^2\right]+L\bar{\E}\left[\int_0^{\s}e^{-\varrho(s)}\|\bar{\z}(s)\|_{\H}^2\d s\right].
	\end{align}
	An application of Gr\"onwall's inequality in (\ref{5p17}) yields 
	\begin{align}\label{5p18}
		&\bar{\E}\left[e^{-\varrho(\s)}\|\bar{\z}(\s)\|_{\H}^2\right]\leq \bar{\E}\left[\|\bar{\z}(0)\|_{\H}^2\right]e^{LT}.
	\end{align}
	Thus the initial data  $\bar{\u}_1(0)=\bar{\u}_2(0)={\u}_0$ leads to $\bar{\z}(\s)=0$, $\bar{\mathbb{P}}$-a.s. But the fact that $\tau_N\to T$ as $N\to\infty$ provide $\bar{\z}(t)=0$ and hence $\bar{\u}_1(t) = \bar{\u}_2(t)$ for all $t \in[0, T ]$, $\bar{\mathbb{P}}$-a.s., and the pathwise uniqueness follows.
	
	\vskip 0.2 cm
	\noindent\textbf{Case (2):} {$d=2,3$ and $r\in(3,\infty)$.}
Let us now consider the case $d=2,3$ and $r\in(3,\infty)$.	Using H\"older's and Young's inequalities, we estimate $\langle\B(\bar{\u}_1)-\B(\bar{\u}_2),\bar{\z}\rangle=-\B(\bar{\z},\bar{\z}),\bar{\u}_2\rangle$ as  
	\begin{align}\label{5p19}
		|\langle\B(\bar{\z},\bar{\z}),\bar{\u}_2\rangle|&\leq\|\nabla\bar{\z}\|_{\H}\|\bar{\u}_2\bar{\z}\|_{\H}\leq\frac{\mu }{2}\|\nabla\bar{\z}\|_{\H}^2+\frac{1}{2\mu }\||\bar{\u}_2|\bar{\z}\|_{\H}^2.
	\end{align}
	Taking the term $\||\bar{\u}_2|\bar{\z}\|_{\H}^2$ from \eqref{5p19} and using H\"older's and Young's inequalities, we estimate
	\begin{align}\label{5p20}
		&\int_{\mathcal{O}}|\bar{\u}_2(x)|^2|\bar{\z}(x)|^2\d x=\int_{\mathcal{O}}|\bar{\u}_2(x)|^2|\bar{\z}(x)|^{\frac{4}{r-1}}|\bar{\z}(x)|^{\frac{2(r-3)}{r-1}}\d x\nonumber\\&\leq\left(\int_{\mathcal{O}}|\bar{\u}_2(x)|^{r-1}|\bar{\z}(x)|^2\d x\right)^{\frac{2}{r-1}}\left(\int_{\mathcal{O}}|\bar{\z}(x)|^2\d x\right)^{\frac{r-3}{r-1}}\nonumber\\&\leq\frac{\beta\mu }{2}\left(\int_{\mathcal{O}}|\bar{\u}_2(x)|^{r-1}|\bar{\z}(x)|^2\d x\right)+\frac{r-3}{r-1}\left(\frac{4}{\beta\mu (r-1)}\right)^{\frac{2}{r-3}}\left(\int_{\mathcal{O}}|\bar{\z}(x)|^2\d x\right),
	\end{align}
	for $r>3$. Using \eqref{5p20} in \eqref{5p19}, we deduce
	\begin{align}\label{5p21}
		|\langle\B(\bar{\z},\bar{\z}),\bar{\u}_2\rangle|&\leq\frac{\mu }{2}\|\nabla\bar{\z}\|_{\H}^2+\frac{\beta}{4}\||\bar{\u}_2|^{\frac{r-1}{2}}\bar{\z}\|_{\H}^2+\hat{\zeta}\|\bar{\z}\|_{\H}^2,
	\end{align}
where 
\begin{align}\label{eqn-zeta}
    \hat{\zeta}=\frac{r-3}{2\mu(r-1)}\left(\frac{4}{\beta\mu (r-1)}\right)^{\frac{2}{r-3}}.
\end{align}
From \eqref{2.23}, we deduce
	\begin{align}\label{5p2121}
		\beta	\langle\mathcal{C}(\bar{\u}_1)-\mathcal{C}(\bar{\u}_2),\bar{\z}\rangle\geq \frac{\beta}{2}\||\bar{\u}_1|^{\frac{r-1}{2}}\bar{\z}\|_{\H}^2+\frac{\beta}{2}\||\bar{\u}_2|^{\frac{r-1}{2}}\bar{\z}\|_{\H}^2.
	\end{align}
	Thus, using the above two estimates and \eqref{2.23}  in   \eqref{5p13}, we infer, $\bar{\mathbb{P}}$-a.s., 
	\begin{align}\label{5p22}
		&\|\bar{\z}(\s)\|_{\H}^2 +\mu \int_0^{\s}\|\nabla\bar{\z}(s)\|_{\H}^2\d s
        +2\alpha \int_0^{\s}\|\bar{\z}(s)\|_{\H}^2\d s +\frac{\beta}{2^r}\int_0^{\s}\|\bar{\z}(s)\|_{\wi\L^{r+1}}^{r+1}\d s\nonumber\\&\leq\|\bar{\z}(0)\|_{\H}^2 +2\hat{\zeta}\int_0^{\s}\|\bar{\z}(s)\|_{\H}^2\d s +\int_0^{t} \|\widehat{\boldsymbol{\sigma}}(s)\|^2_{\mathcal{L}_{\Q}}\d
		s + \int_0^{t}\left(\widehat{\boldsymbol{\sigma}}(s)\d\W(s),\bar{\z}(s)\right).
	\end{align}
Taking expectation in (\ref{5p22}), and  then using Hypothesis \ref{hyp} (H.3), we obtain  
	\begin{align}\label{5p23}
		&\bar{\E}\left[\|\bar{\z}(\s)\|_{\H}^2+\mu \int_0^{\s}\|\nabla\bar{\z}(s)\|_{\H}^2\d s
        +2\alpha \int_0^{\s}\|\bar{\z}(s)\|_{\H}^2\d s
        +\frac{\beta}{2^r}\int_0^{\s}\|\bar{\z}(s)\|_{\wi\L^{r+1}}^{r+1}\d s\right]\nonumber\\&\leq  \bar{\E}\left[\|\bar{\z}(0)\|_{\H}^2\right]+(L+2\hat{\zeta})\bar{\E}\left[\int_0^{\s}\|\bar{\z}(s)\|_{\H}^2\d s\right].
	\end{align}
	Applying  Gr\"onwall's inequality in (\ref{5p23}),  we arrive at 
	\begin{align}\label{5p24}
		&\bar{\E}\left[\|\bar{\z}(\s)\|_{\H}^2\right]\leq \bar{\E}\left[\|\bar{\z}(0)\|_{\H}^2\right]e^{(L+2\hat{\zeta})T}.
	\end{align}
	Thus the initial data  $\bar{\u}_1(0)=\bar{\u}_2(0)={\u}_0$ leads to $\bar{\z}(\s)=0$, $\bar{\mathbb{P}}$-a.s. Using the fact that $\tau_N\to T$, $\bar{\mathbb{P}}$-a.s., implies $\bar{\z}(t)=0$ and hence $\bar{\u}_1(t) = \bar{\u}_2(t)$, $\bar{\mathbb{P}}$-a.s., for all $t \in[0, T ]$, and the pathwise uniqueness follows. 
	
	\vskip 0.2 cm
	\noindent\textbf{Case (3):} {$d=r=3$ with $2\beta\mu\geq 1$.}
	For this case, the estimate 
	\begin{align}\label{5p25}
		|\langle\B(\bar{\z},\bar{\z}),\bar{\u}_2\rangle|&\leq\||\bar{\u}_2|\bar{\z}\|_{\H}\|\nabla\bar{\z}\|_{\H} \leq\theta\mu \|\nabla\bar{\z}\|_{\H}^2+\frac{1}{4\mu\theta}\||\bar{\u}_2|\bar{\z}\|_{\H}^2,
	\end{align}
for some $0<\theta\leq 1$,	helps us to obtain 
	\begin{align}\label{5p60}
		&\bar{\E}\bigg[\|\bar{\z}(\s)\|_{\H}^2+2\mu(1-\theta) \int_0^{\s}\|\nabla\bar{\z}(s)\|_{\H}^2\d s
        +2\alpha \int_0^{\s}\|\bar{\z}(s)\|_{\H}^2\d s 
        \nonumber\\ &\quad +\left(\beta-\frac{1}{2\mu\theta}\right)\int_0^{\s}\|\bar{\z}(s)\|_{\wi\L^{r+1}}^{r+1}\d s\bigg]\nonumber\\&\leq  \bar{\E}\left[\|\bar{\z}(0)\|_{\H}^2\right]+L\bar{\E}\left[\int_0^{\s}\|\bar{\z}(s)\|_{\H}^2\d s\right],
	\end{align}
	and the pathwise uniqueness follows. 
\end{proof}

\begin{proof}[Proof of Theorem \ref{thm3.5}]
	By Theorem \ref{thm3.4},  we infer the existence of a martingale solution and by Proposition \ref{prop5.4}, we know that the weak solutions are pathwise unique. Therefore, assertion (2) of  Theorem \ref{thm3.5} follow from \cite[Theorems 2]{OMa} and  \cite[Theorem 8]{HZa}. Assertion (1) is a direct consequence of Propositions \ref{prop5.1} and \ref{prop5.3}.
\end{proof}

\section{Invariant Measures and Ergodicity}\label{sec6}\setcounter{equation}{0}

In this section, we aim to establish the existence and uniqueness of invariant probability measures associated with the stochastic system \eqref{31} in unbounded domains under Hypothesis \ref{hyp} on the noise coefficient $\boldsymbol{\sigma}(\cdot,\cdot)$.

\subsection{Existence of at least one invariant measure} 
In this subsection, we prove that there exists at least one invariant probability  measure associated with the stochastic system \eqref{32}. 

For any bounded Borel function $\psi\in \mathcal{B}_b(\H)$ and 
$t\geq 0$, we define
\begin{align}\label{eqn-semigroup}
    (\mathrm{T}_t\psi)(\u_0) = \E \big[ \psi(\u(t,\u_0)) \big], \;\;\; \u_0\in\H,
\end{align}
where $\u(\cdot,\u_0)$ is the pathwise unique strong solution to the stochastic system \eqref{32} on $[0,\infty)$ corresponding to initial data $\u_0$.

By Propositions~\ref{prop5.1} and \ref{prop5.3}, the trajectories $\u(\cdot,\u_0)$ belong to $\C([0,T]; \mathbb{H})$. Consequently, the family $\{\mathrm{T}_t\}_{t\geq0}$ forms a stochastically continuous semigroup on the Banach space $\C_b(\H)$. In particular, for $t>0$, for every $\psi\in\C_b(\H)$ and every $\u_0\in\H$,
\begin{align*}
    \lim_{t\to0}\mathrm{T}_t \psi(\u_0)= \psi(\u_0).
\end{align*}

As a direct consequence of Theorem~\ref{thm3.5}, we obtain the following result, whose proof follows the same arguments as in \cite[Section 9.2]{DaZ}. Since the reasoning is standard, we omit the details here.
\begin{proposition}
   Let $\mathbf{F}\in \mathrm{L}^{2}(0,T;\V')$  and Hypothesis \ref{hyp} hold. Then, the family $\{\mathrm{T}_t\}_{t\geq0}$ is Markov, that is, $\mathrm{T}_{t+s}= \mathrm{T}_{t}\mathrm{T}_{s}$ for all $t,s\geq 0$.
\end{proposition}

\begin{definition}\label{def-bw-feller}
   A semigroup $\{\mathrm{T}_t\}_{t \geq 0}$ is said to be \emph{$bw$-Feller} if, for every bounded sequentially weakly continuous function 
$\psi : \H \to \R$ and every $t > 0$, the function $\mathrm{T}_t \psi : \H \to \R$ is also bounded and sequentially weakly continuous. 

Equivalently, $\{\mathrm{T}_t\}_{t \geq 0}$ is $bw$-Feller if   for each $t > 0$,  whenever $\u_{0,m} \rightharpoonup \u_0$ weakly in $\H$, it holds that
\begin{align}
\mathrm{T}_t \psi(\u_{0,m}) \to \mathrm{T}_t \psi(\u_0),
\end{align}
for all bounded sequentially weakly continuous functions $\psi : \H \to \R$.
\end{definition}

Let us now recall an abstract result from the article \cite{Maslowski+Seidler_1999}, which will be used to demonstrate the existence of an invariant measure for the stochastic system \eqref{32}.

\begin{theorem}[{\cite[Proposition 3.1]{Maslowski+Seidler_1999}}]\label{thm-invaeiant-measure}
    Assume that the semigroup $\{\mathrm{T}_{t}\}_{t\geq 0}$ is $bw$-weakly Feller. Suppose that we can find a Borel probability measure $\nu$ on $\H$ and $T_0\geq 0$ such that for any $\varepsilon>0$, there exists $R>0$ satisfying
    \begin{align}
        \sup_{T\geq T_0}\frac{1}{T}\int_0^T \mathrm{T}^*_{t}\nu \big(\H\setminus \mathcal{B}_{R}\big)\d t < \varepsilon,
    \end{align}
where $\mathcal{B}_R=\{\phi\in\H : \|\phi\|_{\H}\leq R\}$. Then, there exists an invariant probability measure for the semigroup $\{\mathrm{T}_{t}\}_{t\geq 0}$.
\end{theorem}

In order to demonstrate that the semigroup $\{\mathrm{T}_t\}_{t\geq0}$ is $bw$-Feller, we require the following result:

\begin{lemma}\label{lem-solution-convergence}
Let $\u_0\in\H$, $\mathbf{F}\in \mathrm{L}^{2}(0,T;\V')$ and Hypothesis \ref{hyp} be satisfied. Assume that an $\H$-valued sequence $\{\u_{0,m}\}_{m\in\N}$ weakly converges to $\u_0$ in $\H$. 
Let $\u^m$ be the pathwise strong solution of the stochastic system \eqref{32} on $[0,\infty)$ corresponding to initial data $\u_{0,m}$. Then, for every $T>0$, there exist:
    \begin{itemize}
        \item a subsequence $\{m_k\}_{k\in\N}$,
        \item a stochastic basis $(\widehat{\Omega},\widehat{\mathscr{F}},\{\widehat{\mathscr{F}}_t\}_{t\in[0,T]},\widehat{\mathbb{P}})$,
        \item a $\mathrm{Q}$-Wiener process $\{\widehat{\mathrm{W}}(t)\}_{t\in[0,T]}$ defined on this basis,
    \item and $\{\widehat{\mathscr{F}}_t\}_{t\in[0,T]}$-progressively measurable processes $\widehat{\u}$, $\{\widehat{\u}^{m_k}\}_{k\in\N}$, $t\in[0,T]$ (defined on the same basis) with laws supported in $\mathscr{Y}$ such that 
    \begin{align}
        \widehat{\u}^{m_k} \text{ has the same law as } \u^{m_k} \text{ on } \mathscr{Y} \text{ and } \widehat{\u}^{m_k}\to \widehat{\u} \text{ in } \mathscr{Y}, \;\; \widehat{\mathbb{P}}\text{-a.s.},
    \end{align}
    and the system 
    \begin{align}
        (\widehat{\Omega},\widehat{\mathscr{F}},\{\widehat{\mathscr{F}}_t\}_{t\geq 0},\widehat{\mathbb{P}}, \widehat{\W},\widehat{\u})
    \end{align}
    is a martingale solution to the stochastic system \eqref{32} on $[0,T]$ with the initial data $\u_0$ and external forcing $\mathbf{F}$.  
    \end{itemize}
\end{lemma}
\begin{proof}
In view of \cite[Lemma 27.2]{Robinson_2020}, there exists $R_1>0$  such that $\sup\limits_{m\in\N}\|\u_{0,m}\|_{\H}\leq R_1$. 
Theorem \ref{thm3.5} implies that for every stochastic basis $(\Omega,\mathscr{F},\{\mathscr{F}_t\}_{t\geq 0},\mathbb{P})$ and $\Q$-Wiener processes $\{\W(t)\}_{t\geq 0}$ defined on this stochastic basis,  there exists a unique progressively measurable process $\u^m :[0,T]\times\Omega\to\H$ with $\mathbb{P}$-a.s. paths
		\begin{align*}
\u^m(\cdot,\omega)\in\C([0,T];\H)\cap\mathrm{L}^2(0,T;\V)\cap\mathrm{L}^{r+1}(0,T;\wi\L^{r+1})
		\end{align*}
		such that for all $t \in[0,T]$  and all $\v\in\V\cap\wi\L^{r+1}$, $\mathbb{P}\text{-a.s.}$:
		\begin{align}
			 (\u^m(t),\v) 
              & =(\u_{0,m},\v)-\int_0^t\langle\mu \A\u^m(s)+\B(\u^m(s)) + \alpha \u^m(s) +\beta\mathcal{C}(\u^m(s)) - \mathbf{F}_m(s),\v\rangle\d s 
              \nonumber\\ & \quad + \int_0^t\left(\boldsymbol{\sigma}(s,\u^m(s))\d\W(s),\v\right). 
		\end{align}
Also, we have that $\u_m$, for each $m$, satisfies
\begin{align*}
		&  \E \left[   \sup_{0\leq t \leq T} \|\u^m(t)\|^{2}_{\H}+ \int_{0}^{T}\left(\mu \|\nabla\u^m(s)\|^2_{\H}+ \alpha \|\u^m(s)\|^2_{\H}+2\beta\|\u^m(s)\|^{r+1}_{\wi\L^{r+1}}\right)\d s\right]\nonumber\\&
		\leq  C\left( \|\u_{0,m}\|_{\H}, \|\mathbf{F}\|_{\mathrm{L}^{2}(0,T;\V^{\prime})}, \mu, \alpha, K ,T\right) \leq C\left( R_1, \|\mathbf{F}\|_{\mathrm{L}^{2}(0,T;\V^{\prime})}, \mu, \alpha , K ,T\right) .
	\end{align*}
    Following the same lines as in the proof of Lemma \ref{lem4.2}, we obtain that the family of measures $\{\mathscr{L}(\u^m):n\in\N\}$ is tight on $(\mathscr{Y},\mathcal{T})$.  The remainder of the proof proceeds exactly as in Theorem \ref{thm3.4}, making use of Theorem \ref{thm-JvST}, Corollary \ref{cor-convergence-B+C}, and the martingale representation theorem \cite[Section 8.2]{DaZ}.
\end{proof}

In the following result, we demonstrate that the semigroup $\{\mathrm{T}_t\}_{t\geq 0}$ is $bw$-Feller.
\begin{proposition}\label{prop-bw-feller}
   The semigroup $\{\mathrm{T}_t\}_{t\geq 0}$ is $bw$-Feller.
\end{proposition}
\begin{proof}
    Fix $t > 0$, $\u_0 \in \H$, and let $\{\u_{0,m}\}_{m \in \N}$ be a sequence in $\H$ such that $\u_{0,m} \rightharpoonup \u_0$ weakly in $\H$. 
Let $\psi : \H \to \R$ be a bounded sequentially weakly continuous function, and choose an auxiliary time $T \in (t, \infty)$. Since the function $\mathrm{T}_{t}\psi : \H \to \R$ is clearly bounded, it remains to show that it is sequentially weakly continuous. 

Let $\u^m(\cdot):=\u(\cdot,\u_{0,m})$ and $\u(\cdot):=\u(\cdot,\u_{0})$ be two strong solutions of the stochastic system \eqref{32} on $[0,\infty)$ with initial data $\u_{0,m}$ and $\u_0$, respectively. Let us suppose that these processes are defined on the stochastic basis $(\Omega,\mathscr{F},\{\mathscr{F}_t\}_{t\geq 0},\mathbb{P})$. By Lemma \ref{lem-solution-convergence}, for every $T>0$, there exist:
    \begin{itemize}
        \item a subsequence $\{m_k\}_{k\in\N}$,
        \item a stochastic basis $(\widehat{\Omega},\widehat{\mathscr{F}},\{\widehat{\mathscr{F}}_t\}_{t\in[0,T]},\widehat{\mathbb{P}})$,
        \item a $\mathrm{Q}$-Wiener process $\{\widehat{\mathrm{W}}(t)\}_{t\in[0,T]}$ defined on this basis,
    \item and $\{\widehat{\mathscr{F}}_t\}_{t\in[0,T]}$-progressively measurable processes $\widehat{\u}$, $\{\widehat{\u}^{m_k}\}_{k\in\N}$, $t\in[0,T]$ (defined on the same basis) with laws supported in $\mathscr{Y}$ such that 
    \begin{align}\label{eqn-conver-Y}
        \widehat{\u}^{m_k} \text{ has the same law as } \u^{m_k} \text{ on } \mathscr{Y} \text{ and } \widehat{\u}^{m_k}\to \widehat{\u} \text{ in } \mathscr{Y}, \;\; \widehat{\mathbb{P}}\text{-a.s.}
    \end{align}
    and the system 
    \begin{align}
        (\widehat{\Omega},\widehat{\mathscr{F}},\{\widehat{\mathscr{F}}_t\}_{t\geq 0},\widehat{\mathbb{P}}, \widehat{\W},\widehat{\u})
    \end{align}
    is a martingale solution to the stochastic system \eqref{32} on $[0,T]$ with the initial data $\u_0$ and external forcing $\mathbf{F}$.  
    \end{itemize}
    From \eqref{eqn-conver-Y}, we infer that for all $t\in[0,T]$
    \begin{align*}
        \widehat{\u}^{m_k}(t) \to \widehat{\u}(t) \;\; \text{ weakly in } \H, \;\; \; \widehat{\mathbb{P}}\text{-a.s.}
    \end{align*}
    Since $\psi:\H\to\R$ is sequentially weakly continuous, we have  for all $t\in[0,T]$
    \begin{align*}
        \psi(\widehat{\u}^{m_k}(t)) \to \psi(\widehat{\u}(t))\;\; \text{ in } \R, \;\; \; \widehat{\mathbb{P}}\text{-a.s.}
    \end{align*}
    Consequently, the boundedness of $\psi:\H\to\R$ allows us to invoke the Lebesgue dominated convergence theorem, from which we deduce
    \begin{align}\label{eqn-634}
    \lim_{k\to\infty}\widehat{\E}\left[\psi(\widehat{\u}^{m_k}(t))\right] = \widehat{\E}\left[\psi(\widehat{\u}(t))\right],  \;\; \text{ for all } \; t\in[0,T].
    \end{align}
    Since $\widehat{\u}^{m_k}$ and $\u^{m_k}$ have the same law on $\mathscr{Y}$, we achieve
    \begin{align}\label{eqn-635}
        \widehat{\E}\left[\psi(\widehat{\u}^{m_k}(t))\right]= {\E}\left[\psi({\u}^{m_k}(t))\right] = \mathrm{T}_t\psi(\u_{0,m_k}),  \;\; \text{ for all } \; t\in[0,T].
    \end{align}

Since, by assumption, $(\Omega, \mathscr{F}, \{\mathscr{F}_t\}_{t \ge 0}, \mathbb{P}, \W, \u)$ is a martingale solution of the stochastic system \eqref{32} with the initial data $\u_0$, and $(\widehat{\Omega}, \widehat{\mathscr{F}}, \{\widehat{\mathscr{F}}_t\}_{t \ge 0}, \widehat{\mathbb{P}}, \widehat{\W}, \widehat{\u})$ is likewise a martingale solution of the same system with the same initial data $\u_0$, then, by the uniqueness in law of martingale solutions to the stochastic system \eqref{32} (Theorem \ref{thm3.5}), we conclude that the processes $\u$ and $\widehat{\u}$ share the same law on the space $\mathscr{Y}$. Therefore,
\begin{align}\label{eqn-636}
    \widehat{\E}\left[\psi(\widehat{\u}(t))\right]= {\E}\left[\psi({\u}(t))\right] = \mathrm{T}_t\psi(\u_{0}).
\end{align}
Now, combining \eqref{eqn-634}, \eqref{eqn-635} and \eqref{eqn-636}, we find 
\begin{align*}
    \lim_{k\to\infty}\mathrm{T}_t\psi(\u_{0,m_k})=\mathrm{T}_t\psi(\u_{0}).
\end{align*}
By a contradiction argument (since the martingale solutions are unique in law), we conclude that the whole sequence $\{\mathrm{T}_t\psi(\u_{0,m})\}_{m\in\N}$ is convergent and we have 
\begin{align*}
\lim_{m\to\infty}\mathrm{T}_t\psi(\u_{0,m})=\mathrm{T}_t\psi(\u_{0}), \;\; \text{ for all } \; t\in[0,T].
\end{align*}
Hence the proof is completed.
\end{proof}

\begin{lemma}\label{lem-suffi-cond}
    Assume that $\u_0\in\H$, $\mathbf{F}\in \mathrm{L}^{2}(0,T;\V')$, $\alpha>\frac{L}{2}$ and  $\{\u(t)\}_{t\geq0}$ be the unique strong solution to the stochastic system \eqref{32} corresponding to the initial data $\u_0$ and external forcing $\mathbf{F}$. Then, we can find $T_0\geq 0$ such that for every $\varepsilon>0$ there exists $R>0$ satisfying 
    \begin{align}
        \sup_{T\geq T_0} \frac{1}{T} \int_0^T (\mathrm{T}^*_t\delta_{\u_0})(\H\setminus \mathcal{B}_R) \d t \leq \varepsilon,
    \end{align}
    where $\mathcal{B}_R=\{\phi\in\H : \|\phi\|_{\H}\leq R\}$.
\end{lemma}

\begin{proof}
Let us first recall that we assume $\alpha>\frac{L}{2}$.    From the It\^o formula satisfied by the solution of the stochastic system \eqref{32}, we have, $\mathbb{P}$-a.s.
    \begin{align}
	&	\|{\u}(t)\|_{\H}^2+2\mu \int_0^t\|\nabla{\u}(s)\|_{\H}^2\d s +2\alpha \int_0^t\|{\u}(s)\|_{\H}^2\d s +2\beta\int_0^t\|{\u}(s)\|_{\widetilde{\L}^{r+1}}^{r+1}\d s\nonumber\\& = \|{\u_0}\|_{\H}^2 + \int_0^t \left\langle  \mathbf{F}(s),{\u}(s) \right\rangle \d s + 2 \int_0^{t}(\boldsymbol{\sigma}(s,{\u}(s))\d {\W}(s), {\u}(s))+  \int_0^{t}\|\boldsymbol{\sigma}(s,{\u}(s))\|_{\mathcal{L}_{\Q}}^2\d s.
\end{align}
Since $\int_0^{t}(\boldsymbol{\sigma}(s,{\u}(s))\d {\W}(s), {\u}(s))$ is a martingale, after taking expectation, we get
\begin{align}
	&	\E \left[\|{\u}(t)\|_{\H}^2\right]+2\mu \E \left[\int_0^t\|\nabla{\u}(s)\|_{\H}^2\d s\right] +2\alpha \E \left[\int_0^t\|{\u}(s)\|_{\H}^2\d s\right] +2\beta\E \left[\int_0^t\|{\u}(s)\|_{\widetilde{\L}^{r+1}}^{r+1}\d s\right] \nonumber\\& = \|{\u_0}\|_{\H}^2 + \E \left[\int_0^t \left\langle  \mathbf{F}(s),{\u}(s) \right\rangle \d s \right] + \E \left[ \int_0^{t}\|\boldsymbol{\sigma}(s,{\u}(s))\|_{\mathcal{L}_{\Q}}^2\d s\right],
\end{align}
for all $t\in[0,T]$. Using H\"older's and Young's inequalities, and Hypothesis \ref{hyp}, we obtain
\begin{align*}
    &	\E \left[\|{\u}(t)\|_{\H}^2\right]+2\mu \E \left[\int_0^t\|\nabla{\u}(s)\|_{\H}^2\d s\right] +2\alpha \E \left[\int_0^t\|{\u}(s)\|_{\H}^2\d s\right] +2\beta\E \left[\int_0^t\|{\u}(s)\|_{\widetilde{\L}^{r+1}}^{r+1}\d s\right] \nonumber\\& \leq  \|{\u_0}\|_{\H}^2 +  \frac{\|\mathbf{F}\|^2_{\mathrm{L}^2(0,T;\V^{\prime})}}{4\min\left\{\mu,\frac12\left(\alpha-\frac{L}{2}\right)\right\}} + \min\left\{\mu,\frac12\left(\alpha-\frac{L}{2}\right)\right\} \E \left[\int_0^t\|{\u}(s)\|_{\V}^2\d s\right] 
    \nonumber\\ & \quad +  \E \left[ \int_0^{t}\|\boldsymbol{\sigma}(s,{\u}(s))-\boldsymbol{\sigma}(s,\boldsymbol{0})\|_{\mathcal{L}_{\Q}}^2\d s\right] +  \E \left[ \int_0^{t}\|\boldsymbol{\sigma}(s,\boldsymbol{0})\|_{\mathcal{L}_{\Q}}^2\d s\right] 
    \nonumber\\ & \quad + 2 \E \left[ \int_0^{t}\|\boldsymbol{\sigma}(s,{\u}(s))-\boldsymbol{\sigma}(s,\boldsymbol{0})\|_{\mathcal{L}_{\Q}}\|\boldsymbol{\sigma}(s,\boldsymbol{0})\|_{\mathcal{L}_{\Q}}\d s\right] 
    \nonumber\\& \leq  \|{\u_0}\|_{\H}^2 +  \frac{\|\mathbf{F}\|^2_{\mathrm{L}^2(0,T;\V^{\prime})}}{4\min\left\{\mu,\frac12\left(\alpha-\frac{L}{2}\right)\right\}} + \min\left\{\mu,\frac12\left(\alpha-\frac{L}{2}\right)\right\} \E \left[\int_0^t\|{\u}(s)\|_{\V}^2\d s\right] 
    \nonumber\\ & \quad + L \E \left[ \int_0^{t}\|{\u}(s)\|_{\H}^2\d s\right] + Kt 
    + 2\sqrt{K}\sqrt{L} \;  \E \left[ \int_0^{t}\|{\u}(s)\|_{\H}\d s\right] 
    \nonumber\\& \leq  \|{\u_0}\|_{\H}^2 +  \frac{\|\mathbf{F}\|^2_{\mathrm{L}^2(0,T;\V^{\prime})}}{4\min\left\{\mu,\frac12(\alpha-\frac{L}{2})\right\}} + \mu \E \left[\int_0^t\|\nabla{\u}(s)\|_{\H}^2\d s\right]  + \left(\alpha+\frac{L}{2}\right) \E \left[\int_0^t\|{\u}(s)\|_{\H}^2\d s\right] 
    \nonumber\\ & \quad + \left\{1 + \frac{2L}{\alpha-\frac{L}{2}}\right\}Kt,
\end{align*}
for all $t\in[0,T]$, which implies
\begin{align}\label{eqn-641}
    &	 \frac{1}{T}\E \left[\int_0^T\|{\u}(s)\|_{\H}^2\d s\right]   \leq \frac{1}{\left(\alpha-\frac{L}{2}\right)}\left[\frac{1}{T} \|{\u_0}\|_{\H}^2 +  \frac{\|\mathbf{F}\|^2_{\mathrm{L}^2(0,T;\V^{\prime})}}{4T\min\left\{\mu,\left(\alpha-\frac{L}{2}\right)\right\}} + \left\{1 + \frac{L}{\alpha-\frac{L}{2}}\right\}K\right].
\end{align}
Next, making use of the Chebyshev inequality along with inequality \eqref{eqn-641}, we reach at
\begin{align}\label{eqn-642}
     \frac{1}{T} \int_0^T (\mathrm{T}^*_t\delta_{\u_0})(\H\setminus \mathcal{B}_R) \d t 
     &=  \frac{1}{T} \int_0^T \mathbb{P}\big(\{\|\u(t)\|_{\H}>R\}\big) \d t
     \nonumber \\ & \leq \frac{1}{TR^2}\E \left[\int_0^T\|{\u}(t)\|_{\H}^2\d t\right]
     \nonumber\\ & \leq \frac{1}{\left(\alpha-\frac{L}{2}\right)R^2}\left[\frac{1}{T} \|{\u_0}\|_{\H}^2 +  \frac{\|\mathbf{F}\|^2_{\mathrm{L}^2(0,T;\V^{\prime})}}{4T\min\left\{\mu,\left(\alpha-\frac{L}{2}\right)\right\}} + \left\{1 + \frac{L}{\alpha-\frac{L}{2}}\right\}K\right].
\end{align}
Hence, from \eqref{eqn-642}, we immediately conclude the proof.
\end{proof}

        Now, we are ready to provide the first main result of this section, that is, the existence of an invariant probability measure associated with the stochastic system \eqref{32}.
        \begin{theorem}\label{EIM}
           Assume that $\mathbf{F}\in \mathrm{L}^{2}(0,T;\V')$ and  Hypothesis \ref{hyp} is satisfied. Then, for all the cases given in Table \ref{Table-3alpha}, there exists an invariant probability measure of the semigroup $\{\mathrm{T}_{t}\}_{t\geq 0}$ defined in \eqref{eqn-semigroup}, that is, there exists a probability measure $\nu$ on $\H$ such that
            \begin{align*}
                \mathrm{T}^*_t \nu =\nu, \ \text{ for all }\ t\in[0,T].
            \end{align*}
        \end{theorem}
        \begin{proof}
            In view of Proposition \ref{prop-bw-feller}, Lemma \ref{lem-suffi-cond} and Theorem \ref{thm-invaeiant-measure}, we complete the proof.
        \end{proof}
		
\subsection{Existence of at most one invariant measure} 
In this subsection, for all the cases given in Table \ref{Table-4}, we discuss that there exists at most one invariant probability measure (uniqueness) of the semigroup $\{\mathrm{T}_{t}\}_{t\geq 0}$ defined in \eqref{eqn-semigroup}.

The following theorem helps us to provide the corresponding exponential stability results for the stochastic system \eqref{32}.

\begin{theorem}\label{exps1}
For $d\in\{2,3\}$, let $\u(\cdot)$ and $\v(\cdot)$ be two strong solutions of the stochastic system \eqref{32} with the initial data $\u_0,\v_0\in\H$, respectively. Then, for all the cases given in Table \ref{Table-4}, we have 
\begin{align}
\E\left[\|\u(t)-\v(t)\|_{\H}^2\right] & \leq \|\u_0-\v_0\|_{\H}^2e^{-(2\alpha-(2\hat{\zeta}+L))t}, && \text{ for } r\in(3,\infty),\label{513} \\
\E\left[\|\u(t)-\v(t)\|_{\H}^2\right]  & \leq \|\u_0-\v_0\|_{\H}^2e^{-(2\alpha-L)t}, && \text{ for } r=3 \text{ with } 2\beta\mu\geq 1,\label{5133}\\
\E\left[\|\u(t)-\v(t)\|_{\H}^2\right]  & \leq \|\u_0-\v_0\|_{\H}^2e^{-\left(2\alpha - \frac{1}{2\mu}-L\right)t}, && \text{ for } r\in(3,\infty) \text{ with } 2\beta\mu\geq 1,\label{51333}
\end{align}
where $\hat{\zeta}$ is defined in \eqref{eqn-zeta}. 
\end{theorem}
\begin{proof}
	Let us define $\w(\cdot)=\u(\cdot)-\v(\cdot)$. Then, $\w(\cdot)$ satisfies the following energy equality: 
	\begin{align}\label{514}
	\|\w(t)\|_{\H}^2&=\|\w_0\|_{\H}^2-2\mu\int_0^t\|\nabla\w(s)\|_{\H}^2\d s-2\alpha\int_0^t\|\w(s)\|_{\H}^2\d s
	\nonumber\\ & \quad -2\beta\int_0^t\langle
	\mathcal{C}(\u(s))-\mathcal{C}(\v(s)),\w(s)\rangle\d s -2\int_0^t\left<\B(\u(s))-\B(\v(s)),\w(s)\right>\d s
	\nonumber\\&\quad +\int_0^{t}\|\widetilde{\boldsymbol{\sigma}}(s)\|_{\mathcal{L}_{\Q}}^2\d s +2\int_0^{t}(\widetilde{\boldsymbol{\sigma}}(s)\d\W(s),\w(s)),
	\end{align}
	where $\widetilde{\boldsymbol{\sigma}}(\cdot)=\boldsymbol{\sigma}(\cdot,\u(\cdot))-\boldsymbol{\sigma}(\cdot,\v(\cdot))$.
    \vskip 0.2 cm
	\noindent\textbf{Case (1):} {$d=2,3$ and $r\in(3,\infty)$.} Firstly,  taking expectation in \eqref{514} and then using Hypothesis \ref{hyp} (H.3), \eqref{5p21} and \eqref{5p2121}, one can easily see that 
	\begin{align}\label{515}
	\E\left[	\|\w(t)\|_{\H}^2\right] + 2\mu \E \left[\int_0^t \|\nabla\w(s)\|_{\H}^2\d s \right]\leq\|\w_0\|_{\H}^2-(2\alpha-(2\hat{\zeta}+L))\int_0^t\E\left[\|\w(s)\|_{\H}^2\right]\d s,
	\end{align}
	where $\hat{\zeta}$ is defined in \eqref{eqn-zeta}. 	Thus, an application of  Gr\"onwall's inequality yields 
	\begin{align}
	\E\left[	\|\w(t)\|_{\H}^2\right]\leq \|\w_0\|_{\H}^2e^{-(2\alpha -(2\hat{\zeta}+L))t},
	\end{align}
	and for $ 2\alpha > 2\hat{\zeta}+L$, we obtain the required result \eqref{513}.

Secondly, as in \cite{YZ}, for $r> 3$, one can estimate $|\langle\B(\u-\v,\u-\v),\v\rangle|$ as (see \cite{YZ} and \cite[Remark 2.7]{Gautam+Mohan_2025})
	\begin{align}\label{2.26}
&	|\langle\B(\u-\v,\u-\v),\v\rangle|
 \leq \mu \|\nabla(\u-\v)\|_{\H}^2+\frac{1}{4\mu } \||\v|^{\frac{r-1}{2}}(\u-\v)\|_{\H}^2 + \frac{1}{4\mu }\|\u-\v\|_{\H}^2.
	\end{align}    
    Taking expectation in \eqref{514} and then using Hypothesis \ref{hyp} (H.3), \eqref{5p2121} and \eqref{2.26}, one can find that for $2\beta\mu\geq 1$
	\begin{align}\label{515}
	\E\left[	\|\w(t)\|_{\H}^2\right]  \leq\|\w_0\|_{\H}^2-\bigg(2\alpha- \frac{1}{2\mu} - L\bigg)\int_0^t\E\left[\|\w(s)\|_{\H}^2\right]\d s,
	\end{align}
    Thus, an application of  Gr\"onwall's inequality yields 
	\begin{align}
	\E\left[	\|\w(t)\|_{\H}^2\right]\leq \|\w_0\|_{\H}^2e^{-\left(2\alpha -\frac{1}{2\mu} - L\right))t},
	\end{align}
	and for $ 2\alpha > \frac{1}{2\mu} + L$,  we obtain the required result \eqref{51333}. 
    
    \vskip 0.2 cm
	\noindent\textbf{Case (2):} {$d=2,3$ and $r=3$ with $2\beta\mu\geq 1$.} Taking expectation in \eqref{514} and then using Hypothesis \ref{hyp} (H.3), \eqref{5p2121} and \eqref{5p25}, we find that for some $0<\theta\leq 1$
	\begin{align}\label{5151}
	\E\left[	\|\w(t)\|_{\H}^2\right] + 2\mu(1-\theta) \E \left[\int_0^t \|\nabla\w(s)\|_{\H}^2\d s \right]\leq\|\w_0\|_{\H}^2-(2\alpha-L)\int_0^t\E\left[\|\w(s)\|_{\H}^2\right]\d s.
	\end{align}
	Thus, an application of Gr\"onwall's inequality yields 
	\begin{align}
	\E\left[	\|\w(t)\|_{\H}^2\right]\leq \|\w_0\|_{\H}^2e^{-(2\alpha - L))t},
	\end{align}
	and for $ 2\alpha > L$,  we obtain the required result \eqref{5133}. Hence the proof is completed. 
\end{proof}

\begin{remark}
    It is worth mentioning here that for Poincar\'e domains (which may be bounded or unbounded) with Poincar\'e constant $\lambda_1$ (that is, $\|\u\|_{\H}^2\leq \frac{1}{\lambda_1}\|\nabla\u\|_{\H}^2$ for $\u\in\H_0^1(\mathcal{O})$), one can obtain 
    \begin{align*}
        \E\left[\|\u(t)-\v(t)\|_{\H}^2\right] & \leq \|\u_0-\v_0\|_{\H}^2e^{-(2\alpha+\mu\lambda_1-(2\hat{\zeta}+L))t}, \;\;\; \text{ for } \;r\in(3,\infty),
    \end{align*}
    which ensures the exponential stability for $ 2\alpha + \mu\lambda_1 > 2\hat{\zeta}+L$.
\end{remark}

We now provide the uniqueness of the invariant probability measures of the semigroup $\{\mathrm{T}_{t}\}_{t\geq 0}$ obtained in Theorem \ref{EIM}. The strategy of the proof follows the same steps as in \cite{KK+MTM-SCBF} (see \cite[Theorem 6.7]{KK+MTM-SCBF}).

\begin{theorem}\label{UEIM}
Let the conditions given in Theorem \ref{exps1} hold true and $\u_0\in\H$ be given. Then, for all the cases given in Table \ref{Table-4}, there is a unique invariant measure $\upnu$ to the stochastic system \eqref{32}. The measure $\upnu$ is ergodic and strongly mixing, that is, 
\begin{align}\label{6.9a}
\lim_{t\to\infty}\mathrm{T}_t\varphi(\u_0)=\int_{\H}\varphi(\v_0)\d\upnu(\v_0), \ \upnu\text{-a.s., for all }\ \u_0\in\H\ \text{ and }\  \varphi\in\C_b(\H).
\end{align} 
\end{theorem}
\begin{proof}
See the proof of \cite[Theorem 6.7]{KK+MTM-SCBF}.
\end{proof}

\begin{remark}
   We note that, in the case of additive white noise with a trace-class operator, one has $L=0$ in Hypothesis \ref{hyp}. Consequently, Theorem \ref{UEIM} applies to all scenarios listed in Table \ref{Table-4} with $L=0$. In addition, from Theorem \ref{EIM}, we recover the result obtained in \cite[Theorem 5.3]{KKMTM-DCDSB}.
\end{remark}

		\medskip\noindent
		{\bf Acknowledgments:}  This work is funded by national funds through the FCT - Fundação para a Ciência e a Tecnologia, I.P., under the scope of the projects {UID/297/2025} and {UID/PRR/297/2025} (Center for Mathematics and Applications - NOVA Math). M. T. Mohan would  like to thank the Department of Science and Technology (DST) Science $\&$ Engineering Research Board (SERB), India for a MATRICS grant (MTR/2021/000066). A substantial part of this work was completed during K. Kinra’s visit to the Indian Institute of Technology Roorkee, India, whose gracious hospitality greatly assisted him in carrying out this research.

	\medskip\noindent	{\bf Data availability:} 	Data sharing not applicable to this article as no datasets were generated or analysed during the current study.

	\medskip\noindent
	\textbf{Declarations}: During the preparation of this work, the authors have not used AI tools.
	
	\medskip\noindent
	\textbf{Author Contributions}: All authors contributed equally.
	
	\medskip\noindent
	\textbf{Conflict of interest:} The author declares no conflict of interest.

	\end{document}